\documentclass[12pt]{amsart}
\usepackage[latin1]{inputenc}
\usepackage{mathrsfs}
\usepackage{mathtools}
\usepackage{amsmath}
\usepackage{amsfonts}
\usepackage{amssymb}
\usepackage{graphicx}
\usepackage{fourier}
\usepackage{dutchcal}
\usepackage[colorlinks=true,citecolor=green,linkcolor=blue]{hyperref}
\usepackage{enumerate}
\usepackage{esint}
\usepackage{bm}
\usepackage{ulem}
\usepackage{xcolor} 
\usepackage{verbatim}        

\usepackage{subcaption}
\usepackage{epstopdf}

\usepackage{scalerel}

\usepackage[margin=1in]{geometry}

\usepackage[english]{babel}
\newtheorem{theorem}{Theorem}[section]
\newtheorem{lemma}[theorem]{Lemma}

\theoremstyle{definition}
\newtheorem{definition}[theorem]{Definition}

\newtheorem{proposition}[theorem]{Proposition}
\theoremstyle{remark}
\newtheorem{remark}[theorem]{Remark}

\newcommand{\inner}[1]{\left\langle#1\right\rangle}
\newcommand{\norm}[1]{\left\lVert#1\right\rVert}
\newcommand{\abs}[1]{\left\lvert#1\right\rvert}
\newcommand{\pa}[1]{\left( #1 \right)}
\newcommand{\rpa}[1]{\left[ #1 \right]}
\newcommand{\br}[1]{\left\lbrace #1\right\rbrace}

\newcommand{\R}{\mathbb{R}}

\newcommand{\N}{\mathbb{N}}

\newcommand{\bs}{\overline{s}}
\newcommand{\bp}{\overline{p}}

\newcommand{\bc}{\overline{c}}

\numberwithin{equation}{section}

\definecolor{bostonuniversityred}{rgb}{0.8, 0.0, 0.0}
 
\definecolor{byzantium}{rgb}{0.44, 0.16, 0.39}

\newcommand{\marcel}[1]{{\color{black}#1}}

\newcommand{\na}{\nabla}
\newcommand{\lam}{\lambda}
\newcommand{\p}{\partial} 

\newcommand{\wt}{\widetilde}
\renewcommand{\O}{\Omega}

\newcommand{\einf}{e_{\infty}}
\newcommand{\sinf}{s_{\infty}}
\newcommand{\cinf}{c_{\infty}}
\newcommand{\pinf}{p_{\infty}}
\newcommand{\hh}{\mathfrak{h}}
\newcommand{\bh}{\mathbf{h}}
\newcommand{\eps}{\varepsilon}
\renewcommand{\epsilon}{\varepsilon}
\renewcommand{\phi}{\varphi}
\newcommand{\LO}[1]{L^{#1}(\Omega)}
\renewcommand{\H}{\mathscr{H}}
\newcommand{\E}{\mathcal{E}}
\newcommand{\M}{\mathscr{M}}
\newcommand{\h}{\mathcal{h}}
\renewcommand{\d}{\mathcal{d}}
\renewcommand{\r}{\mathcal{r}}
\newcommand{\m}{\d_{\mathscr{M}}}
\renewcommand{\bh}{\mathbf{h}}
\newcommand{\clsi}{C_{\mathrm{LSI}}}
\newcommand{\ccpk}{C_{\mathrm{CPK}}}
\newcommand{\su}{\mathcal{S}}

\title[Irreversible enzyme reactions]{Quantitative dynamics of irreversible enzyme reaction-diffusion systems}
\author[M. Braukhoff]{Marcel Braukhoff}
\address{Marcel Braukhoff \hfill\break
	Mathematisches Institut der Heinrich Heine Universit\"at D\"usseldorf, Universit\"atsstra\ss e 1, D-40225 D\"usseldorf, Germany}
\email{marcel.braukhoff@hhu.de}
\author[A. Einav]{Amit Einav}
\address{Amit Einav \hfill\break 
	School of Mathematical Sciences, Durham University, Upper Mountjoy Campus, Stockton Road DH1 3LE Durham, United Kingdom}
\email{amit.einav@durham.ac.uk}
\author[B.Q. Tang]{Bao Quoc Tang}
\address{Bao Quoc Tang \hfill\break
	Institute of Mathematics and Scientific Computing, University of Graz, 
	Heinrichstrasse 36, 8010 Graz, Austria}
\email{quoc.tang@uni-graz.at, baotangquoc@gmail.com}
\thanks{
	We would like to thank J\'an Elia\v{s} for the discussion concerning the case of non-diffusive enzyme and complex molecules.\\	
	The third author has been partially supported by NAWI Graz.}

\begin{document}

\begin{abstract}
	In this work we investigate the convergence to equilibrium for mass action reaction-diffusion systems which model irreversible enzyme reactions. Using the standard entropy method in this situation is not feasible as the irreversibility of the system implies that the concentrations of the substrate and the complex decay to zero. The key idea we utilise in this work to circumvent this issue is to introduce a family of cut-off partial entropy functions which, when combined with the dissipation of a mass like term of the substrate and the complex, yield an explicit exponential convergence to equilibrium. This method is also applicable in the case where the enzyme and complex molecules do not diffuse, corresponding to chemically relevant situation where these molecules are large in size. 
\end{abstract}

\maketitle


\setcounter{tocdepth}{1}
\tableofcontents

\section{Introduction}\label{sec:intro}
The focus of our work will be on the well known irreversible enzyme reaction system
\begin{equation}\label{ER}
	E + S \underset{k_f}{\overset{k_r}{\leftrightarrows}} C \xrightarrow{k_c} E + P
\end{equation}
where $S$ represents the substrate of the system, $E$ the enzymes, $C$ the intermediate complex (which we will refer to as the complex for simplicity) and $P$ the product.
\subsection{The setting of the problem}\label{subsec:setting}
Our study will concern itself with the inhomogeneous setting of \eqref{ER} which manifests itself in the presence of diffusion of (some of) the system's elements. By applying Fick's law of diffusion and the law of mass action we find that the reaction-diffusion system that corresponds to \eqref{ER} is given by
\begin{equation}\label{Sys}
	\begin{cases}
		\p_t e(x,t) - d_e\Delta e(x,t) = -k_fe(x,t)s(x,t) + (k_r+k_c)c(x,t), &x\in\Omega,t>0,\\
		\p_t s(x,t) - d_s\Delta s(x,t) = -k_fe(x,t)s(x,t) + k_rc(x,t), &x\in\Omega, t>0,\\
		\p_t c(x,t) - d_c\Delta c(x,t) = k_fe(x,t)s(x,t) - (k_r+k_c)c(x,t), &x\in\Omega,t>0,\\
		\p_t p(x,t) - d_p\Delta p(x,t) = k_cc(x,t), &x\in\Omega, t>0,\\
		{d_e\p_\nu e(x,t) = d_s\p_\nu s(x,t) = d_c\p_\nu c(x,t) = d_p\p_\nu p(x,t) = 0}, &x\in\p\Omega, t>0,\\
		e(x,0) = e_0(x), \;s(x,0) = s_0(x),\; c(x,0) = c_0(x),\; p(x,0) = p_0(x), & x\in\Omega,
	\end{cases}
\end{equation}
where $e(x,t), s(x,t), c(x,t), p(x,t)$ are the concentrations of $E$, $S$, $C$ and $P$, respectively, at $x\in\Omega$ and $t>0$. The homogeneous Neumann boundary condition indicates that the system is closed, and in order that the equations would make chemical sense we also require that the initial concentrations $e_0(x), s_0(x)$, $c_0(x)$ and $p_0(x)$ are non-negative functions\footnote{{As the Neumann condition is connected to the diffusion of the concentration, we have elected to write it as 
$$d_e\p_\nu e(x,t) = d_s\p_\nu s(x,t) = d_c\p_\nu c(x,t) = d_p\p_\nu p(x,t) = 0$$
to indicate that we do not require it when the diffusion coefficient is zero.}}.  


\medskip
Looking at the system \eqref{Sys}, one notices immediately that the equation that governs the concentration of $P$, $p(x,t)$, is decoupled from the rest of the system and is completely solvable once $c(x,t)$ has been found. Therefore, our main focus in the majority of this work will be on the dynamics of the sub-system of \eqref{Sys} which includes $e$, $s$ and $c$ alone. 

Much like many other reaction-diffusion systems in a bounded domain, one expects that the combination of the chemical reactions and the diffusion will result in a state of equilibrium that is composed of constant concentrations. Using the (formal) conservation laws 
\begin{equation}\label{law1}
	\int_{\O}(e(x,t) + c(x,t))dx = M_0:=\int_{\O} (e_0(x)+c_0(x))dx,\quad \forall t\geq 0,
\end{equation}
\begin{equation}\label{law2}
	\int_{\O}(s(x,t)+c(x,t)+p(x,t))dx =M_1:= \int_{\O}(s_0(x)+c_0(x)+p_0(x))dx,\quad \forall t\geq 0,
\end{equation}
and under the assumption that $\abs{\O}=1$ for simplicity (which can always be achieved by a simple rescaling of the spatial variable) we find that if all elements in the system diffuse, i.e. if $d_e,d_s,d_c$ and $d_p$ are strictly positive, then the equations that determine the \textit{constant} equilibrium concentrations $e_{\infty}, c_{\infty}, s_{\infty}$ and $p_\infty$ are
\begin{equation*}
	\begin{cases}
		-k_fe_{\infty}s_{\infty} + (k_r+k_c)c_{\infty} = 0,\\
		-k_fe_{\infty}s_{\infty} + k_cc_{\infty} = 0,\\
		k_c c_\infty = 0\\
		e_{\infty} + c_{\infty} = M_0,\\
		s_{\infty} + c_{\infty} + p_{\infty} = M_1,
	\end{cases}
\end{equation*}
from which we find that 
$$e_\infty=M_0,\quad c_\infty=0,\quad s_\infty=0, \quad p_\infty=M_1.$$
This equilibrium carries within it the chemical intuition of the process, as was expected: As time increases, the substrates get completely converted into product, the complex is used up, and the enzymes ``gobble up'' whatever left overs remain in the system.

The above, however, is not the true equilibrium when essential parts of the system \textit{do not diffuse}. In this case, we can't a-priori guarantee that an equilibrium state for these non-diffusing concentration, if one exists, will be a constant function. This situation is chemically feasible, for instance when the molecules of the complex and the enzymes are large enough to deter diffusion. In terms of our system \eqref{Sys} this situation corresponds to the case where $d_e=d_c=0$ and $d_s$ and $d_p$ are strictly positive. The lack of diffusion in the complex $c$ is not very problematic, yet the lack of diffusion in the enzymes $e$ complicates matter further. However, in this situation one can find another (formal) conservation law of the form\footnote{In this case we have that $\p_t\pa{e(x,t)+c(x,t)}=0$.}
\begin{equation}\label{law3}
	 e(x,t)+c(x,t)=e_0(x)+c_0(x),
\end{equation}
which assists in balancing the system. With this in mind, we see that if an equilibrium of the form $\einf(x)$, $c_\infty(x)$ and $s_\infty$ exists \footnote{The equilibrium for $e$ and $c$ could be a function of $x$, but $\sinf$ is still assumed to be constant due to the diffusion in $s$.}, then it must satisfy
\begin{equation*}
	\begin{cases}
		-k_fe_{\infty}(x)s_{\infty} + (k_r+k_c)c_{\infty}(x) = 0,\\
		-k_fe_{\infty}(x)s_{\infty} + k_cc_{\infty}(x) = 0,\\
		e_{\infty}(x) + c_{\infty}(x) = e_0(x)+c_0(x),\\
	\end{cases}
\end{equation*}
from which we find that $s_\infty=c_\infty(x)=0$ and $\einf(x)=e_0(x)+c_0(x)$. The fact that $c_\infty(x)=0$ would lead us to expect that $p(x,t)$ also converges to a constant $\pinf$ and due to \eqref{law2} we conclude that the suspected equilibrium in this case is given by
\begin{equation}\label{equilibrium_degenerate}
	\einf(x) = e_0(x) + c_0(x)\, \quad \sinf = \cinf (x)= 0, \quad \pinf = M_1. 
\end{equation}

The main goal of our work is to explicitly and quantitatively explore the rate of convergence to equilibrium of the solutions to \eqref{Sys} in these two cases in the strongest form possible - the $L^\infty$ norm. 

\subsection{Known results}\label{subsec:known_results}
The study of the trend to equilibrium for (bio-)chemical reaction-diffusion systems has  witnessed significant progress in the last decades. The first results in this direction can be attributed to \cite{groger1983asymptotic,groger1992free,glitzky1996free} where the large time behaviour of two dimensional such systems was studied qualitatively. Quantitative results, i.e. explicit convergence rates and constants, have been provided in \cite{desvillettes2006exponential,desvillettes2008entropy} for special systems, and have been extended later in \cite{desvillettes2017trend,mielke2015uniform} to more complicated ones. The most general equilibration results, which are currently feasible, are those found in \cite{mielke2016uniform,fellner2018convergence}. Despite these developments and advances, the quantitative large time behaviour of the enzyme reaction-diffusion system \eqref{Sys} was, to our knowledge, left completely untouched. The main reason, in our opinion, for this omission lies with the \textit{irreversible} nature of \eqref{ER},  and its impact the \textit{entropy method} which is commonly used to investigate quantitative long time behaviour. We shall address this point shortly, as our work will show how one can ``modify'' this method to attain our claimed results.

\subsection{Main results}\label{subsec:main}
As was indicated in the end of \S \ref{subsec:setting}, our work is devoted to the investigation of two cases for the system \ref{Sys}. Our main results can be expressed by the following theorems:
\begin{theorem}\label{thm:main}
	Let $\O\subset \R^n$ be a bounded, open domain with $C^{2+\zeta}$, $\zeta>0$, boundary, $\p\Omega$. Assume that $d_e,d_s,d_c$ and $d_p$ are strictly positive constants and that the initial data  $e_0(x), s_0(x), c_0(x)$ and $p_0(x) $ are all bounded non-negative functions. Then  there exists a unique non-negative bounded classical solution to \eqref{Sys}. Moreover, there exists an explicit $\gamma>0$ such that for any $\eta>0$ there exists an explicit constant $\mathcal{C}_{\eta}$, depending only on geometric properties and the initial data which blows up as $\eta$ goes to zero, such that
	\begin{equation}\label{eq:main_convergence_I}
		\begin{gathered}
	    \norm{c(t)}_{\LO{\infty}} + \norm{s(t)}_{\LO{\infty}}  \leq \mathcal{C}_{\eta}e^{-\frac{2\gamma t}{n\pa{1+\eta}}},\\
		\norm{e(t) - \einf}_{\LO{\infty}}\leq \mathcal{C}_{\eta}e^{-\frac{\gamma t}{n\pa{1+\eta}}}.
		\end{gathered}
	\end{equation}
 In addition for any $\eps>0$ and $\eta>0$ with $n\pa{1+\eta}\geq 4$ there exists an explicit constant $\mathcal{C}_{\eta,\eps}$, depending only on geometric properties and the initial data which blows up as $\eta$ or $\eps$ go to zero, such that
 	\begin{equation}\label{eq:main_convergence_II}
 	\begin{gathered}
 		\norm{p(t)-p_\infty}_{\LO{\infty}} \leq \mathcal{C}_{\eta,\eps}\pa{1+ t^{\frac{4}{n\pa{1+\eta}}\delta_{\scaleto{\frac{2d_p}{C_P}\pa{1-\eps},\gamma}{10pt}\mathstrut}}}e^{-\min\pa{\frac{4d_p}{nC_P\pa{1+\eta}}\pa{1-\eps},\frac{2\gamma}{n\pa{1+\eta}}}t}
 	\end{gathered}
 \end{equation}
where 
$$\delta_{x,y}=\begin{cases}
	1 & x=y \\
	0 & x\not=y
\end{cases}$$ 
 and $C_P$ is the Poincar\'e constant associated to the domain $\O$.
 \end{theorem}


\begin{theorem}\label{thm2}
	Let $\Omega\subset\R^n$ be a bounded, open domain with $C^{2+\zeta}, \zeta>0$ boundary, $\p\Omega$. Assume that $d_s$ and $d_p$ are strictly positive constants while $d_e=d_c=0$. Assume in addition that  the initial data $e_0(x)$, $s_0(x)$, $c_0(x)$ and $p_0(x)$ are all bounded non-negative functions and that there exists some $\beta>0$ such that
	\begin{equation}\label{lower_bounded_e0+c0}
		e_0(x) + c_0(x) \geq \beta \quad\text{a.e. } x\in\Omega.
	\end{equation}
	Then there exists a unique non-negative bounded strong solution to \eqref{Sys}. Moreover, there exists an explicit constant $\gamma>0$ such that for any $\eta>0$ there exists an explicit constant $\mathcal{C}_{\eta}$, depending only on geometric properties and the initial data which blows up as $\eta$ goes to zero, such that
	\begin{equation*}
		\begin{gathered}
			\|s(t)\|_{\LO{\infty}} \leq \mathcal{C}_{\eta}e^{-\frac{2\gamma}{n\pa{1+\eta}}t},\\
			\|e(x,t)-\einf(x)\|_{\LO{\infty}} + \|c(t)\|_{\LO{\infty}}  \leq \mathcal{C}_{\eta}\pa{1+t^{\delta_{\scaleto{k_r+k_c,\frac{2\gamma} {n\pa{1+\eta}}\mathstrut}{10pt}}}}e^{-\min\pa{k_r+k_c,\frac{2\gamma}{n\pa{1+\eta}}}t}.
		\end{gathered}
	\end{equation*}
In addition for any $\eps>0$ and $\eta>0$ with $n\pa{1+\eta}\geq 4$ there exists an explicit constant $\mathcal{C}_{\eta,\eps}$, depending only on geometric properties and the initial data which blows up as $\eta$ or $\eps$ go to zero, such that
\begin{equation}\nonumber
	\begin{gathered}
		\norm{p(t)-p_\infty}_{\LO{\infty}} \leq \mathcal{C}_{\eta,\eps}\pa{1+ t^{\frac{4}{n\pa{1+\eta}}\delta_{\scaleto{\frac{2d_p}{C_P}\pa{1-\eps},\gamma}{10pt}\mathstrut}}}e^{-\min\pa{\frac{4d_p}{nC_P\pa{1+\eta}}\pa{1-\eps},\frac{2\gamma}{n\pa{1+\eta}}}t}
	\end{gathered}
\end{equation}
\end{theorem}
\begin{remark}\label{rem:strong_solution}
	Our notion of \textit{strong solutions} to \eqref{Sys} {is as follows: For any $p\in[1,\infty)$, any component of the solution belongs to $C([0,\infty);L^p(\Omega))$ and is absolutely continuous  on $(0,\infty)$ with respect to $L^p(\Omega)$. Moreover, the time derivatives, and the spatial derivatives up to second order of any concentration, which is diffusing, are in $L^p\pa{(\tau,T);L^p\pa{\O}}$ for any $T>\tau>0$,} and the equations {and boundary conditions are satisfied a.e.\ in $\O\times (0,T)$ and a.e.\ in $\partial \O\times (0,T)$ respectively, for any $T>0$.}  
\end{remark}

\begin{remark}\label{rem:explicit_gamma}
	During our proofs we will be able to provide an explicit form to $\gamma$ in each of the theorems. We will show that one can choose
	\begin{equation}\label{eq:choice_for_gamma}
		\begin{aligned}
			\gamma =\min&\left(\frac{\pa{d_e \clsi -6\pa{\frac{\pa{k_c+k_r}}{M_0}+k_f}\max\pa{\eps_c,\eps_s}} }{\pa{1+\pa{\log\pa{1+\frac{M_1}{\eps_s}}+\frac{k\pa{2k_r+k_c}}{2k_r}}\frac{16\pa{\eps_s+M_1}}{\pa{1-\log(2)}M_0}}},\right.\\
			&\quad\frac{\frac{kk_c}{2}-k_c-k_r-2k_f\eps_s}{\pa{1+k+\log\pa{1+\frac{M_0}{\eps_c}}+\pa{\log\pa{1+\frac{M_1}{\eps_s}}+\frac{k\pa{2k_r+k_c}}{2k_r}}\pa{\frac{2k_r}{k_fM_0}+\frac{16\pa{\eps_s+M_1}}{\pa{1-\log(2)}M_0}}}},\\
			&\quad \left.\frac{ d_c  d_s \clsi}{d_s+d_c\pa{1+\pa{\log\pa{1+\frac{M_1}{\eps_s}}+\frac{k\pa{2k_r+k_c}}{2k_r}}}}\right)
		\end{aligned}
	\end{equation} 
	with 
	\begin{equation}\label{eq:possible_choices_full_diffusion}
		\begin{gathered}
			\eps_c= \frac{\clsi M_0}{12\pa{k_c+k_r+M_0k_f}\max\pa{1, \frac{k_r}{M_0k_f}}},\quad 
			\eps_s=\frac{k_r}{M_0 k_f}\eps_c,\quad 
			k= \frac{4\pa{k_c+k_r+k_f\eps_s}}{k_c}.
		\end{gathered}
	\end{equation}
	for Theorem \ref{thm:main}, where $\clsi$ is the log-Sobolev constant that is associated to the domain $\O$, and 
	\begin{equation}\label{eq:choice_for_lambda}
		\begin{aligned}
			\gamma =&\min\left(\frac{\frac{kk_c}{2}-k_f\eps_s}{\pa{1+k+\log\pa{1+\frac{k_r}{k_f\eps_s}}+\pa{\log\pa{1+\frac{M_1}{\eps_s}}+\frac{k\pa{2k_r+k_c}}{2k_r}}\pa{\frac{2k_r}{k_f\beta}+\frac{16\pa{\eps_s+M_1}}{\pa{1-\log(2)}\beta}}}}\right.\\
			&\qquad\left.\frac{ d_s\clsi}{1+\pa{\log\pa{1+\frac{M_1}{\eps_s}}+\frac{k\pa{2k_r+k_c}}{2k_r}}}\right),
		\end{aligned}
	\end{equation} 
	with 
	\begin{equation}\label{eq:possible_choices_partial_diffusion}
	\begin{gathered}
		\eps_s\in \pa{0,\infty},\quad
		k= \frac{4k_f \eps_s}{k_c}.
	\end{gathered}
\end{equation}
	These choices choices are clearly far from optimal, though explicit. Optimal convergence rate is a subtle issue and remains open in most chemical reaction-diffusion systems. Some discussion about this issue in the case where all diffusion coefficients are strictly positive will be given in \S \ref{subsec:optimal_rate}.
\end{remark}
\begin{remark}\label{remark:convergence_without_lower_bound}
	In our analysis, the lower bound condition \eqref{lower_bounded_e0+c0} is essential to be able to obtain the explicit exponential convergence to equilibrium. This condition is easily satisfied when we require that the initial enzyme concentration $e_0$ is present everywhere in $\Omega$. In the case where \eqref{lower_bounded_e0+c0} does not hold numerical solutions suggest that convergence to equilibrium can still be expected, as is indicated in Fig. \ref{fig1} in \S\ref{subsec:convergence_without_lower_bound}. A rigorous proof, however, remains an open problem.
\end{remark}
A common method one employs to investigate the \textit{quantitative} long time behaviour of chemical reaction-diffusion systems (and many other physically, chemically and biologically relevant equations) is the so-called \textit{entropy method}: By considering the connection between a natural Lypunov functional of the system, the \textit{entropy}, and its dissipation, usually via a functional inequality, one recovers an explicit rate of convergence to equilibrium. This convergence rate can be boosted up to stronger norms in many situations, at least as long as the system has some smoothing effects. 

This method has been extremely successful in dealing with reaction-diffusion systems which model \textit{reversible} reactions, see e.g. \cite{desvillettes2006exponential,mielke2015uniform,desvillettes2017trend,fellner2018convergence} and references therein. A fundamental property of these systems which help facilitate the entropy method is the existence of a \textit{strictly positive} equilibrium, which allows the consideration of natural entropies such as the Boltzmann entropy (to be defined shortly). This property, however, is not necessarily true in most open or irreversible reaction systems, precluding the consideration of the aforementioned entropies and, in our opinion, resulting in a relatively sparse study of the quantitative large time behaviour of such systems. The current work serves, to our knowledge, as the first study in this direction for the well known enzyme reaction \eqref{ER}, and we believe that the method introduced herein will be applicable to many other open and irreversible systems,

Let us delve deeper into the entropies one encounters in the study of chemical reaction-diffusion systems, the issues of a zero equilibrium and how we propose to overcome it in this work.

A natural entropy to consider in many chemical reaction-diffusion systems is generated by the \textit{Boltzmann entropy function}
\begin{equation}\label{eq:boltzmann_entropy_function}
	\hh(x)=x\log x -x +1,\qquad x\geq 0.
\end{equation}
In particular, one uses this entropy function to define a \textit{relative entropy} functional that measures the ``entropic distance'' between a solution to an equation, $f(x)$, and its equilibrium $f_\infty$:
\begin{equation}\label{relative_entropy}
	\mathcal{E}(f|f_\infty) = \int_{\O}h\pa{f(x)|f_\infty}dx
\end{equation}
where
\begin{equation}\label{eq:boltzmann_relative_entropy_function}
	h\pa{x|y}=x\log \pa{\frac{x}{y}} -x +y=y\hh\pa{\frac{x}{y}},\quad \quad x,y>0.
\end{equation}
Both $\mathcal{E}$ and $h$ are naturally not well defined when $f_\infty=0$ which is exactly the equilibrium we have (or suspect) for our substrate and complex concentrations $s$ and $c$. Similar issues occur when the equilibrium is spatially inhomogeneous, i.e. $f_\infty = f_{\infty}(x)$, with 
$$\abs{\br{x\in\O\;|\;f_{\infty}(x)=0}}> 0.$$ 
This situation can indeed occur, as was shown in \cite{jabin2011selection} where one finds that $f_\infty$ can be a sum of Dirac masses. 

\medskip
The first key idea and strategy that will guide us in showing our main results is to \textit{modify} our Boltzmann entropy by defining a new relative entropy function that is ``cut'' when the concentration becomes ``small enough'':
\begin{equation}\label{eq:cut_relative_entropy_function}
	h_\eps\pa{x|\eps}=\begin{cases}
		x\log \pa{\frac{x}{\eps}} -x +\eps&  x\geq \eps,\\
		0 & x<\eps.
	\end{cases}
\end{equation}
With this ``cut-off'' entropy, we will consider a \textit{partial} entropy functional
\begin{equation*}
	\H_{\eps_c,\eps_s}(e,s,c) = \int_{\O}h(e(x)|\einf)dx + \int_{\O}h_{\eps_c}(c(x)|\eps_c)dx + \int_{\O}h_{\eps_s}(s(x)|\eps_s)dx,
\end{equation*}
where $\eps_c, \eps_s$ are to be chosen in a meaningful way. It is not clear if $\H_{\eps_c,\eps_s}$ is decreasing along the evolution of the system. Moreover, as we expect $c$ and $s$ to converge to zero, we expect $\H_{\eps_c,\eps_s}(e,s,c)$ to eventually become dependent only on $e$ and $\einf$ for any fixed $\eps_c$ and $\eps_{s}$ - the smallness of $\H_{\eps_c,\eps_s}(e,s,c)$ can only give us information on the convergence of $e$ to $\einf$.

This is the point where we introduce our second key idea: Combine this partial entropy functional with a sum of masses from the substrate and complex which decreases and, together with the partial entropy, will drive these concentration to zero\footnote{One can think of this as a hypocoercivity idea.}. This sum of masses will be of the form
\begin{equation*}
	\M(s,c) = \int_{\O}(c(x) + \eta s(x))dx
\end{equation*}
with a suitable choice of $\eta>0$. The drive of the concentration towards zero by $\M(s,c)$ is expressed by the fact that we will find an explicit \textit{mass density}, $\m(x)$, such that
\begin{equation}\label{mass_diss}
	\frac{d}{dt}\M(s,c)= -\int_{\O}\m(x)dx \leq 0.
\end{equation}
With $\H_{\eps_c,\eps_s}(e,s,c)$ and $\M(s,c)$ at hand we will define our \textit{total} entropy functional to be
\begin{equation*}
	\E_{\eps_c,\eps_s,k}(e,s,c)= \H_{\eps_c,\eps_s}(e,s,c) + k\M(s,c)
\end{equation*}
for an appropriately chosen constant $k>0$. 

Showing the exponential convergence to equilibrium of $\E_{\eps_c,\eps_s,k}$ will take lead from ideas that govern the entropy method. Indeed, while $\H_{\eps_c,\eps_s}(e,s,c)$ might not be decreasing along the flow of \eqref{Sys}, we will show that we could find a \textit{partial entropic density}, $\d_{\eps_c,\eps_{s}}(x)$, such that
\begin{equation}\label{partial_entropy_diss}
	\frac{d}{dt}\E_{\eps_c,\eps_s,k}(e,s,c) \lesssim\begin{cases}
		- \int_{\O}\big(\d_{\eps_c,\eps_s}(x) +h(e(x)|\overline{e}) + \m(x)\big)dx & d_e,d_s,d_c>0,\\
		- \int_{\O}\big(\d_{\eps_c,\eps_s}(x)  + \m(x)\big)dx & d_e=d_c=0,
	\end{cases}
\end{equation}
where $\overline{e} = \int_{\O}e(x)dx$, and where the appropriate constant may depend on $\eps_c$, $\eps_s$ and $k$. \eqref{partial_entropy_diss}, together with the structure of $\d_{\eps_c,\eps_s}$ and $\m$, will imply that for suitable choices for  $\eps_c$, $\eps_s$ and $k$ we will get that
\begin{equation*}
	\begin{cases}
		- \int_{\O}\big(\d_{\eps_c,\eps_s}(x) +h(e(x)|\overline{e}) + \m(x)\big)dx & d_e,d_s,d_c>0,\\
		- \int_{\O}\big(\d_{\eps_c,\eps_s}(x)  + \m(x)\big)dx & d_e=d_c=0,
	\end{cases} \lesssim -\E_{\eps_c,\eps_s,k}(e,s,c),
\end{equation*}
from which the exponential decay of $\E_{\eps_c,\eps_s,k}$ follows. It is worth mentioning that the above inequality is a \textit{purely functional inequality} that may be of interest in other related problems. 

With the decay of the entropy at hand, the boundedness of the solution to \eqref{Sys} will imply the desired $L^\infty$ convergence in the regularising case of full diffusion fairly easily. As could be expected, the case where $d_e=d_c=0$ is more complicated as it precludes regularisation for these concentration. However, as $s$ still enjoys regularisation and its convergence can be boosted to an $L^\infty$ one, the ODE nature of the equations for $e$ and $c$ together with the behaviour of $s$ will give us the desired $L^\infty$ estimation. As predicted, the explicit convergence of $p$ to its equilibrium will follow immediately from our conservation law \eqref{law2} and the long time behaviour of $c$.

We would like to mention that the idea of using a ``truncated entropy'' functional has been used before, see for instance \cite{GV10, BRZ2020}, yet to our knowledge this is the first time we it has been used to attain \textit{quantiative} convergence rates to equilibrium. 

\subsection{The structure of the work}\label{subsec:structure} in \S\ref{sec:entropy} we will define our entropy and will employ the ideas of the entropy method to achieve an exponential convergence to equilibrium in both our cases under the assumption of the existence of strong solutions. In \S \ref{sec:unfiorm} we shall use the convergence of this entropy and regularising properties of our system to conclude theorems \ref{thm:main} and \ref{thm2}. We will conclude with some final thoughts in \S\ref{sec:final} which will be followed by an Appendix where we will consider a few technical lemmas and theorems that have been used along our work. 

\section{The Modified Entropy}\label{sec:entropy}
The goal of this section is to define our entropy, which will comprise of ``cut off'' Boltzmann entropy and a decreasing mass-like term, and to explore its evolution. We remind the reader that we assume throughout the presented work that $\abs{\O}=1$. Simple modification can be made to accommodate the general case.  
\begin{definition}
	For a given non-negative functions $e(x),c(x)$ and $s(x)$, strictly positive coefficients $k_r,k_f$ and $k_c$, strictly positive constants $\eps_{s}$ and $k$, and strictly positive functions $\eps_c(x)$ and $\einf(x)$,  we define the partial mass function, $\M$, the Boltzmann entropy-like function, $\H_{\eps_c,\eps_s}$, and the entropy functional $\E_{\eps_c,\eps_s,k}$ as
	\begin{equation}\label{eq: mass_like}
			\M(c,s): = \int_{\O}\pa{c(x) + \frac {1}{2}\pa{\frac{2k_r+k_c}{k_r}}s(x)}dx,
	\end{equation}
	 \begin{equation}\label{eq:boltzmann_like_entropy}
	 	\begin{split}
	 		\H_{\eps_c,\eps_s}\pa{e,c,s}:=& \int_{\O} h(e(x)|e_\infty(x))dx +\int_{\O} h_{\eps_c(x)}(c(x)|\eps_c(x))dx +\int_{\O} h_{\eps_s}(s(x)|\eps_s)dx,
	 	\end{split}
	 \end{equation}
 and
 \begin{equation}\label{eq:def_of_entropy}
 	\E_{\eps_c,\eps_s,k}\pa{e,c,s}:=\H_{\eps_c,\eps_s}\pa{e,c,s}+k\M(c,s).
 \end{equation}
\end{definition}
The subscripts of $\H_{\eps_c,\eps_s}$, $\eps_c$ and $\eps_s$, correspond to the choice of entropic cut off we'll perform. Their choices will be motivated from the \textit{reversible} chemical reaction that the substrate and intermediate compound undergo. The additional parameter for the entropy $\E_{\eps_c,\eps_s,k}$, $k$, corresponds to the ratio of the mass like element that we need to add to drive the convergence  to equilibrium once we've reached our entropic threshold $\eps_c$ and $\eps_s$. 

The main theorem we will show in this section is the following:
\begin{theorem}\label{thm:entropic_convergence}
	Let $e(x,t),c(x,t)$ and $s(x,t)$ be non-negative bounded strong solutions to the irreversible enzyme system \eqref{Sys} with initial data $e_0(x),c_0(x)$ and $s_0(x)$. Then
	\begin{enumerate}[(i)]
		\item \label{item:full_diffusion} if $d_e,d_s$ and $d_c$ are strictly positive constants, and $\eps_c$ and $\eps_s$ are constants such that
		\begin{equation}\label{eq:eps_conection}
			\begin{gathered}
				\frac{\eps_c}{M_0\eps_s}=\frac{k_f}{k_r},
			\end{gathered}
		\end{equation} 
	then for any $\gamma$ such that
		\begin{equation}\label{eq:conditions_gamma_full_diffusion}
		\begin{aligned}
			\gamma \leq &\min\left(\frac{\pa{d_e\clsi-6\pa{\frac{\pa{k_c+k_r}}{M_0}+k_f}\max\pa{\eps_c,\eps_s}} }{\pa{1+\pa{\log\pa{1+\frac{M_1}{\eps_s}}+\frac{k\pa{2k_r+k_c}}{2k_r}}\frac{16\pa{\eps_s+M_1}}{\pa{1-\log(2)}M_0}}},\right.\\
			&\qquad \frac{\frac{kk_c}{2}-k_c-k_r-2k_f\eps_s}{\pa{1+k+\log\pa{1+\frac{M_0}{\eps_s}}+\pa{\log\pa{1+\frac{M_1}{\eps_c}}+\frac{k\pa{2k_r+k_c}}{2k_r}}\pa{\frac{2k_r}{k_fM_0}+\frac{16\pa{\eps_s+M_1}}{\pa{1-\log(2)}M_0}}}},\\
			&\qquad\left.\frac{ d_c  d_s\clsi}{d_s+d_c\pa{1+\log\pa{1+\frac{M_1}{\eps_s}}+\frac{k\pa{2k_r+k_c}}{2k_r}}}\right),
		\end{aligned}
	\end{equation} 
	we have that
	\begin{equation}\label{eq:entropic_inequality}
		\frac{d}{dt}\E_{\eps_c,\eps_s,k}(e(t),c(t),s(t)) + \gamma \E_{\eps_c,\eps_s,k}(e(t),c(t),s(t)) \leq 0,
	\end{equation}
	and consequently 
	\begin{equation}\label{eq:entropic_convergence}
		\E_{\eps_c,\eps_s,k}(e(t),c(t),s(t))  \leq \E_{\eps_c,\eps_s,k}(e_0,c_0,s_0)e^{-\gamma t}. 
	\end{equation}
	\item \label{item:partial_diffusion} if $d_e=d_c=0$ and $d_s>0$, and if there exists $\beta>0$ such that 
	\begin{equation}\label{eq:lower_bound_for_einf}
		\einf(x)=e_0(x)+c_0(x)\geq \beta, \quad\text{a.e. } x\in\Omega,
	\end{equation}
	then for any strictly positive functions $\eps_c(x)$ and constant $\eps_s$ such that
\begin{equation}\label{eq:eps_conection_partial_diffusion}
	\begin{gathered}
		\frac{\eps_c(x)}{\einf(x)}=\frac{k_f}{k_r}\eps_s,
	\end{gathered}
\end{equation} 
	and any $\gamma$ such that
	\begin{equation}\label{eq:conditions_gamma_partial_diffusion}
		\begin{aligned}
			\gamma \leq &\min\left(\frac{\frac{kk_c}{2}-k_f\eps_s}{\pa{1+k+\log\pa{1+\frac{k_r}{k_f\eps_s}}+\pa{\log\pa{1+\frac{M_1}{\eps_s}}+\frac{k\pa{2k_r+k_c}}{2k_r}}\pa{\frac{2k_r}{k_f\beta}+\frac{16\pa{\eps_s+M_1}}{\pa{1-\log(2)}\beta}}}}\right.\\
			&\qquad\left.\frac{ d_s\clsi}{1+\log\pa{1+\frac{M_1}{\eps_s}}+\frac{k\pa{2k_r+k_c}}{2k_r}}\right),
		\end{aligned}
	\end{equation} 
	we have that \eqref{eq:entropic_inequality}  and \eqref{eq:entropic_convergence} are valid.
	\end{enumerate}
\end{theorem}
\begin{remark}\label{rem:possible_choices}
	Possible choices for $\eps_c$, $\eps_s$ and $k$ in \eqref{item:full_diffusion} that give an explicit positive $\gamma$ that equals the expression in the right hand side of \eqref{eq:conditions_gamma_full_diffusion} are
	\begin{equation}\nonumber
		\begin{gathered}
			\eps_s= \frac{d_e\clsi M_0}{12\pa{k_c+k_r+M_0k_f}\max\pa{1, \frac{M_0k_f}{k_r}}},\\
			\eps_c= \frac{M_0 k_f}{k_r}\eps_s\\
			k= \frac{4\pa{k_c+k_r+2k_f\eps_s}}{k_c}.
		\end{gathered}
	\end{equation}
	Similarly, for \eqref{item:partial_diffusion} one can choose
		\begin{equation}\nonumber
		\begin{gathered}
			\eps_s\in \pa{0,\infty},\\
			\eps_c(x)= \frac{ k_f \eps_s}{k_r}\einf(x)=\frac{ k_f \eps_s}{k_r}\pa{e_0(x)+c_0(x)} \geq \frac{ k_f \eps_s}{k_r}\beta >0 \\
			k= \frac{4k_f \eps_s}{k_c}.
		\end{gathered}
	\end{equation}
\end{remark}
\begin{remark}\label{rem:eps_condition_equivalent}
	Condition \eqref{eq:eps_conection} gives us one ingredient in how one chooses the thresholds for our substrate and intermediate compound. Note that it is strongly related to the reversible chemical reaction in \eqref{ER}, as was eluded before. \\
	Looking at condition \eqref{eq:eps_conection_partial_diffusion}, on the other hand, one might wonder why $\eps_s$ is required to remain a constant while $\eps_c(x)$ is allowed to be changed to a function. The fact that some change is required is evident from the fact that our equilibrium state for $e$ is no longer constant. However, when one differentiates the entropy $\E_{\eps_c,\eps_s,k}$ with respect to time in the case where $d_e=d_c=0$ the only term that brings out a Laplacian, and as such requires additional integration by parts, is that which is induced from $h_{\eps_s}\pa{s(x)|\eps_{s}}$. Keeping $\eps_{s}$ constant is a vital simplification to the estimation of the evolution of this term (as will be shown shortly). 
\end{remark}
In order to prove Theorem \ref{thm:entropic_convergence} we will explore the dissipation properties of $\M$ and $\H_{\eps_c,\eps_s}$, starting with the simple mass-like term.
\begin{lemma}\label{lem:decay_of_mass}
	Let $c(x,t)$ and $s(x,t)$ be non-negative strong solutions to the irreversible enzyme system \eqref{Sys}. Then 
	\begin{equation}\label{eq:decay_of_mass}
		\frac{d}{dt}\M(c(t),s(t))=-\frac{k_fk_c}{2k_r}\int_{\O}e(x,t)s(x,t)dx - \frac{k_c}{2}\int_{\O}c(x,t)dx 
	\end{equation}
\end{lemma}
\begin{proof}
	As $c(x,t)$ and $s(x,t)$ are strong solutions to our system of equations we find that by integration by parts\footnote{More information about it is given in Appendix \ref{secapp:additional_proofs}.} one has that 
	$$\frac{d}{dt}\M(c(t),s(t))=\int_{\O}\Big(d_c\Delta c(x,t)+k_fe(x,t)s(x,t)-\pa{k_r+k_c}c(x,t)$$
	$$+\frac{1}{2}\pa{\frac{2k_r+k_c}{k_r}}\pa{d_s\Delta s(x,t)-k_fe(x,t)s(x,t)+k_rc(x,t)}\Big)dx$$
	$$=-\frac{k_fk_c}{2k_r}\int_{\O}e(x,t)s(x,t)dx - \frac{k_c}{2}\int_{\O}c(x,t)dx, $$
	giving us the desired result.
\end{proof}
The investigation of the Boltzmann-like entropy is a bit more involved. To simplify the computations that will follow we define a few new functions that relate to the generators of $\H_{\eps_c,\eps_s}$ and its dissipation, as well as the generator of the dissipation of $\M$. To be able to do so we introduce another entropically relevant function, which makes its appearance in the entropic dissipation term:
\begin{equation}\label{eq:dissipation_entropy}
	\h(x)=x-\log x -1,\qquad x>0.
\end{equation}
Note that much like $\hh$, $\h$ is non-negative and $\h(x)=0$ if and only if $x=1$. We will also need a geometric constant for our definitions, $C_{\mathrm{LSI}}$, which is the log-Sobolev constant of the domain $\O$, i.e. the constant for which we have that
\begin{equation}\label{eq:LSI}
	\int_{\O}\frac{\abs{\na f(x)}^2}{f(x)}dx \geq C_{\mathrm{LSI}}\int_{\O}h\pa{f(x)|\overline{f}}dx,
\end{equation}
for any non-negative $f\in H^1\pa{\O}$ where $\overline{f}=\int_{\O}f(x)dx$. For more information on the above inequality we refer the reader to \cite{desvillettes2013exponential}. 
\begin{definition}
	For a given non-negative functions $e(x),c(x)$ and $s(x)$, strictly positive constants $k_r,k_f$ and $\eps_s$, and strictly positive functions $\eps_c(x)$ and $\einf(x)$
we define the \textit{mass density}
	\begin{equation}\label{eq: mass_density}
		\m(x): = \frac{k_f}{k_r}e(x)s(x) + c(x),
	\end{equation}
	and the \textit{partial entropy production density}
	\begin{equation}\label{eq:entropy_dissipation_density}
		\d_{\eps_c,\eps_s}(x):=\begin{dcases}
			\begin{array}{l}
			 k_f c(x)\bh\pa{\frac{e(x)s(x)}{c(x)}\Big|\frac{\einf(x)\eps_s}{\eps_c(x)}} +k_ce(x)h\pa{\frac{c(x)}{e(x)}\Big|\frac{\eps_c(x)}{\einf(x)}} \\ 
			 \qquad\qquad +C_{\mathrm{LSI}}d_sh\pa{s(x)|\overline{s_{\eps_s}}} +C_{\mathrm{LSI}}d_ch\pa{c(x)|\overline{c_{\eps_c(x)}}}
			\end{array} & x\in \O_1\\
			 \begin{array}{l}
			 	k_rc(x)\h\pa{\frac{e(x)s(x)}{\einf(x) \eps_s}}+k_cc(x)\h\pa{\frac{e(x)}{\einf(x)}}\\
			 	\qquad \qquad+k_fh\pa{e(x)s(x)|\einf(x) \eps_s}+C_{\mathrm{LSI}}d_sh\pa{s(x)|\overline{s_{\eps_s}}}
			 \end{array} & x\in \O_2,\\
		 	\begin{array}{l}
		 		 \pa{k_r+k_c}e(x)h\pa{\frac{c(x)}{e(x)}\Big|\frac{\eps_c(x)}{\einf(x)}}+k_fc(x)s(x)h\pa{\frac{e(x)}{c(x)}\Big|\frac{\einf(x)}{\eps_c(x)}} \\
		 		 \qquad\qquad +C_{\mathrm{LSI}}d_ch\pa{c(x)|\overline{c_{\eps_c(x)}}}
		 	\end{array} & x\in\O_3\\			
	 		k_fs(x)h\pa{e(x)|\einf(x)}+\pa{k_r+k_c}c(x)\h\pa{\frac{e(x)}{\einf(x)}} & x\in\O_4
			\end{dcases},
	\end{equation}
 where 
\begin{equation}\label{eq:def_of_domains}
	\begin{gathered}
		\O_1=\br{x\in\O\;|\; c(x)\geq  \eps_c(x),\;s(x)\geq  \eps_s},\quad \O_2=\br{x\in\O\;|\; c(x)< \eps_c(x),\;s(x)\geq  \eps_s},\\
		\O_3=\br{x\in\O\;|\; c(x)\geq  \eps_c(x),\;s(x)< \eps_s},\quad \O_4=\br{x\in\O\;|\; c(x)<  \eps_c(x),\;s(x)< \eps_s},
	\end{gathered}
\end{equation}
\begin{equation}\label{eq:def_of_overline_f_eps}
	\overline{f_\eps} = \int_{\O}\max\pa{f(x),\eps}dx,
\end{equation}
and $\bh(x|y)=(x-y)\log\pa{\frac{x}{y}}$ for $x,y\geq 0$ with the value of $\infty$ when $y=0$. 
\end{definition}

\begin{remark}\label{rem:explanation_on_partial_entropy_production_density}
	The appearance of the function $\d_{\eps_c,\eps_s}$ and its choice of name will become apparent when we will start differentiating $\H_{\eps_c,\eps_s}$. We would like to emphasise that its form is not surprising when considering the domain $\O_1,\O_2,\O_3$ and $\O_4$. Indeed, intuitively speaking in any domain where $c$ or $s$ are bigger than the threshold $\eps_c$ or $\eps_s$ respectively we find the relative entropy terms that push us towards the threshold, while if $c$ or $s$ are too small these terms mostly are replaced by linear terms (in $c$ or $s$) that are very small. 
\end{remark}
With this auxiliary functions at hand we can state our first entropy inequality:
\begin{theorem}\label{thm:first_entropy_inequality}
	Let $e(x,t),c(x,t)$ and $s(x,t)$ be non-negative bounded strong solutions to the irreversible enzyme system \eqref{Sys}.Then
	\begin{itemize}
		\item \label{item:entropic_estimation_full_diffusion} if all diffusion coefficients are strictly positive, then assuming that \eqref{eq:eps_conection} is satisfied we have that
			\begin{equation}\label{eq:first_entropy_inequality_full_diffusion}
			\begin{gathered}
				\frac{d}{dt}\E_{\eps_c,\eps_s,k}(t)  +\int_{\O}\Big( \d_{\eps_c,\eps_s}(x,t)+ \pa{\frac{kk_c}{2}-k_c-k_r-2k_f\eps_s}\m(x,t)\\
				+\pa{d_e\clsi-6\pa{\frac{\pa{k_c+k_r}}{M_0}+k_f}\max\pa{\eps_c,\eps_s}}  h\pa{e(x,t)|\overline{e(t)}}\Big)dx \leq 0,
			\end{gathered}
		\end{equation}
		where $\overline{f}=\int_{\O}f(x)dx$.
		\item if $d_e=d_c=0$ and $\eps_s$, $\eps_c(x)$ and $\einf(x)$ satisfy \eqref{eq:eps_conection_partial_diffusion} we have that
		\begin{equation}\label{eq:first_entropy_inequality_partial_diffusion}
		\begin{gathered}
		\frac{d}{dt}\E_{\eps_c,\eps_s,k}(t)  +\int_{\O}\d_{\eps_c,\eps_s}(x,t)dx 
	+\pa{\frac{k_ck}{2}-k_f\eps_s}\int_{\O}\m(x,t)dx \leq 0. 
		\end{gathered}
	\end{equation}
	\end{itemize}
\end{theorem} 
\begin{remark}\label{rem:einf_in_full_diffusion}
	When all diffusion coefficients are strictly positive, we find that the suspected equilibrium of $e$ is a constant $\einf$ that satisfies 
	$$\einf=M_0=\int_{\O}\pa{e_0(x)+c_0(x)}dx=\int_{\O}\pa{e(x,t)+c(x,t)}dx,$$
	We see that in this case \eqref{eq:eps_conection} is equivalent to
	\begin{equation}\nonumber
		\begin{gathered}
			\frac{\eps_c}{\einf\eps_s}=\frac{k_f}{k_r},
		\end{gathered}
	\end{equation} 
	which we will use in our proof.
\end{remark}
One technical lemma we require to prove Theorem \ref{thm:first_entropy_inequality} is the following:
\begin{lemma}\label{lem:truncated_log_sobolev}
	Let $\O$ be a bounded domain with $C^1$ boundary and let $f\in  H^2\pa{\Omega}$ be a non-negative function such that $\na f(x)\cdot \mathcal{n}(x)=0$, where $\mathcal{n}(x)$ is the outer normal to $\O$ at the point $x\in \partial \O$\footnote{The normal gradient is well defined due to a standard Trace theorem. More details are given in Remark \ref{rem:outward_gradient_is_defined} in Appendix \ref{secapp:additional_proofs}}. Then for any $\eps>0$ 
	  \begin{equation}\label{eq:truncated_log_sobolev}
	  	-\int_{\O \cap \br{x\;|\; f(x)\geq  \eps}}\log\pa{\frac{f(x)}{\eps}}\Delta f(x)dx \geq C_{\mathrm{LSI}} \int_{\O \cap \br{x\;|\; f(x)\geq  \eps}}h\pa{f(x)|\overline{f_{\eps}}}dx,
	  \end{equation}
 where $\overline{f_{\eps}}$ was defined in \eqref{eq:def_of_overline_f_eps} and $C_{\mathrm{LSI}}$ is the log-Sobolev constant of the domain $\O$. 
\end{lemma}
\begin{proof}
	We start by noticing that that 
	$$\log\pa{\frac{f(x)}{\eps}}\chi_{\br{z\;|\;f(z)\geq  \eps}}(x)=\max\pa{\log\pa{\frac{f(x)}{\eps}},0}.$$
	To continue we will assume that there exists $\eta>0$ for which $f(x)\geq \eta$ in $\O$. Then
	$$\abs{\log\pa{\frac{f(x)}{\eps}}} \leq \abs{\log \pa{\frac{\eta}{\eps}}}+\frac{\abs{f(x)}}{\eps}\in L^2\pa{\O},$$
	and $\log\pa{\frac{f(x)}{\eps}}$ has a weak derivative\footnote{This is immediate from the fact that $\log(x)$ is $C^\infty$ on $[\mu,\infty)$ and has bounded derivatives on this interval. See, for instance, \cite{LiebLoss},} which satisfies
	$$\abs{\na\pa{\log\pa{\frac{f(x)}{\eps}}}}= \frac{\abs{\na f(x)}}{f(x)} \leq \frac{\abs{\na f(x)}}{\eta}\in L^2\pa{\O}.$$
	Thus $\log\pa{\frac{f(x)}{\eps}}\in H^1\pa{\O}$ and, as shown in Appendix \ref{secapp:additional_proofs}, we have that
		\begin{equation}\nonumber
			\begin{gathered}
						-\int_{\O \cap \br{x\;|\; f(x)\geq  \eps}}\log\pa{\frac{f(x)}{\eps}}\Delta f(x)dx
						= \int_{\O\cap \br{x\;|\; f(x)>  \eps}}\nabla\pa{\log\pa{\frac{f(x)}{\eps}}}  \cdot \na f(x) dx\\
						=\int_{\O\cap \br{x\;|\; f(x)>  \eps}}\frac{\abs{\nabla f(x)}^2}{f(x)}dx.
			\end{gathered}
	\end{equation}
	Denoting by $f_\eps(x)=\max\pa{f(x),\eps}$ we find that since $f\in H^1\pa{\O}$, so is $f_\eps$ and 
	\begin{equation}\nonumber
		\nabla f_\eps(x) 
		=\begin{cases}
			\na f(x)	& f(x)>\eps  \\  0 & f(x)\leq \eps
		\end{cases}.
	\end{equation} 
	(this is also mentioned in Appendix \ref{secapp:additional_proofs} and can be found in \cite{LiebLoss}). As such, the log-Sobolev inequality \eqref{eq:LSI} and the above imply that
	\begin{equation}\label{eq:truncated_LSI_I}
		\begin{gathered}
			-\int_{\O\cap \br{x\;|\; f(x)\geq  \eps}}\log\pa{\frac{f(x)}{\eps}}\Delta f(x)dx = \int_{\O}\frac{\abs{\nabla f_\eps(x)}^2}{f_\eps(x)}dx 	\\
			\geq C_{\mathrm{LSI}}\int_{\O}h\pa{f_\eps(x) |\overline{f_\eps}}dx \geq C_{\mathrm{LSI}}\int_{\O \cap \br{x\;|\; f(x)\geq  \eps}}h\pa{f_\eps(x) |\overline{f_\eps}}\\
			=C_{\mathrm{LSI}}\int_{\O \cap \br{x\;|\; f(x)\geq  \eps}}h\pa{f(x) |\overline{f_\eps}},			
		\end{gathered}
	\end{equation}
	which is the desired result.
	
	We now turn our attention to the case where $f$ is only non-negative on $\O$. \\
	For any $n\in\N$ we define
	$$f_n(x)=f(x)+\frac{1}{n},$$
	and notice that
	$$-\int_{\O\cap \br{x\;|\; f_n(x)\geq  \eps}}\log\pa{\frac{f_n(x)}{\eps}}\Delta f_n(x)dx=-\int_{\O\cap \br{x\;|\; f_n(x)\geq  \eps}}\log\pa{\frac{f_n(x)}{\eps}}\Delta f(x)dx.$$
	Since
	$$\abs{\log\pa{\frac{f_n(x)}{\eps}}\chi_{\br{z\;|\;f_n(z)\geq \eps}}(x)\Delta f(x)}\leq \frac{\abs{f_n(x)}}{\eps}\abs{\Delta f(x)}\leq \frac{\abs{f(x)}+1}{\eps}\abs{\Delta f(x)} \in L^1\pa{\O} $$
	we conclude from the Dominated Convergence Theory that\footnote{Here we have used the fact that $f_n(x) \geq f(x)$ by definition. Thus, if $f(x)\geq \eps$ then so are $f_n(x)$ for all $n$, while if $f(x)<\eps$ then we know that for $n$ large enough $f_n(x)$ satisfies the same. This shows that
	$$\chi_{\br{z\;|\;f_n(z) \geq \eps}}(x)\underset{n\to\infty}{\longrightarrow}\chi_{\br{z\;|\;f(z) \geq \eps}}(x)$$ pointwise.}
	\begin{equation}\label{eq:truncated_first_limit}
		\begin{split}
			\lim_{n\to\infty}&\pa{-\int_{\O\cap \br{x\;|\; f_n(x)\geq  \eps}}\log\pa{\frac{f_n(x)}{\eps}}\Delta f_n(x)dx}=\\	
			\lim_{n\to\infty}&\pa{-\int_{\O\cap \br{x\;|\; f_n(x)\geq  \eps}}\log\pa{\frac{f_n(x)}{\eps}}\Delta f(x)dx}\\\\
			&=-\int_{\O\cap \br{x\;|\; f(x)\geq  \eps}}\log\pa{\frac{f(x)}{\eps}}\Delta f(x)dx.
		\end{split}
	\end{equation}
	On the other hand, since $h(x|y)$ is a non-negative function and 
	$$h\pa{f_n(x)|\overline{\pa{f_n}_{\eps}}}\chi_{\br{z\;|\;f_n(z) \geq \eps}}(x)\underset{n\to\infty}{\longrightarrow}h\pa{f(x)|\overline{f_\eps}}\chi_{\br{z\;|\;f(z) \geq \eps}}(x)$$
	pointwise\footnote{Here we have also used the fact that $\overline{\pa{f_n}_{\eps}}\underset{n\to\infty}{\longrightarrow}\overline{f_\eps}$ according to the Dominated Convergence Theorem.}, 
	we can use Fatou's Lemma to conclude that
	$$C_{\mathrm{LSI}}\int_{\O\cap \br{x\;|\; f(x) \geq \eps}}h\pa{f(x)|\overline{f_\eps}} dx\leq \liminf_{n\to\infty}C_{\mathrm{LSI}}\int_{\O\cap \br{x\;|\; f_n(x) \geq \eps}}h\pa{f_n(x)|\overline{\pa{f_n}_\eps}}dx$$
	$$\leq \liminf_{n\to\infty}\pa{-\int_{\O\cap \br{x\;|\; f_n(x)\geq  \eps}}\log\pa{\frac{f_n(x)}{\eps}}\Delta f_n(x)dx}=-\int_{\O\cap \br{x\;|\; f(x)\geq  \eps}}\log\pa{\frac{f(x)}{\eps}}\Delta f(x)dx$$
	where we have used \eqref{eq:truncated_log_sobolev} and \eqref{eq:truncated_first_limit}. The proof is thus complete.
\end{proof}

\begin{proof}[Proof of Theorem \ref{thm:first_entropy_inequality}]
	We shall use the abusive notation $\E_{\eps_c,\eps_s,k}(t)$ for $\E_{\eps_c,\eps_s,k}\pa{e(t),c(t),s(t)}$ in this proof, as well as drop the $x$ variable from $\einf(x)$ and $\eps_c(x)$ for most of our estimations besides those where differences between the full diffusive and partial diffusive cases arise. \\
	From the definition of $\E_{\eps_c,\eps_s,k}$, Lemma \ref{lem:decay_of_mass}, and the fact that 
	$$\frac{d}{dx}h_{\eps}(x|\eps)=\begin{cases}
		\log \pa{\frac{x}{\eps}} &  x\geq \eps.\\
		0 & x<\eps
	\end{cases}$$
we find that
$$\frac{d}{dt}\E_{\eps_c,\eps_s,k}(t)=\int_{\O}\log\pa{\frac{e(x,t)}{\einf}}\partial_te(x,t)dx+\int_{\O\cap \br{x\;|\;c(x,t)\geq \eps_c}}\log\pa{\frac{c(x,t)}{\eps_c}}\partial_tc(x,t)dx$$
$$+\int_{\O\cap \br{x\;|\;s(x,t)\geq \eps_s}}\log\pa{\frac{s(x,t)}{\eps_s}}\partial_ts(x,t)dx+k\frac{d}{dt}\M(c(t),s(t))$$
$$= \int_{\O}\pa{d_e\Delta e(x,t)-k_fe(x,t)s(x,t)+\pa{k_r+k_c}c(x,t)}\log\pa{\frac{e(x,t)}{\einf}}dx $$
$$+\int_{\O\cap \br{x\;|\;c(x,t)\geq \eps_c}}\pa{d_c\Delta c(x,t)+k_fe(x,t)s(x,t)-\pa{k_r+k_c}c(x,t)}\log\pa{\frac{c(x,t)}{\eps_c}}dx$$
$$+\int_{\O\cap \br{x\;|\;s(x,t)\geq \eps_s}}\pa{d_s\Delta s(x,t)-k_fe(x,t)s(x,t)+k_r c(x,t)}\log\pa{\frac{s(x,t)}{\eps_s}}dx$$
$$	-\frac{k_fk_ck}{2k_r}\int_{\O}e(x,t)s(x,t)dx - \frac{k_ck}{2}\int_{\O}c(x,t)dx=\mathbf{I}+\mathbf{II}+\mathbf{III}$$
where
\begin{equation*}
\begin{aligned}
\mathbf{I}=\int_{\O}d_e\Delta e(x,t)\log\pa{\frac{e(x,t)}{\einf}}dx +\int_{\O\cap \br{x\;|\;c(x,t)\geq \eps_c}}d_c\Delta c(x,t)\log\pa{\frac{c(x,t)}{\eps_c}}dx\\
+\int_{\O\cap \br{x\;|\;s(x,t)\geq \eps_s}}d_s\Delta s(x,t)\log\pa{\frac{s(x,t)}{\eps_s}}dx,
\end{aligned}
\end{equation*}
\begin{equation}\label{II}
\begin{aligned}
\mathbf{II}=\int_{\O}&\pa{-k_fe(x,t)s(x,t)+\pa{k_r+k_c}c(x,t)}\log\pa{\frac{e(x,t)}{\einf}}dx\\
&\quad +\int_{\O\cap \br{x\;|\;c(x,t)\geq \eps_c}}\pa{k_fe(x,t)s(x,t)-\pa{k_r+k_c}c(x,t)}\log\pa{\frac{c(x,t)}{\eps_c}}dx\\
&\quad +\int_{\O\cap \br{x\;|\;s(x,t)\geq \eps_s}}\pa{-k_fe(x,t)s(x,t)+k_r c(x,t)}\log\pa{\frac{s(x,t)}{\eps_s}}dx
\end{aligned}
\end{equation}
and
\begin{equation}\label{eq:def_of_III}
	\mathbf{III}=-\frac{k_fk_ck}{2k_r}\int_{\O}e(x,t)s(x,t)dx - \frac{k_ck}{2}\int_{\O}c(x,t)dx = -\frac{kk_c}{2}\int_{\O}\m(x)dx.
\end{equation}
As $\mathbf{III}$ is already a multiple of the integration of $\m$, we are left with estimating $\mathbf{I}$ and $\mathbf{II}$. \\
To simplify the coming integrals we will drop the $t$ variable (even though the division to domains we will use in the estimation of $\mathbf{II}$ will depend on it via $c(x,t)$, $s(x,t)$ and $e(x,t)$). Using Lemma \ref{lem:truncated_log_sobolev} (whose conditions are satisfied according to the assumptions) when all diffusion constants are strictly positive (and thus $\einf$ and $\eps_c$ are constants) we find that
$$\mathbf{I} \leq  -d_e\int_{\O}\frac{\abs{\na e(x)}^2}{e(x)}dx -C_{\mathrm{LSI}}d_c \int_{\O\cap \br{x\;|\;c(x)\geq \eps_c}}h\pa{c(x)|\overline{c_{\eps_c}}}dx$$
$$-C_{\mathrm{LSI}}d_s \int_{\O\cap \br{x\;|\;s(x)\geq \eps_s}}h\pa{s(x)|\overline{s_{\eps_s}}}dx$$
from which we attain, using the log-Sobolev inequality on $\O$ \eqref{eq:LSI}
\begin{equation}\label{eq:estimation_on_I}
	\begin{split}
		\mathbf{I} \leq - C_{\mathrm{LSI}}\Big(d_e\int_{\O}h(e(x)|\overline{e})dx+&d_c \int_{\O\cap \br{x\;|\;c(x)\geq \eps_c}}h\pa{c(x)|\overline{c_{\eps_c}}}dx\\
		&+d_s \int_{\O\cap \br{x\;|\;s(x)\geq \eps_s}}h\pa{s(x)|\overline{s_{\eps_s}}}dx\Big),
	\end{split}
\end{equation}
where $\overline{e}=\int_{\O}e(x)dx$. In the case $d_e=d_c=0$ the above remains true as in this case
\begin{equation*}
	\begin{aligned}
		\mathbf{I}=\int_{\O\cap \br{x\;|\;s(x,t)\geq \eps_s}}d_s\Delta s(x,t)\log\pa{\frac{s(x,t)}{\eps_s}}dx,
	\end{aligned}
\end{equation*}
and applying Lemma \ref{lem:truncated_log_sobolev} yields the desired result. It is important to note that we are allowed to use this lemma since $\eps_s$ is \textit{a constant} (as was briefly discussed in Remark \ref{rem:eps_condition_equivalent}).
\\
The estimation of $\mathbf{II}$ is slightly more complicated and will require us to both divide the domain $\O$ into the subdomains $\O_1,\O_2,\O_3$ and $\O_4$, defined in \eqref{eq:def_of_domains}, and to consider the two difference cases for $\eps_c$ and $\einf$
$$\begin{array}{ccc}
	\einf=M_0\quad& \eps_c\text{ is a constant that satisfies }\eqref{eq:eps_conection} & d_e,d_c,d_s>0\\ 
	\einf(x)=e_0(x)+c_0(x)\quad&  \eps_c\text{ is a function that satisfies }\eqref{eq:eps_conection_partial_diffusion} & d_e=d_c=0
\end{array}.$$
Writing $\mathbf{II}=\int_{\O}\r(x)dx$
with 
$$\r(x)=\pa{-k_fe(x)s(x)+\pa{k_r+k_c}c(x)}\log\pa{\frac{e(x)}{\einf}}$$
$$+\pa{k_fe(x)s(x)-\pa{k_r+k_c}c(x)}\chi_{\br{z\;|\;c(z)\geq \eps_c}}(x)\log\pa{\frac{c(x)}{\eps_c}}$$
$$+\pa{-k_fe(x)s(x)+k_r c(x)}\chi_{\br{z\;|\;s(z)\geq \eps_s}}(x)\log\pa{\frac{s(x)}{\eps_s}}$$
we see that:\\
\underline{For  $x\in \O_1=\br{x\;|\;c(x)\geq \eps_c,\; s(x)\geq  \eps_s}$:} 
$$\r(x)=-\pa{k_fe(x)s(x)-k_r c(x)}\log\pa{\frac{e(x)s(x)\eps_c}{\einf \eps_s c(x)}}+k_c c(x)\log\pa{\frac{e(x)\eps_c}{\einf c(x)}}$$
$$=-k_f c(x)\pa{\frac{e(x)s(x)}{c(x)}-\frac{\einf \eps_s}{\eps_c}}\log\pa{\frac{\frac{e(x)s(x)}{c(x)}}{\frac{\einf \eps_s}{\eps_c}}}
-k_c c(x)\log\pa{\frac{\frac{c(x)}{e(x)}}{\frac{\eps_c}{\einf}}}$$
$$=-k_fc(x)\bh\pa{\frac{e(x)s(x)}{c(x)}\Big| \frac{\einf \eps_s}{\eps_c}}-k_ce(x)h\pa{\frac{c(x)}{e(x)}\Big| \frac{\eps_c}{\einf}}-k_c c(x)+k_c \eps_c \frac{e(x)}{\einf}$$
$$=-\d_{\eps_c,\eps_s}(x) +C_{\mathrm{LSI}}d_sh\pa{s(x)|\overline{s_{\eps_s}}}+C_{\mathrm{LSI}}d_ch\pa{c(x)|\overline{c_{\eps_c}}}-k_c c(x)+k_c \eps_c \frac{e(x)}{\einf}$$
where we have used condition \eqref{eq:eps_conection} when all the diffusion coefficients are strictly positive and condition \eqref{eq:eps_conection_partial_diffusion} when $d_e=d_c=0$ together with the definition of $\d_{\eps_c,\eps_s}$, \eqref{eq:entropy_dissipation_density}. The last term will be estimated for our two distinct cases.
\begin{itemize}
	\item \textit{All diffusion coefficients are strictly positive.}  In this case we see that by using the inequality
	\begin{equation}\label{eq:important_inequality}
		x-1 \leq 6\pa{\sqrt{x}-1}^2,\quad \forall x\geq 2,
	\end{equation}
	and the fact that $c(x)\geq \eps_c$ on $\O_1$, 
	$$-k_c c(x)+k_c \eps_c \frac{e(x)}{\einf} \leq \begin{cases}
		k_c \pa{2\eps_c-c(x)}	& e(x) \leq 2\einf\\
		k_c \eps_c\pa{\frac{e(x)}{\einf}-1} & e(x) \geq 2\einf
	\end{cases}$$
	$$\leq \begin{cases}
		k_c c(x)	& e(x) \leq 2\einf\\
		\frac{6k_c\eps_c}{\einf}\pa{\sqrt{e}-\sqrt{\einf}}^2 & e(x) \geq 2\einf
	\end{cases}.$$
	Since
	\begin{equation}\label{eq:total_mass_and_average_of_e}
		\overline{e}=\int_\O e(x)dx \leq \int_{\O}\pa{e(x)+c(x)}dx= M_0=\einf
	\end{equation}
	we have that for $e(x)\geq \einf$
	$$\pa{\sqrt{e(x)}-\sqrt{\einf}}^2 \leq \pa{\sqrt{e(x)}-\sqrt{\overline{e}}}^2.$$
	Combining the above we see that
	\begin{equation}\label{eq:estimation_on_O_1_integrand}
		\begin{gathered}
			\r(x) \leq -\d_{\eps_c,\eps_s}(x) +C_{\mathrm{LSI}}d_sh\pa{s(x)|\overline{s_{\eps_s}}}+C_{\mathrm{LSI}}d_ch\pa{c(x)|\overline{c_{\eps_c}}}\\
			+k_c c(x)+\frac{6k_c\eps_c}{\einf}\pa{\sqrt{e(x)}-\sqrt{\overline{e}}}^2,\quad \forall x\in\O_1.
		\end{gathered}
	\end{equation}
	and as such 
		\begin{equation}\label{eq:estimation_on_O_1_integral_full_diffusion}
		\begin{split}
			\int_{\O_1}\r(x)dx \leq &\int_{\O_1}\pa{-\d_{\eps_c,\eps_s}(x) +C_{\mathrm{LSI}}d_sh\pa{s(x)|\overline{s_{\eps_s}}}+C_{\mathrm{LSI}}d_ch\pa{c(x)|\overline{c_{\eps_c}}}}dx\\
			+&k_c\int_{\O_1}c(x)dx +\frac{6k_c\eps_c}{\einf}\int_{\O_1}\pa{\sqrt{e(x)}-\sqrt{\overline{e}}}^2dx. 
		\end{split}
	\end{equation}
	\item \textit{$d_e=d_c=0$}. In this case we see that as 
	\begin{equation}\label{eq:e_control_in_partial_diffusion}
		e(x) \leq e(x)+c(x) =e_0(x)+c_0(x)=\einf(x),
	\end{equation}
	(the conservation law \eqref{law3} holds for strong solutions) the fact that $c(x) \geq \eps_c(x)$ on $\O_1$ implies that
	$$-k_c c(x)+k_c \eps_c(x) \frac{e(x)}{\einf(x)} \leq 0$$
	yielding the bound 
		\begin{equation}\label{eq:estimation_on_O_1_integral_partial_diffusion}
		\begin{split}
			\int_{\O_1}\r(x)dx \leq &\int_{\O_1}\pa{-\d_{\eps_c,\eps_s}(x) +C_{\mathrm{LSI}}d_sh\pa{s(x)|\overline{s_{\eps_s}}}}dx.
		\end{split}
	\end{equation}
\end{itemize}

\underline{For  $x\in \O_2=\br{x\in\O\;|\; c(x)< \eps_c,\;s(x)\geq  \eps_s}$:}  
\begin{equation}\label{b9}
\begin{aligned}
\r(x)&=\pa{-k_f e(x)s(x)+k_r c(x)}\log\pa{\frac{e(x)s(x)}{\einf \eps_s}} + k_c c(x)\log\pa{\frac{e(x)}{\einf}}\\
&=-k_f h\pa{e(x)s(x)|\einf\eps_s}-k_fe(x)s(x)+k_f\einf \eps_s \\
&\quad -k_rc(x) \h\pa{\frac{e(x)s(x)}{\einf\eps_s}}+k_r c(x)\pa{\frac{e(x)s(x)}{\einf \eps_s}-1}\\
&\quad -k_cc(x) \h\pa{\frac{e(x)}{\einf}}+k_c c(x)\pa{\frac{e(x)}{\einf }-1}.
\end{aligned}
\end{equation}
Thus
\begin{equation}\label{eq:estimation_on_r_for_omega_2}
	\begin{gathered}
		\r(x)=-\d_{\eps_c,\eps_s}(x) +C_{\mathrm{LSI}}d_sh\pa{s(x)|\overline{s_{\eps_s}}} +\underbrace{\pa{k_c+k_r} c(x)\pa{\frac{e(x)}{\einf }-1}}_{\mathbf{A}}\\
			+\underbrace{k_r\frac{c(x)e(x)}{\einf}\pa{\frac{s(x)}{\eps_s}-1}+k_f\eps_s\pa{\einf-\frac{e(x)s(x)}{\eps_s}}}_{\mathbf{B}}.
	\end{gathered}
\end{equation}
Again, to estimate $\mathbf{A}$ and $\mathbf{B}$ we will consider our two cases.
\begin{itemize}
	\item \textit{All diffusion coefficients are strictly positive.}  In this case we see that, much like the estimation on $\O_1$
	$$\mathbf{A} = \pa{k_c+k_r} c(x)\pa{\frac{e(x)}{\einf }-1} \leq \begin{cases}
		\pa{k_c+k_r} c(x)	& e(x)\leq 2\einf \\
		\frac{6\pa{k_c+k_r} c(x)}{\einf}\pa{\sqrt{e(x)}-\sqrt{\einf}}^2 & e(x)\geq 2\einf
	\end{cases}$$
	$$\leq \pa{k_c+k_r} c(x)+\frac{6\pa{k_c+k_r} \eps_c}{\einf} \pa{\sqrt{e(x)}-\sqrt{\overline{e}}}^2,$$
	where we have used the fact that $c(x)\leq \eps_c$ on $\O_2$. Using this fact again, together with the fact that $s(x)\geq \eps_s$ on $\O_2$ and condition \eqref{eq:eps_conection}, we find that
	$$\mathbf{B} \leq k_r\eps_c\frac{e(x)}{\einf}\pa{\frac{s(x)}{\eps_s}-1}+k_f\eps_s\pa{\einf-\frac{e(x)s(x)}{\eps_s}}$$
	$$=k_f\eps_s e(x)\pa{\frac{s(x)}{\eps_s}-1}+k_f\eps_s\pa{\einf-\frac{e(x)s(x)}{\eps_s}}=k_f\eps_s\pa{\einf-e(x)}.$$
	Thus
	$$\int_{\O_2}\mathbf{B}dx \leq k_f \eps_s\int_{\O_2}\pa{\pa{\einf -\overline{e}}+\pa{\overline{e}-e(x)}}dx$$
	$$=\underbrace{k_f\eps_s\abs{\O_2}\int_{\O}c(x)dx}_{\text{from }\eqref{eq:total_mass_and_average_of_e}}+k_f\int_{\O_2}\eps_s\pa{\overline{e}-e(x)}dx$$
	$$\leq k_f\eps_s\int_{\O}c(x)dx+ k_f\int_{\O_2\cap\br{x\;|\;e(x)\leq \overline{e}\leq 2e(x)}}\eps_s\pa{\overline{e}-e(x)}dx+k_f\int_{\O_2\cap\br{x\;|\;\overline{e}\geq 2e(x)} }\eps_s\pa{\overline{e}-e(x)}dx$$
	$$\leq k_f\eps_s\int_{\O}c(x)dx+ k_f\int_{\O_2\cap\br{x\;|\;e(x)\leq \overline{e}\leq 2e(x)}}\eps_se(x)dx+k_f\eps_s\int_{\O_2\cap\br{x\;|\;\frac{\overline{e}}{e(x)}\geq 2} }e(x)\pa{\frac{\overline{e}}{e(x)}-1}dx$$
	$$\leq k_f\eps_s\int_{\O}c(x)dx+k_f\int_{\O_2}e(x)s(x)dx+6k_f\eps_s \int_{\O_2}\pa{\sqrt{e(x)}-\sqrt{\overline{e}}}^2dx,$$
	where we have used the fact that $s(x)\geq \eps_s$ again, as well as inequality \eqref{eq:important_inequality}.
	
	Combining the estimations on $\mathbf{A}$ and $\mathbf{B}$ with \eqref{eq:estimation_on_r_for_omega_2} yields
	\begin{equation}\label{eq:estimation_on_O_2_integral_full_diffusion}
		\begin{split}
			\int_{\O_2}\r(x)dx \leq &\int_{\O_2}\pa{-\d_{\eps_c,\eps_s}(x) +d_sh\pa{s(x)|\overline{s_{\eps_s}}}}dx\\
			+&\pa{k_c+k_r}\int_{\O_2}c(x)dx+k_f\eps_s\int_{\O}c(x)dx+k_f\int_{\O_2}e(x)s(x)dx\\
			+&6\pa{k_f\eps_s+\frac{\pa{k_r+k_c}\eps_c}{\einf}}\int_{\O_2}\pa{\sqrt{e(x)}-\sqrt{\overline{e}}}^2dx.
		\end{split}
	\end{equation}
	\item \textit{$d_e=d_c=0$}. In this case we have that as $e(x)\leq \einf(x)$
	$$\mathbf{A} = \pa{k_c+k_r} c(x)\pa{\frac{e(x)}{\einf(x) }-1} \leq 0, $$
	and using condition \eqref{eq:eps_conection_partial_diffusion} together with the facts that $c(x)< \eps_c(x)$ and $s(x)\geq \eps_s$ we have that exactly like in the previous case
	$$\mathbf{B} \leq k_f \eps_s\pa{\einf(x)-e(x)} = k_f\eps_s c(x),$$
	where we have used \eqref{eq:e_control_in_partial_diffusion} in the last step. We conclude that in this case
		\begin{equation}\label{eq:estimation_on_O_2_integral_partial_diffusion}
		\begin{split}
			\int_{\O_2}\r(x)dx \leq &\int_{\O_2}\pa{-\d_{\eps_c,\eps_s}(x) +d_sh\pa{s(x)|\overline{s_{\eps_s}}}}dx + k_f\eps_s\int_{\O_2}c(x)dx.
		\end{split}
	\end{equation}
\end{itemize}

\underline{For  $x\in \O_3=\br{x\in\O\;|\; c(x)\geq \eps_c,\;s(x)< \eps_s}$:}  
\begin{equation}\label{b12}
\begin{aligned}
\r(x)&=-\pa{k_r+k_c}c(x)\log\pa{\frac{\frac{c(x)}{e(x)}}{\frac{\eps_c}{\einf}}}-k_f e(x)s(x)\log\pa{\frac{\frac{e(x)}{c(x)}}{\frac{\einf}{\eps_c}}}\\
&=-\pa{k_r+k_c}e(x)h\pa{\frac{c(x)}{e(x)}\Big| \frac{\eps_c}{\einf}}+\frac{k_r+k_c}{\einf}\pa{\eps_c e(x)-\einf c(x)}\\
&\quad-k_fs(x)c(x)h\pa{\frac{e(x)}{c(x)}\Big| \frac{\einf}{\eps_c}}-\frac{k_f s(x)}{\eps_c}\pa{ \eps_c e(x)-\einf c(x)}.
\end{aligned}
\end{equation}
Thus
\begin{equation}\label{eq:estimation_on_r_for_omega_3}
	\begin{split}
		\r(x)=-\d_{\eps_c,\eps_s}(x) &+C_{\mathrm{LSI}}d_ch\pa{c(x)|\overline{c_{\eps_c}}} \\
		&+\underbrace{\pa{\frac{k_r+k_c}{\einf}-\frac{k_f s(x)}{\eps_c}}\pa{\eps_c e(x)-\einf c(x)}}_{\mathbf{D}}.
	\end{split}
\end{equation}
We will estimate $\mathbf{D}$ in our two distinct cases.
\begin{itemize}
	\item \textit{All diffusion coefficients are strictly positive.}  In this case we see that if $\eps_c e(x) - \einf c(x) \leq 0$ then since $s(x) \leq \eps_s$ on $\O_3$
$$-\frac{k_f s(x)}{\eps_c}\pa{\eps_c e(x)-\einf c(x)}\leq -\frac{k_f \eps_s}{\eps_c}\pa{\eps_c e(x)-\einf c(x)}=-\frac{k_r }{\einf}\pa{\eps_c e(x)-\einf c(x)},$$
where we have used \eqref{eq:eps_conection}. As such
$$\mathbf{D} \leq \frac{k_c}{\einf}\pa{\eps_c e(x)-\einf c(x)}\leq 0.$$
If, on the other hand, $\eps_c e(x) - \einf c(x) \geq 0$ then 
$$\mathbf{D}  \leq\frac{k_r+k_c}{\einf}\pa{\eps_c e(x)-\einf c(x)}$$
i.e. for all $x\in\O_3$
\begin{equation}\label{eq:estimation_on_D_for_O_3}
	\mathbf{D}  \leq \max\pa{\frac{k_r+k_c}{\einf}\pa{\eps_c e(x)-\einf c(x)},0}.
\end{equation}
Using the fact that $c(x)\geq \eps_c$ on $\O_3$ we conclude that
$$\frac{k_r+k_c}{\einf} \pa{\eps_c e(x)-\einf c(x)}\leq \pa{k_r+k_c} \eps_c \pa{\frac{e(x)}{\einf}-1}$$
$$\leq \begin{cases}
	\pa{k_r+k_c}\eps_c& e(x)\leq 2\einf  \\
	\frac{6\pa{k_r+k_c} \eps_c}{\einf}\pa{\sqrt{e(x)}-\sqrt{\einf}}^2 & e(x)\geq 2\einf  
\end{cases}$$
$$\leq \pa{k_r+k_c} c(x)+\frac{6\pa{k_r+k_c}\eps_c}{\einf}\pa{\sqrt{e(x)}-\sqrt{\overline{e}}}^2$$

where we once again used inequality \eqref{eq:important_inequality} and a similar calculation to that we have performed when investigating $\O_1$.

From the above, \eqref{eq:estimation_on_r_for_omega_3} and \eqref{eq:estimation_on_D_for_O_3} we conclude that
\begin{equation}\label{eq:estimation_on_O_3_integrand}
	\begin{split}
		\r(x)dx \leq &-\d_{\eps_c,\eps_s}(x) +d_ch\pa{c(x)|\overline{c_{\eps_c}}}dx\\
		+&\pa{k_c+k_r} c(x)+\frac{6\pa{k_c+k_r}\eps_c}{\einf}\pa{\sqrt{e(x)}-\sqrt{\overline{e}}}^2
	\end{split}
\end{equation}
and as such
\begin{equation}\label{eq:estimation_on_O_3_integral_full_diffusion}
	\begin{split}
		\int_{\O_3}\r(x)dx \leq &\int_{\O_3}\pa{-\d_{\eps_c,\eps_s}(x) +d_ch\pa{c(x)|\overline{c_{\eps_c}}}}dx\\
		+&\pa{k_c+k_r}\int_{\O_3}c(x)dx+\frac{6\pa{k_r+k_c}\eps_c}{\einf}\int_{\O_3}\pa{\sqrt{e(x)}-\sqrt{\overline{e}}}^2dx.
	\end{split}
\end{equation}
	\item \textit{$d_e=d_c=0$}. In this case we see that since $\eps_c(x)\leq c(x)$ and $e(x)\leq \einf(x)$
	$$\eps_c(x)e(x) - \einf(x)c(x) \leq 0.$$
	Using condition \eqref{eq:eps_conection_partial_diffusion} instead of \eqref{eq:eps_conection} and following the same estimation that was shown in the previous case we find that
	$$\mathbf{D} \leq \frac{k_c}{\einf(x)}\pa{\eps_c(x) e(x)-\einf c(x)}\leq 0.$$
	We conclude that in this case
	\begin{equation}\label{eq:estimation_on_O_3_integral_partial_diffusion}
		\begin{split}
			\int_{\O_3}\r(x)dx \leq &-\int_{\O_3}\d_{\eps_c,\eps_s}(x) dx.
		\end{split}
	\end{equation}
\end{itemize}

\underline{For  $x\in \O_4=\br{x\in\O\;|\; c(x)<  \eps_c,\;s(x)< \eps_s}$:}  
\begin{equation}\label{b14}
\begin{aligned}
\r(x)&=\pa{-k_fe(x)s(x)+\pa{k_r+k_c}c(x)}\log\pa{\frac{e(x)}{\einf}}\\
&=-k_f s(x)h\pa{e(x)|\einf}-k_fs(x)e(x)+k_f\einf s(x)\\
&\quad -\pa{k_r+k_c}c(x)\h\pa{\frac{e(x)}{\einf}} + \frac{\pa{k_r+k_c}}{\einf}c(x)e(x)-\pa{k_r+k_c}c(x).
\end{aligned}
\end{equation}
Thus
\begin{equation}\label{eq:estimation_on_r_for_omega_4}
	\begin{gathered}
		\r(x)=-\d_{\eps_c,\eps_s}(x) +k_fs(x)\pa{\einf-e(x)} + \frac{\pa{k_r+k_c}c(x)}{\einf}\pa{e(x)-\einf}.
	\end{gathered}
\end{equation}
Unsurprisingly, the last term will be estimated for our two distinct cases.
\begin{itemize}
	\item \textit{All diffusion coefficients are strictly positive.}  In this case we notice that following similar ideas to those presented in the investigation of $\O_2$ and the fact that $s(x)\leq \eps_s$ on $\O_4$ we find that
	$$\int_{\O_4}k_fs(x)\pa{\einf-e(x)}dx =k_f\underbrace{\pa{\einf-\overline{e}}}_{\geq 0}\int_{\O_4}s(x)dx+ k_f \int_{\O_4 }s(x)\pa{\overline{e}-e(x)}dx$$
	$$\leq k_f \eps_s\abs{\O_4}\underbrace{\int_{\O}c(x)dx}_{\text{from }\eqref{eq:total_mass_and_average_of_e}}+k_f \int_{\O_4 \cap \br{x\;|\; e(x)\leq \overline{e} \leq 2e(x)}}s(x)\pa{\overline{e}-e(x)}dx$$
	$$+k_f \int_{\O_4 \cap \br{x\;|\; \overline{e} \geq 2e(x) }}s(x)\pa{\overline{e}-e(x)}dx$$
	$$\leq  k_f \eps_s \int_{\O}c(x)dx+k_f \int_{\O_4 }e(x)s(x)dx+6k_f\eps_s \int_{\O_4}\pa{\sqrt{e(x)}-\sqrt{\overline{e}}}^2dx.$$
	Moreover, much like previous estimations (for instance on $\O_3$) we find that as $c(x)\leq \eps_c$ on $\O_4$
	$$\frac{\pa{k_r+k_c}c(x)}{\einf}\pa{e(x)-\einf} \leq \begin{cases}
		\pa{k_r+k_c}c(x)	& e(x)\leq 2\einf \\
		\frac{6\pa{k_r+k_c}\eps_c}{\einf}\pa{\sqrt{e(x)}-\sqrt{\einf} }^2  &  e(x)\geq 2\einf
	\end{cases} $$
	$$\leq \pa{k_r+k_c}c(x)+ \frac{6\pa{k_r+k_c}\eps_c}{\einf}\pa{\sqrt{e(x)}-\sqrt{\overline{e}} }^2.$$
	These inequalities together with \eqref{eq:estimation_on_r_for_omega_4} yield 
	\begin{equation}\label{eq:estimation_on_O_4_integral_full_diffusion}
		\begin{split}
			\int_{\O_4}\r(x)dx \leq &-\int_{\O_4}\d_{\eps_c,\eps_s}(x) dx +  k_f \eps_s \int_{\O}c(x)dx+k_f \int_{\O_4 }e(x)s(x)dx\\
			+&\pa{k_c+k_r}\int_{\O_4}c(x)dx+6\pa{\frac{\pa{k_r+k_c}\eps_c}{\einf}+k_f\eps_s}\int_{\O_4}\pa{\sqrt{e(x)}-\sqrt{\overline{e}}}^2dx.
		\end{split}
	\end{equation}
	\item \textit{$d_e=d_c=0$}. In this case, since $e(x)\leq \einf(x)$ and $s(x)\leq \eps_s$ we find that due to \eqref{eq:e_control_in_partial_diffusion}
	$$k_fs(x)\pa{\einf(x)-e(x)} + \frac{\pa{k_r+k_c}c(x)}{\einf(x)}\pa{e(x)-\einf(x)} \leq k_fs(x)c(x) \leq k_f\eps_s c(x).$$
	We conclude that in this case
	\begin{equation}\label{eq:estimation_on_O_4_integral_partial_diffusion}
		\begin{split}
			\int_{\O_4}\r(x)dx \leq &-\int_{\O_4}\d_{\eps_c,\eps_s}(x) dx +  k_f \eps_s \int_{\O_4}c(x)dx.
		\end{split}
	\end{equation}
\end{itemize}

Using the fact that $\O_1$, $\O_2$, $\O_3$ and $\O_4$ are mutually disjoint with 
$$\bigcup_{i=1}^4 \O_i=\O,\quad \O_1\cup\O_2=\br{x\;|\; s(x)\geq \eps_s},\quad \O_1\cup\O_3=\br{x\;|\; c(x)\geq \eps_c},$$
and the fact that
$$\mathbf{II}=\int_{\O_1}\r(x)dx+\int_{\O_2}\r(x)dx+\int_{\O_3}\r(x)dx+\int_{\O_4}\r(x)dx$$
we see that \eqref{eq:estimation_on_O_1_integral_full_diffusion}, \eqref{eq:estimation_on_O_2_integral_full_diffusion}, \eqref{eq:estimation_on_O_3_integral_full_diffusion} and \eqref{eq:estimation_on_O_4_integral_full_diffusion} imply that when all diffusion constants are strictly positive
\begin{equation}\label{eq:esitmation_on_II_full_diffusion}
	\begin{gathered}
		\mathbf{II} \leq -\int_{\O}\d_{\eps_c,\eps_s}(x)dx +\clsi \int_{\O \cap \br{x\;|\; s(x)\geq \eps_s}}d_sh\pa{s(x)|\overline{s_{\eps_s}}}dx\\
		+\clsi \int_{\O \cap \br{x\;|\; c(x)\geq \eps_c}}d_ch\pa{c(x)|\overline{c_{\eps_c}}}dx\\
		+ \pa{k_c+k_r+2k_f\eps_s}\int_{\O}c(x)dx + k_f\int_{\O}e(x)s(x)dx \\ +6\pa{\frac{\pa{k_c+k_r}\eps_c}{\einf}+k_f\eps_s}\int_{\O}\pa{\sqrt{e(x)}-\sqrt{\overline{e}}}^2dx
	\end{gathered}
\end{equation}
and \eqref{eq:estimation_on_O_1_integral_partial_diffusion}, \eqref{eq:estimation_on_O_2_integral_partial_diffusion}, \eqref{eq:estimation_on_O_3_integral_partial_diffusion} and \eqref{eq:estimation_on_O_4_integral_partial_diffusion} imply that when $d_e=d_c=0$
\begin{equation}\label{eq:esitmation_on_II_partial_diffusion}
\begin{gathered}
	\mathbf{II} \leq -\int_{\O}\d_{\eps_c,\eps_s}(x)dx +\clsi \int_{\O \cap \br{x\;|\; s(x)\geq \eps_s}}d_sh\pa{s(x)|\overline{s_{\eps_s}}}dx
	+ k_f\eps_s\int_{\O}c(x)dx. 
\end{gathered}
\end{equation}
Combining the \eqref{eq:esitmation_on_II_full_diffusion} with \eqref{eq:def_of_III}, \eqref{eq:estimation_on_I} and the fact that
$$\frac{d}{dt}\E_{\eps_c,\eps_s,k}(t) = \mathbf{I}+\mathbf{II}+\mathbf{III}$$
yields the estimation
\begin{equation}\nonumber
	\begin{gathered}
		\frac{d}{dt}\E_{\eps_c,\eps_s,k}(t)  \leq -\int_{\O}\d_{\eps_c,\eps_s}(x,t)dx -d_e\clsi\int_{\O }h\pa{e(x,t)|\overline{e(t)}}dx\\
		+ \pa{k_c+k_r+2k_f\eps_s-\frac{k_ck}{2}}\int_{\O}c(x,t)dx + \pa{k_r-\frac{k_ck}{2}}\frac{k_f}{k_r}\int_{\O}e(x,t)s(x,t)dx \\ +6\pa{\frac{\pa{k_c+k_r}\eps_c}{\einf}+k_f\eps_s}\int_{\O}\pa{\sqrt{e(x,t)}-\sqrt{\overline{e(t)}}}^2dx\\
	\end{gathered}
\end{equation}
when all diffusion coefficients are strictly positive which, together with the inequality $\pa{\sqrt{x}-\sqrt{y}}^2\leq h\pa{x|y}$ and the definition of $\m(x)$, shows that
\begin{equation}\nonumber 
	\begin{gathered}
		\frac{d}{dt}\E_{\eps_c,\eps_s,k}(t)  \leq -\int_{\O}\d_{\eps_c,\eps_s}(x,t)dx -\pa{\frac{kk_c}{2}-k_c-k_r-2k_f\eps_s}\int_{\O}\m(x,t)\\
		-\pa{d_e\clsi -6\pa{\frac{\pa{k_c+k_r}}{\einf}+k_f}\max\pa{\eps_c,\eps_s}} \int_{\O }h\pa{e(x,t)|\overline{e(t)}}dx,
		\end{gathered}
\end{equation}
which is the desired inequality in this case.\\
Similarly \eqref{eq:esitmation_on_II_partial_diffusion} will imply that when $d_e=d_c=0$
\begin{equation}\nonumber
	\begin{gathered}
		\frac{d}{dt}\E_{\eps_c,\eps_s,k}(t)  \leq -\int_{\O}\d_{\eps_c,\eps_s}(x,t)dx 
		+ k_f\eps_s\int_{\O}c(x)dx-\frac{k_ck}{2}\int_{\O}\m(x,t)dx,
	\end{gathered}
\end{equation}
	showing the second desired inequality. The proof is thus complete.
\end{proof}
We now have the tools to show our main theorem for this section.
\begin{proof}[Proof of Theorem \ref{thm:entropic_convergence}]
	Following from Theorem \ref{thm:first_entropy_inequality} we see that when all diffusion coefficients are strictly positive \eqref{eq:entropic_inequality} will follow immediately from \eqref{eq:first_entropy_inequality_full_diffusion} if
		\begin{equation}\label{eq:validity_condition_for_full_diffusion}
		\begin{gathered}
			\gamma\E_{\eps_c,\eps_s,k}(t)  \leq \int_{\O}\Big( \d_{\eps_c,\eps_s}(x,t)+ \pa{\frac{kk_c}{2}-k_c-k_r-2k_f\eps_s}\m(x,t)\\
			+\pa{d_e\clsi-6\pa{\frac{\pa{k_c+k_r}}{M_0}+k_f}\max\pa{\eps_c,\eps_s}}  h\pa{e(x,t)|\overline{e(t)}}\Big)dx.
		\end{gathered}
	\end{equation}
and when $d_e=d_c=0$ \eqref{eq:entropic_inequality} will follow immediately from \eqref{eq:first_entropy_inequality_partial_diffusion} if
\begin{equation}\label{eq:validity_condition_for_partial_diffusion}
	\begin{gathered}
		\gamma\E_{\eps_c,\eps_s,k}(t)  \leq \int_{\O}\d_{\eps_c,\eps_s}(x,t)dx 
		+\pa{\frac{k_ck}{2}-k_f\eps_s}\int_{\O}\m(x,t)dx.
	\end{gathered}
\end{equation}
	To show \eqref{eq:validity_condition_for_full_diffusion} and \eqref{eq:validity_condition_for_partial_diffusion} we will use the definition of $\E_{\eps_c,\eps_s,k}$,
\begin{equation}\nonumber
	\begin{split}
		\E_{\eps_c,\eps_s,k}\pa{e,c,s}=& \int_{\O} h(e(x)|e_\infty)dx +\int_{\O} h_{\eps_c}(c(x)|\eps_c)dx +\int_{\O} h_{\eps_s}(s(x)|\eps_s)dx\\
		+&k\int_{\O}\pa{c(x) + \frac {1}{2}\pa{\frac{2k_r+k_c}{k_r}}s(x)}dx,
	\end{split}
\end{equation}
and bound each term in the above expression by terms that appear in the right hand side of \eqref{eq:validity_condition_for_full_diffusion} or \eqref{eq:validity_condition_for_partial_diffusion}.\\
Much like in the proof of Theorem \ref{thm:first_entropy_inequality} we shall drop the $t$ variable from our estimation, showing, as was mentioned in \S\ref{subsec:main}, that the connection between the total entropy and an appropriate production term is of functional nature.\\ 
\underline{The term $h\pa{e(x)|e_\infty}$:}
\begin{itemize}
	\item \textit{All diffusion coefficients are strictly positive.} In this case $\einf=M_0$ and using the identity
	\begin{equation}\label{eq:entropic_identity}
		h(x|y)=h(x|z)+x\log\pa{\frac{z}{y}}+y-z
	\end{equation}
	we see that 
	\begin{equation}\label{eq:estimation_on_h_e_einf_full_diffusion}
		\begin{gathered}
			\int_{\O}h\pa{e(x)|\einf}dx = \int_{\O}\pa{h\pa{e(x)|\overline{e}}+e(x)\log\pa{\frac{\overline{e}}{\einf}}+\einf-\overline{e}}dx \\
			\leq \int_{\O}h\pa{e(x)|\overline{e}}dx + \pa{\einf -\overline{e}}=\int_{\O}h\pa{e(x)|\overline{e}}dx +\int_{\O}c(x)dx\\
			\leq \int_{\O}h\pa{e(x)|\overline{e}}dx +\int_{\O}\m(x)dx,
		\end{gathered}
	\end{equation}
	where we have used the fact that $\overline{e}\leq \einf = \overline{e}+\int_{\O}c(x)dx$.
	\item \textit{$d_e=d_c=0$}. In this case since $e(x)\leq e(x)+c(x)=e_0(x)+c_0(x)=\einf(x)$ we get that 
	$$h\pa{e(x)|\einf(x)}=e(x)\log\pa{\frac{e(x)}{\einf(x)}}-e(x)+\einf(x) \leq c(x)$$
	and as such
		\begin{equation}\label{eq:estimation_on_h_e_einf_partial_diffusion}
		\begin{gathered}
			\int_{\O}h\pa{e(x)|\einf(x)}dx \leq \int_{\O}c(x)dx \leq \int_{\O}\m(x)dx.
		\end{gathered}
	\end{equation}
\end{itemize}
\underline{The term $h_{\eps_c}\pa{c(x)|\eps_c}$:}
\begin{itemize}
	\item \textit{All diffusion coefficients are strictly positive.} In this case using \eqref{eq:entropic_identity} again we find that
	$$\int_{\O}h_{\eps_c}\pa{c(x)|\eps_c}dx=\int_{\br{x\;|\;c(x)\geq \eps_c}}h\pa{c(x)|\eps_c}dx$$
	$$=\int_{\br{x\;|\;c(x)\geq \eps_c}}\pa{h\pa{c(x)|\overline{c_{\eps_c}}}+c(x)\log\pa{\frac{\overline{c_{\eps_c}}}{\eps_c}}+\eps_c-\overline{c_{\eps_c}}}dx.$$
	Since
	$$\eps_c \leq \underbrace{\int_{\O}\max\pa{c(x),\eps_c}dx}_{\overline{c_{\eps_c}}} \leq \int_{\O}c(x)dx + \eps_c\leq M_0+\eps_c$$
	we see that 
	\begin{equation}\label{eq:estimation_on_h_eps_c_full_diffusion}
		\begin{gathered}
			\int_{\O}h_{\eps_c}\pa{c(x)|\eps_c}dx\leq  \int_{\br{x\;|\;c(x)\geq \eps_c}}h\pa{c(x)|\overline{c_{\eps_c}}}dx\\
			+\log\pa{1+\frac{M_0}{\eps_c}}\int_{\br{x\;|\;c(x)\geq \eps_c}}c(x)dx\\
			\leq \frac{1}{d_c\clsi}\int_{\O}\d_{\eps_c,\eps_s}(x)dx + \log\pa{1+\frac{M_0}{\eps_c}}\int_{\O}\m(x)dx.
		\end{gathered}
	\end{equation}
	\item \textit{$d_e=d_c=0$.} In this case we notice that as
	$$c(x) \leq c(x)+e(x)=c_0(x)+e_0(x)=\einf(x)$$
		and since \eqref{eq:eps_conection_partial_diffusion} holds we have that 
		$$\frac{c(x)}{\eps_c(x)}\leq \frac{\einf(x)}{\eps_c(x)}=\frac{k_r}{k_f \eps_s}.$$ 
		Thus
		\begin{equation}\label{eq:estimation_on_h_eps_c_partial_diffusion}
		\begin{gathered}
			\int_{\O}h_{\eps_c}\pa{c(x)|\eps_c}dx = \int_{\br{x\;|\;c(x)\geq \eps_c(x)}}\pa{c(x)\log\pa{\frac{c(x)}{\eps_c(x)}}-c(x)+\eps_c(x)}dx \\
		\leq  \int_{\br{x\;|\;c(x)\geq \eps_c(x)}}c(x)\log\pa{\frac{c(x)}{\eps_c(x)}}dx \leq \log\pa{1+\frac{k_r}{k_f\eps_s }}
		\int_{\O}c(x)dx\\
			\leq \log\pa{1+\frac{k_r}{k_f\eps_s }}\int_{\O}\m(x)dx.
		\end{gathered}
	\end{equation}
\end{itemize}
\underline{The term $h_{\eps_s}\pa{s(x)|\eps_s}$:}\\ 
 Similarly to our previous term we see that using the fact that\footnote{The conservation of mass $$\int_{\O}(s(x,t)+c(x,t)+p(x,t))dx =M_1$$ is valid in both cases.} 
	$$\eps_s \leq \underbrace{\int_{\O}\max\pa{s(x),\eps_s}dx}_{=\overline{s_{\eps_s}}} \leq \int_{\O}s(x)dx + \eps_s\leq M_1+\eps_s$$
	we get that
	\begin{equation}\label{eq:partial_estimation_on_h_eps_s}
		\begin{gathered}
			\int_{\O}h_{\eps_s}\pa{s(x)|\eps_s}dx\leq  \int_{\br{x\;|\;s(x)\geq \eps_s}}h\pa{s(x)|\overline{s_{\eps_s}}}dx+\log\pa{1+\frac{M_1}{\eps_s}}\int_{\O}s(x)dx \\
			\leq \frac{1}{d_s\clsi}\int_{\O}\d_{\eps_c,\eps_s}(x)dx +\log\pa{1+\frac{M_1}{\eps_s}}\int_{\O}s(x)dx.
		\end{gathered}
	\end{equation}
	In order to conclude the above estimation, and estimate the term that is connected to $\M(c,s)$ in $\E_{\eps_c,\eps_s,k}$, we will now bound $\int_{\O}s(x)dx$.
	
	We start by noticing that if $x\geq 8 y$ then 
	$$h(x|y)=x\pa{\log (x) -\log (y)-1}+y \geq x.$$
	As such
	\begin{equation}\label{b4}
		\begin{aligned}
			\int_{\br{x\;|\; s(x) \geq 8 \overline{s_{\eps_s}}}}s(x)dx &\leq \int_{\br{x\;|\; s(x) \geq 8 \overline{s_{\eps_s}}}}h\pa{s(x)|\overline{s_{\eps_s}}}dx\leq  \int_{\br{x\;|\;s(x)\geq \eps_s}}h\pa{s(x)|\overline{s_{\eps_s}}}dx\\
			&\leq \frac{1}{d_s\clsi}\int_{\O}\d_{\eps_c,\eps_s}(x)dx
		\end{aligned}
	\end{equation}
	To deal with the case where $s(x)\leq 8 \overline{s_{\eps_s}}$ we will need to consider our two cases separately.
	\begin{itemize}
		\item \textit{All diffusion coefficients are strictly positive.} In this case we need to consider two options:
		\begin{itemize}
			\item If $e(x) \leq \frac{\einf}{2}$ then as $h(x|y)$ is decreasing on $[0,y)$ we have that
			$$\min_{x\in \rpa{0,\frac{y}{2}}}h(x|y)=h\pa{\frac{y}{2}\Big| y}=\frac{\pa{1-\log (2)}y}{2}.$$
			and as such
			\begin{equation}\label{b5}
				\begin{aligned}
					\int_{\br{x\;|\; s(x) < 8 \overline{s_{\eps_s}}\wedge e(x)\leq \frac{\einf}{2}}}s(x)dx\leq \frac{16\overline{s_{\eps_s}}}{\pa{1-\log(2)}\einf}\int_{\br{x\;|\; s(x) < 8 \overline{s_{\eps_s}}\wedge e(x)\leq \frac{\einf}{2}}}h\pa{e(x)|\einf}dx\\
					\leq \frac{16\pa{\eps_s+M_1}}{\pa{1-\log(2)}\einf}\int_{\Omega}h\pa{e(x)|\einf}dx \leq \frac{16\pa{\eps_s+M_1}}{\pa{1-\log(2)}\einf}\pa{\int_{\O}h\pa{e(x)|\overline{e}}dx +\int_{\O}\m(x)dx}
				\end{aligned}
			\end{equation}
			where we have used \eqref{eq:estimation_on_h_e_einf_full_diffusion}.
			\item If $e(x) >\frac{\einf}{2}$ then
			\begin{equation}\label{b6}
				\int_{\br{x\;|\; s(x) < 8 \overline{s_{\eps_s}}\wedge e(x)>\frac{\einf}{2}}}s(x)dx \leq \frac{2}{\einf}\int_{\br{x\;|\; s(x) < 8 \overline{s_{\eps_s}}\wedge  e(x)\leq \frac{\einf}{2}}}e(x)s(x)
				\leq \frac{2k_r}{k_f\einf}\int_{\O}\m(x)dx.
			\end{equation}
		\end{itemize}
		Thus
		\begin{equation}\label{eq:estimation_of_integral_of_s_on_O_full_diffusion}
			\begin{split}
				\int_{\O}s(x)dx \leq \frac{1}{d_s\clsi}\int_{\O}&\d_{\eps_c,\eps_s}(x)dx +\frac{16\pa{\eps_s+M_1}}{\pa{1-\log(2)}\einf}\int_{\Omega}h\pa{e(x)|\overline{e}}dx\\
				&+ \pa{\frac{2k_r}{k_f\einf}+\frac{16\pa{\eps_s+M_1}}{\pa{1-\log(2)}\einf}}\int_{\O}\m(x)dx.
			\end{split}
		\end{equation}
	\item \textit{$d_e=d_c=0$.} The same options as in the first case need to be considered. The exact same calculation, together with condition \eqref{eq:lower_bound_for_einf}  and \eqref{eq:estimation_on_h_e_einf_partial_diffusion}, show that\footnote{When $e(x)\leq \frac{\einf(x)}{2}$ we have that
	$$1 \leq \frac{2h\pa{e(x)|\einf(x)}}{\pa{1-\log\pa{2}}\einf(x)} \leq \frac{2h\pa{e(x)|\einf(x)}}{\pa{1-\log\pa{2}}\beta} $$.}
			\begin{equation}\label{b5_partial}
		\begin{aligned}
			\int_{\br{x\;|\; s(x) < 8 \overline{s_\eps}\wedge e(x)\leq \frac{\einf(x)}{2}}}s(x)dx 
			\leq \frac{16\pa{\eps_s+M_1}}{\pa{1-\log(2)}\beta}\int_{\Omega}h\pa{e(x)|\einf(x)}dx \\
			\leq \frac{16\pa{\eps_s+M_1}}{\pa{1-\log(2)}\beta}\int_{\O}\m(x)dx.
		\end{aligned}
	\end{equation} 
	and 
		\begin{equation}\label{b6_partial}
		\int_{\br{x\;|\; s(x) < 8 \overline{s_{\eps_s}}\wedge e(x)>\frac{\einf(x)}{2}}}s(x)dx \leq \frac{2}{\beta}\int_{\br{x\;|\; s(x) < 8 \overline{s_{\eps_s}}\wedge e(x)> \frac{\beta}{2}}}e(x)s(x)
		\leq \frac{2k_r}{k_f\beta}\int_{\O}\m(x)dx,
	\end{equation}
from which we find that
\begin{equation}\label{eq:estimation_of_integral_of_s_on_O_partial_diffusion}
\begin{split}
	\int_{\O}s(x)dx \leq \frac{1}{d_s\clsi}\int_{\O}\d_{\eps_c,\eps_s}(x)dx + \pa{\frac{2k_r}{k_f\beta}+\frac{16\pa{\eps_s+M_1}}{\pa{1-\log(2)}\beta}}\int_{\O}\m(x)dx.
\end{split}
\end{equation}
	\end{itemize}

Combining \eqref{eq:estimation_on_h_e_einf_full_diffusion}, \eqref{eq:estimation_on_h_eps_c_full_diffusion}, \eqref{eq:partial_estimation_on_h_eps_s} and \eqref{eq:estimation_of_integral_of_s_on_O_full_diffusion} with the definition of $\E_{\eps_c,\eps_s,k}$ and the facts that $\einf=M_0$ when all diffusion coefficients are strictly positive and $c(x)\leq \m(x)$ we find that
\begin{equation}\nonumber 
	\begin{gathered}
			\E_{\eps_c,\eps_s,k}\pa{e,c,s} \leq \pa{1+\pa{\log\pa{1+\frac{M_1}{\eps_s}}+\frac{k\pa{2k_r+k_c}}{2k_r}}\frac{16\pa{\eps_s+M_1}}{\pa{1-\log(2)}M_0}}\int_{\O}h\pa{e(x)|\overline{e}}dx\\ 
			\pa{1+k+\log\pa{1+\frac{M_0}{\eps_c}}+\pa{\log\pa{1+\frac{M_1}{\eps_s}}+\frac{k\pa{2k_r+k_c}}{2k_r}}\pa{\frac{2k_r}{k_fM_0}+\frac{16\pa{\eps_s+M_1}}{\pa{1-\log(2)}M_0}}}
			\int_{\O}\m(x)dx\\
			+\frac{1}{\clsi}\pa{\frac{1}{d_c}+\frac{1}{d_s}\pa{1+\log\pa{1+\frac{M_1}{\eps_s}}+\frac{k\pa{2k_r+k_c}}{2k_r}}}\int_{\O}\d_{\eps_c,\eps_s}(x)dx.
	\end{gathered}
\end{equation}
when all diffusion coefficients are strictly positive, and 
\begin{equation}\nonumber 
	\begin{gathered}
		\E_{\eps_c,\eps_s,k}\pa{e,c,s} \leq \\ 
		\pa{1+k+\log\pa{1+\frac{k_r}{k_f\eps_s}}+\pa{\log\pa{1+\frac{M_1}{\eps_s}}+\frac{k\pa{2k_r+k_c}}{2k_r}}\pa{\frac{2k_r}{k_f\beta}+\frac{16\pa{\eps_s+M_1}}{\pa{1-\log(2)}\beta}}}
		\int_{\O}\m(x)dx\\
		+\frac{1}{d_s\clsi}\pa{1+\log\pa{1+\frac{M_1}{\eps_s}}+\frac{k\pa{2k_r+k_c}}{2k_r}}\int_{\O}\d_{\eps_c,\eps_s}(x)dx.
	\end{gathered}
\end{equation}
when $d_e=d_c=0$. 

Thus, \eqref{eq:validity_condition_for_full_diffusion} is satisfied when
$$\gamma \leq \frac{\pa{d_e\clsi-6\pa{\frac{\pa{k_c+k_r}}{M_0}+k_f}\max\pa{\eps_c,\eps_s}} }{\pa{1+\pa{\log\pa{1+\frac{M_1}{\eps_s}}+\frac{k\pa{2k_r+k_c}}{2k_r}}\frac{16\pa{\eps_s+M_1}}{\pa{1-\log(2)}M_0}}}$$
and
$$\gamma \leq \frac{\frac{kk_c}{2}-k_c-k_r-2k_f\eps_s}{\pa{1+k+\log\pa{1+\frac{M_0}{\eps_c}}+\pa{\log\pa{1+\frac{M_1}{\eps_s}}+\frac{k\pa{2k_r+k_c}}{2k_r}}\pa{\frac{2k_r}{k_fM_0}+\frac{16\pa{\eps_s+M_1}}{\pa{1-\log(2)}M_0}}}}$$
and
$$\gamma\leq \frac{d_c  d_s \clsi }{d_s+d_c\pa{1+\log\pa{1+\frac{M_1}{\eps_s}}+\frac{k\pa{2k_r+k_c}}{2k_r}}},$$
which yields the expression \eqref{eq:conditions_gamma_full_diffusion}, and \eqref{eq:validity_condition_for_partial_diffusion} is satisfied when
$$\gamma \leq \frac{\frac{kk_c}{2}-k_f\eps_s}{\pa{1+k+\log\pa{1+\frac{k_r}{k_f\eps_s}}+\pa{\log\pa{1+\frac{M_1}{\eps_s}}+\frac{k\pa{2k_r+k_c}}{2k_r}}\pa{\frac{2k_r}{k_f\beta}+\frac{16\pa{\eps_s+M_1}}{\pa{1-\log(2)}\beta}}}}$$
and 
$$\gamma\leq \frac{ d_s \clsi }{1+\log\pa{1+\frac{M_1}{\eps_s}}+\frac{k\pa{2k_r+k_c}}{2k_r}},$$
which yields the expression \eqref{eq:conditions_gamma_partial_diffusion}. This completes the proof.
\end{proof}
With the entropic investigation complete, we can now turn our attention to the $L^\infty$ convergence.

\section{Convergence to Equilibrium}\label{sec:unfiorm}
In this section we will explore how one can use the properties of our system of equations, \eqref{Sys}, to bootstrap the entropic convergence found in Theorem \ref{thm:entropic_convergence} to a uniform one. To do so we start with a couple of theorems that guarantee an existence of non-negative bounded solutions to our system.

%

\begin{theorem}\label{thm:existence_result_full}
	Assume that $\Omega\subset \R^n$ is a bounded, open domain with $C^{2+\zeta}, \zeta>0$ boundary $\partial\Omega$. Assume in addition that all the diffusion coefficients, $d_e, d_s, d_c, d_p$, are strictly positive. Then for any non-negative, bounded initial data $e_0$, $s_0$, $c_0$ and $p_0$, there exists a unique global non-negative, classical solution to \eqref{Sys} which is uniformly bounded in time, i.e. there exists a constant $\su>0$ such that
	\begin{equation*}
		\sup_{t\geq 0}\pa{\norm{e(t)}_{\LO{\infty}} + \norm{c(t)}_{\LO{\infty}} + \norm{s(t)}_{\LO{\infty}} + \norm{p(t)}_{\LO{\infty}}} \leq \su.
	\end{equation*}	
\end{theorem}
\begin{proof}
	The theorem follows from \cite[Theorem 1.2]{morgan2020boundedness}. Indeed, we will check that all assumptions in \cite[Theorem 1.2]{morgan2020boundedness} are satisfied. For the system \eqref{Sys}, assumptions (A1) and (A2) are immediate, assumption (A3) is fulfilled with
	$$h_e(e) = e,\quad h_s(s) = s,\quad h_c(c) = 2c,\quad h_p(p) = p.$$
	By using the lower triangle matrix
	\begin{equation*}
		A = \begin{pmatrix}
			1 & 0 & 0 & 0\\
			1 & 0 & 0 & 0\\
			1 & 0 & 1 & 0\\
			0 & 0 & 0 & 1
		\end{pmatrix},
	\end{equation*}
	assumption (A4) is valid with a linear intermediate sum condition, i.e. $r = 1$. Finally, assumption (A5) is satisfied for $\mu=2$. The fact that $r = 1$ and \cite[Remark 1.2]{morgan2020boundedness}, condition (8) in \cite[Theorem 1.2]{morgan2020boundedness} is satisfied, which guarantees the desired global result and uniform boundedness.
\end{proof}
\begin{theorem}\label{thm:existence_result_degenerate}
	Assume that $\Omega\subset \R^n$ is a bounded, open domain with $C^{2+\zeta}, \zeta>0$ boundary $\partial\Omega$. Assume in addition that $d_s>0, d_p>0$ and $d_e = d_c = 0$. Then for any non-negative, bounded initial data $e_0$, $s_0$, $c_0$ and $p_0$, there exists a unique global non-negative, strong solution to \eqref{Sys} which is uniformly bounded in time, i.e. there exists a constant $\su>0$ such that
	\begin{equation*}
	\sup_{t\geq 0}\pa{\norm{e(t)}_{\LO{\infty}} + \norm{c(t)}_{\LO{\infty}} + \norm{s(t)}_{\LO{\infty}} + \norm{p(t)}_{\LO{\infty}}} \leq \su.
	\end{equation*}	
\end{theorem}
\begin{proof}
	The proof is a fairly standard fixed point argument. As such, we defer it to Appendix \ref{secapp:additional_proofs}.
\end{proof}

With these existence theorems at hand, we can now prove our main results: Theorem \ref{thm:main} and Theorem \ref{thm2}. Before we do so, however, we shall state the following lemmas, whose proofs we leave to Appendix \ref{secapp:additional_proofs}:
\begin{lemma}\label{lem:heat_equation_regularisation}
	Assume that $\Omega\subset \R^n$, $n\geq 1$, is a bounded, open domain with $C^{2+\zeta},\zeta>0,$ boundary. Let $u(x,t)$ be a strong solution to the inhomogeneous heat equation with Neumann conditions 
	\begin{equation}\nonumber 
			\begin{cases}
			\partial_t u(x,t)-d\Delta u(x,t) =f(x,t)& x\in\O,\; t>0\\
			u(x,0)=u_0(x) & x\in\O\\
			\p_\nu u(x,t) = 0, &x\in\p\Omega, t>0.
		\end{cases} 
	\end{equation}
Then, if there exists $p>\frac{n}{2}$ such that 
\begin{equation} \label{eq:heat_conditions}
	\begin{gathered}
		\max\pa{\norm{u(t)}_{\LO{p}},\norm{f(t)}_{\LO{p}}} \leq \mathcal{C} e^{-\delta t}, \\
		\sup_{t\in[0,1]}\norm{u(t)}_{\LO{\infty}}\leq \su,
	\end{gathered}
\end{equation}
then there exist explicit constant $C_{d,n,p}$ that depends only on $\O$, $d$, $n$ and $p$ such that
\begin{equation}\label{eq:heat_L_infty_exponential_regulasiation}
	\norm{u(t)}_{\LO{\infty}} \leq e^{\delta}\max\pa{\su , \mathcal{C}C_{d,n,p}}e^{-\delta t}.
\end{equation}
\end{lemma}

\begin{lemma}\label{lem:heat_equation_avergae_convergence}
	Assume that $\Omega\subset \R^n$, $n\geq 1$, is a bounded, open domain with $C^{2+\zeta}, \zeta>0$, boundary. Let $u(x,t)$ be a strong solution to the inhomogeneous heat equation with Neumann conditions 
	\begin{equation}\nonumber
		\begin{cases}
			\partial_t u(x,t)-d\Delta u(x,t) =f(x,t)& x\in\O,\; t>0\\
			u(x,0)=u_0(x) & x\in\O\\
			\p_\nu u(x,t) = 0, &x\in\p\Omega, t>0.
		\end{cases} 
	\end{equation}
	Then for any $\eps>0$ we have that 
	\begin{equation}\label{eq:L_2_convergence}
		\begin{gathered}
			\norm{u(t)-\overline{u}(t)}_{\LO{2}}^2 \leq e^{-\frac{2d}{C_P}\pa{1-\eps}t}\norm{u_0-\overline{u_0}}_{\LO{2}}^2\\
			+\frac{C_Pe^{-\frac{2d}{C_P}\pa{1-\eps}t}}{2d\eps}\int_{0}^t e^{\frac{2d}{C_P}\pa{1-\eps} s}\norm{f(s)-\overline{f}(s)}^2_{\LO{2}}ds, 
		\end{gathered}
	\end{equation}
where $\overline{g}=\int_{\O}g(x)dx$ and $C_P$ is the Poincar\'e constant associated to the domain, i.e. the positive constant for which
\begin{equation}\label{eq:poincare}
	\norm{f-\overline{f}}_{\LO{2}}\leq C_P \norm{\na f}_{\LO{2}},
\end{equation}
for any $f\in H^1\pa{\O}$. 
\end{lemma}

\begin{proof}[Proof of Theorem \ref{thm:main}]
	From Theorem \ref{thm:existence_result_full} and Theorem \ref{thm:entropic_convergence} we know that a unique non-negative bounded classical solution to \eqref{Sys} exists and satisfies
\begin{equation}\nonumber
	\E_{\eps_c,\eps_s,k}(e(t),c(t),s(t))  \leq \E_{\eps_c,\eps_s,k}(e_0,c_0,s_0)e^{-\gamma t},
\end{equation}
for the parameters indicated in the Theorem \ref{thm:entropic_convergence}. As was seen in Remark \ref{rem:possible_choices} we can make the choices that correspond to \eqref{eq:choice_for_gamma} with \eqref{eq:possible_choices_full_diffusion}.

Using the Csisz\'ar-Kullback-Pinsker inequality 
$$\norm{f-g}_{L^1\pa{\O}} \leq \sqrt{C_{\mathrm{CPK}}\int_{\O}h\pa{f(x)|g(x)}dx}$$
where $C_{\mathrm{CPK}}$ is a fixed known constant (see for instance \cite{arnold2001convex}), together with the definition of $\E_{\eps_c,\eps_s,k}$ we find that
\begin{equation}\label{eq:L_1_convergence}
	\begin{gathered}
		\norm{c(t)}_{\LO{1}} \leq \frac{\E_{\eps_c,\eps_s,k}(e(t),c(t),s(t))}{k}\leq \frac{\E_{\eps_c,\eps_s,k}(e_0,c_0,s_0)}{k}e^{-\gamma t},\\
		\norm{s(t)}_{\LO{1}} \leq \frac{2k_r\E_{\eps_c,\eps_s,k}(e(t),c(t),s(t))}{k\pa{2k_r+k_c}}\leq \frac{2k_r\E_{\eps_c,\eps_s,k}(e_0,c_0,s_0)}{k\pa{2k_r+k_c}}e^{-\gamma t},\\
		\norm{e(t)-\einf}_{\LO{1}}\leq \sqrt{C_{\mathrm{CPK}}\int_{\O}h\pa{e(x,t)|\einf}dx}\\ \leq \sqrt{C_{\mathrm{CPK}}\E_{\eps_c,\eps_s,k}(e(t),c(t),s(t))}
		\leq \sqrt{C_{\mathrm{CPK}}\E_{\eps_c,\eps_s,k}(e_0,c_0,s_0)}e^{-\frac{\gamma t}{2}}.
	\end{gathered}
\end{equation}
Since for any $p\in [1,\infty)$
$$\norm{u}_{L^p\pa{\Omega}} \leq \norm{u}_{L^\infty\pa{\Omega}}^{1-\frac{1}{p}}\norm{u}^{\frac{1}{p}}_{L^1\pa{\Omega}}$$
we see that
\begin{equation}\label{eq:L_p_convergence}
	\begin{gathered}
		\norm{c(t)}_{\LO{p}} \leq \su^{1-\frac{1}{p}}\pa{\frac{\E_{\eps_c,\eps_s,k}(e_0,c_0,s_0)}{k}}^{\frac{1}{p}}e^{-\frac{\gamma t}{p}},\\
		\norm{s(t)}_{\LO{p}} \leq \su^{1-\frac{1}{p}}\pa{\frac{2k_r\E_{\eps_c,\eps_s,k}(e_0,c_0,s_0)}{k\pa{2k_r+k_c}}}^{\frac{1}{p}}e^{-\frac{\gamma t}{p}},\\
		\norm{e(t)-\einf}_{\LO{p}}\leq\pa{\su+M_0}^{1-\frac{1}{p}}\pa{C_{\mathrm{CPK}}\E_{\eps_c,\eps_s,k}(e_0,c_0,s_0)}^{\frac{1}{2p}}e^{-\frac{\gamma t}{2p}},
	\end{gathered}
\end{equation}
where $\su$ is given in Theorem \ref{thm:existence_result_full} and we have used the fact that $\einf=M_0$. 

Denoting by 
\begin{equation}\nonumber
	\begin{gathered}
		f_e(x,t)= -k_fe(x,t)s(x,t) + (k_r+k_c)c(x,t),\\
		f_s(x,t)= -k_fe(x,t)s(x,t) + k_rc(x,t),\\
		f_c(x,t)= k_fe(x,t)s(x,t) - (k_r+k_c)c(x,t),
	\end{gathered}
\end{equation}
we see that the first three equations of \eqref{Sys} can be rewritten as
\begin{equation}\nonumber
	\begin{cases}
		\p_t \pa{e(x,t)-\einf} - d_e\Delta \pa{e(x,t)-\einf} =f_e(x,t)  &x\in\Omega,t>0,\\
		\p_t s(x,t) - d_s\Delta s(x,t) = f_s(x,t) &x\in\Omega, t>0,\\
		\p_t c(x,t) - d_c\Delta c(x,t) = f_c(x,t) &x\in\Omega,t>0,\\
		e(x,0)-\einf = e_0(x)-\einf, \;s(x,0) = s_0(x),\; c(x,0) = c_0(x),& x\in\Omega\\
		\p_\nu e(x,t) = \p_\nu s(x,t) = \p_\nu c(x,t) = 0, &x\in\p\Omega, t>0,
	\end{cases}
\end{equation}
and since \eqref{eq:L_p_convergence} and Theorem \ref{thm:existence_result_full} imply that for any $p\in [1,\infty)$
$$\norm{e(t)s(t)}_{\LO{p}} \leq \su^{2-\frac{1}{p}}\pa{\frac{2k_r\E_{\eps_c,\eps_s,k}(e_0,c_0,s_0)}{k\pa{2k_r+k_c}}}^{\frac{1}{p}}e^{-\frac{\gamma t}{p}} $$
we see that for $p=\frac{n(1+\eta)}{2}$ for any $\eta>0$
\begin{equation}\nonumber
	\begin{gathered}
		\norm{f_e(t)}_{\LO{\frac{n}{2}\pa{1+\eta}}} \leq \pa{ k_f \su + \pa{k_r+k_c}}\su^{1-\frac{2}{n(1+\eta)}}\pa{\frac{\E_{\eps_c,\eps_s,k}(e_0,c_0,s_0)}{k}}^{\frac{2}{n\pa{1+\eta}}}e^{-\frac{2\gamma t}{n\pa{1+\eta}}}\\
		\norm{f_s(t)}_{\LO{\frac{n}{2}\pa{1+\eta}}} \leq \pa{ k_f \su + k_r}\su^{1-\frac{2}{n(1+\eta)}}\pa{\frac{\E_{\eps_c,\eps_s,k}(e_0,c_0,s_0)}{k}}^{\frac{2}{n\pa{1+\eta}}}e^{-\frac{2\gamma t}{n\pa{1+\eta}}}\\
		\norm{f_c(t)}_{\LO{\frac{n}{2}\pa{1+\eta}}} \leq \pa{ k_f \su + \pa{k_r+k_c}}\su^{1-\frac{2}{n(1+\eta)}}\pa{\frac{\E_{\eps_c,\eps_s,k}(e_0,c_0,s_0)}{k}}^{\frac{2}{n\pa{1+\eta}}}e^{-\frac{2\gamma t}{n\pa{1+\eta}}}
	\end{gathered}
\end{equation}
and as such
$$\max\pa{\norm{s(t)}_{\LO{\frac{n}{2}\pa{1+\eta}}},\norm{c(t)}_{\LO{\frac{n}{2}\pa{1+\eta}}},\norm{f_s(t)}_{\LO{\frac{n}{2}\pa{1+\eta}}},\norm{f_c(t)}_{\LO{\frac{n}{2}\pa{1+\eta}}}}$$
$$\leq \max\pa{1,\pa{k_f \su + \pa{k_r+k_c}}}\su^{1-\frac{2}{n(1+\eta)}}\pa{\frac{\E_{\eps_c,\eps_s,k}(e_0,c_0,s_0)}{k}}^{\frac{2}{n\pa{1+\eta}}}e^{-\frac{2\gamma t}{n\pa{1+\eta}}}$$
and
$$\max\pa{\norm{e(t)-\einf}_{\LO{\frac{n}{2}\pa{1+\eta}}},\norm{f_e(t)}_{\LO{\frac{n}{2}\pa{1+\eta}}}} \leq \max\pa{1,\pa{ k_f \su+ \pa{k_r+k_c}}}$$
$$\pa{\su+M_0}^{1-\frac{2}{n(1+\eta)}}\max\pa{\sqrt{\ccpk\E_{\eps_c,\eps_s,k}(e_0,c_0,s_0)},\pa{\frac{\E_{\eps_c,\eps_s,k}(e_0,c_0,s_0)}{k}}}^{\frac{2}{n\pa{1+\eta}}}e^{-\frac{\gamma t}{n\pa{1+\eta}}}$$
Applying Lemma \ref{lem:heat_equation_regularisation} we find that we can find explicit constants $\mathcal{C}_{e,\eta}$, $\mathcal{C}_{s,,\eta}$ and $\mathcal{C}_{c,,\eta}$ depending only geometric on properties, initial datum and $\eta$, that become unbounded as $\eta$ goes to zero, such that  
\begin{equation}\label{eq:exponential_decay_for_e_s_c}
	\begin{gathered}
		\norm{c(t)}_{\LO{\infty}} \leq \mathcal{C}_{c,\eta} e^{-\frac{2\gamma t}{n\pa{1+\eta}}},\\ 
		\norm{s(t)}_{\LO{\infty}} \leq\mathcal{C}_{s,\eta} e^{-\frac{2\gamma t}{n\pa{1+\eta}}},\\
		\norm{e(t)-\einf}_{\LO{\infty}} \leq \mathcal{C}_{e,\eta} e^{-\frac{\gamma t}{n\pa{1+\eta}}},\\
	\end{gathered}
\end{equation}
showing the desired result for $c(x,t)$, $s(x,t)$ and $e(x,t)$. To conclude the proof we consider the equation for $p(x,t)$
\begin{equation}\nonumber
	\begin{cases}
		\p_t p(x,t) - d_p\Delta p(x,t) =f_p(x,t)  &x\in\Omega,t>0,\\
		p(x,0) = p_0(x), & x\in\Omega\\
		\p_\nu p(x,t) = 0, &x\in\p\Omega, t>0,
	\end{cases}
\end{equation}
where $f_p(x,t)=k_cc(x,t)$. According to Lemma \ref{lem:heat_equation_avergae_convergence} we see that for any $\eps>0$
\begin{equation}\nonumber
	\begin{gathered}
		\norm{p(t)-\overline{p}(t)}_{\LO{2}}^2 \leq e^{-\frac{2d_p}{C_p}\pa{1-\eps}t}\norm{p_0-\overline{p_0}}_{\LO{2}}^2\\
		+\frac{C_P e^{-\frac{2d_p}{C_P}\pa{1-\eps}t}}{2d_p \eps}\int_{0}^t e^{\frac{2d_p}{C_P}\pa{1-\eps} s}\norm{f_p(s)-\overline{f_{p}}(s)}^2_{\LO{2}}ds.
	\end{gathered}
\end{equation}
As
$$\norm{f_p(t)}_{\LO{2}}^2 \leq \frac{k_c^2 \su \E_{\eps_c,\eps_s,k}(e_0,c_0,s_0)}{k}\cdot e^{-\gamma t}$$
and 
$$0\leq \overline{f_p}(t) =k_c\overline{c}(t)=k_c\norm{c(t)}_{\LO{1}} \leq  \frac{k_c \E_{\eps_c,\eps_s,k}(e_0,c_0,s_0)}{k}e^{-\gamma t}$$
we can find an appropriate constant $C_{d,\delta,\gamma}$ such that
$$\norm{p(t)-\overline{p}(t)}_{\LO{2}} \leq C_{d,\eps,\gamma}\pa{1+t^{\delta_{\scaleto{\frac{d_p}{C_P}\pa{1-\eps},\frac{\gamma}{2}}{10pt}\mathstrut}}}e^{-\min\pa{\frac{d_p}{C_P}\pa{1-\eps},\frac{\gamma}{2}}t},$$
where we have used the fact that
\begin{equation}\label{eq:exponential_integreal_estimation}
	e^{-\alpha t}\int_{0}^{t}e^{\pa{\alpha-\beta} s}ds =\begin{cases}
		\frac{e^{-\beta t}-e^{-\alpha t}}{\alpha-\beta} & \alpha \not=\beta,\\
	 t e^{-\beta t}& \alpha=\beta,
	\end{cases} \leq  C_{\alpha,\beta}t^{\delta_{\alpha,\beta}}e^{-\min\pa{\alpha,\beta} t}
\end{equation}
for 
\begin{equation}\nonumber
	C_{\alpha,\beta}=\begin{cases}
		\frac{1}{\abs{\alpha-\beta}} & \alpha\not=\beta \\
		1 & \alpha=\beta
	\end{cases}.
\end{equation}
We also notice that
$$\abs{p_\infty-\overline{p}(t)}=\abs{M_1-\int_{\O}p(x,t)dx} = \int_{\O}\pa{c(x,t)+s(x,t)}dx $$
$$= \norm{c(t)}_{\LO{1}}+\norm{s(t)}_{\LO{1}} \leq \frac{2\E_{\eps_c,\eps_s,k}(e_0,c_0,s_0)}{k}e^{-\gamma t},$$
from which we see that
$$\norm{p(t)-p_\infty}_{\LO{2}} \leq \widetilde{C}_{d,\eps}\pa{1+t^{\delta_{\scaleto{\frac{d_p}{C_P}\pa{1-\eps},\frac{\gamma}{2}}{10pt}\mathstrut}}}e^{-\min\pa{\frac{d_p}{C_P}\pa{1-\eps},\frac{\gamma}{2}}t}.$$
As $\norm{p(t)-p_\infty}_{\LO{\infty}} \leq \su +M_1$ according to Theorem \ref{thm:existence_result_full} and since for any $p\geq 2$
$$\norm{u}_{L^p\pa{\Omega}} \leq \norm{u}_{L^\infty\pa{\Omega}}^{1-\frac{2}{p}}\norm{u}^{\frac{2}{p}}_{L^2\pa{\Omega}}$$
we can follow the same steps as those in our investigation of $c(x,t)$, $s(x,t)$ and $e(x,t)$ to conclude that
$$	\norm{p(t)-p_\infty}_{\LO{\infty}} \leq \mathcal{C}_{p,\eta,\eps}\pa{1+ t^{\frac{4}{n\pa{1+\eta}}\delta_{\scaleto{\frac{2d_p}{C_P}\pa{1-\eps},\gamma}{10pt}\mathstrut}}}e^{-\min\pa{\frac{4d_p}{nC_P\pa{1+\eta}}\pa{1-\eps},\frac{2\gamma}{n\pa{1+\eta}}}t}$$
when $\frac{n}{2}\pa{1+\eta}\geq 2$. This completes the proof. 
\end{proof}

\begin{proof}[Proof of Theorem \ref{thm2}]
	Much like the proof of Theorem \ref{thm:main} we use theorems \ref{thm:existence_result_degenerate} and \ref{thm:entropic_convergence} to show that\footnote{Remember that the entropy $\E_{\eps_c,\eps_s,k}(e,c,s)$ is defined the same as in the full diffusion case, only with $\einf$ and $\eps_c$ being functions that satisfy \eqref{eq:eps_conection_partial_diffusion}.}
	\begin{equation}\label{eq:L_1_convergence_partial_diffusion}
		\begin{gathered}
			\norm{c(t)}_{\LO{1}} \leq \frac{\E_{\eps_c,\eps_s,k}(e_0,c_0,s_0)}{k}e^{-\gamma t},\\
			\norm{s(t)}_{\LO{1}} \leq  \frac{2k_r\E_{\eps_c,\eps_s,k}(e_0,c_0,s_0)}{k\pa{2k_r+k_c}}e^{-\gamma t},\\
			\norm{e(x,t)-\einf(x)}_{\LO{1}}=\norm{c(t)}_{\LO{1}} \leq \frac{\E_{\eps_c,\eps_s,k}(e_0,c_0,s_0)}{k}e^{-\gamma t},
		\end{gathered}
	\end{equation}
	with $\gamma$ satisfying \eqref{eq:conditions_gamma_partial_diffusion}. Note that to attain the last inequality we have used the conservation law \eqref{law3}
	$$\einf(x)=e(x,t)+c(x,t)=e_0(x)+c_0(x).$$
	Again, using Remark \ref{rem:possible_choices}, we see that we can make the choices that correspond to \eqref{eq:choice_for_lambda} with \eqref{eq:possible_choices_partial_diffusion}.\\
	Continuing as in the proof of Theorem \ref{thm:main} we see that for any $p\geq 1$ we have that
	\begin{equation}\label{eq:L_p_convergence_partial_diffusion}
		\begin{gathered}
			\norm{c(t)}_{\LO{p}} \leq \su^{1-\frac{1}{p}}\pa{\frac{\E_{\eps_c,\eps_s,k}(e_0,c_0,s_0)}{k}}^{\frac{1}{p}}e^{-\frac{\gamma t}{p}},\\
			\norm{s(t)}_{\LO{p}} \leq \su^{1-\frac{1}{p}}\pa{\frac{2k_r\E_{\eps_c,\eps_s,k}(e_0,c_0,s_0)}{k\pa{2k_r+k_c}}}^{\frac{1}{p}}e^{-\frac{\gamma t}{p}},\\
			\norm{e(t)-\einf}_{\LO{p}}=\norm{c(t)}_{\LO{p}} \leq\su^{1-\frac{1}{p}}\pa{\frac{\E_{\eps_c,\eps_s,k}(e_0,c_0,s_0)}{k}}^{\frac{1}{p}}e^{-\frac{\gamma t}{p}}.
		\end{gathered}
	\end{equation}
    Since $s$ satisfies
	$$\p_t s(x,t) - d_s\Delta s(x,t) = \underbrace{-k_fe(x,t)s(x,t) + k_rc(x,t)}_{f_s(x,t)}$$
	and
	$$\norm{f_s(t)}_{L^p\pa{\O}}\leq \mathcal{C}_{s}e^{-\frac{\gamma t}{p}}$$
	for an appropriate constant, we get from Lemma \ref{lem:heat_equation_regularisation} that for any $\eta>0$ there exists a constant $\mathcal{C}_{s,\eta}$ that blows up as $\eta$ goes to zero such that
	$$\norm{s(t)}_{\LO{\infty}} \leq\mathcal{C}_{s,\eta} e^{-\frac{2\gamma t}{n\pa{1+\eta}}}.$$
	Next we turn our attention to the convergence of $c$. As $c$ satisfies the equation
	$$\p_t c(x,t)=k_fe(x,t)s(x,t) - (k_r+k_c)c(x,t)$$
	we have that
	\begin{equation*}
		c(x,t) = e^{-(k_r+k_c)t}c_0(x) + k_f\int_0^{t}e^{-(k_r+k_c)(t-\xi)}e(x,\xi)s(x,\xi)d\xi.
	\end{equation*}
	Since $c$ is also non-negative we find that 
	\begin{align*}
		\|c(t)\|_{\LO{\infty}} &\leq e^{-(k_r+k_c)t}\pa{\|c_0\|_{\LO{\infty}} + k_f\int_0^t e^{(k_r+k_c)\xi}\|e(\xi)\|_{\LO{\infty}}\|s(\xi)\|_{\LO{\infty}}d\xi}\\
		&\leq e^{-(k_r+k_c)t}\pa{\|c_0\|_{\LO{\infty}} + \mathcal{C}_{s,\eta}\su\int_0^t e^{\pa{k_r+k_c-\frac{2\gamma}{n\pa{1+\eta}}}\xi}d\xi}\\
		 \leq &\mathcal{C}_{c,s,\eta}\pa{1+t^{\delta_{\scaleto{k_r+k_c,\frac{2\gamma}{n\pa{1+\eta}}}{10pt}\mathstrut}}}e^{-\min\pa{k_r+k_c,\frac{2\gamma}{n\pa{1+\eta}}}t},
	\end{align*}
	where we have used \eqref{eq:exponential_integreal_estimation}. Using the conservation law \eqref{law3} again we get that
	$$\norm{e(x,t)-\einf(x)}_{L^\infty\pa{\O}}=\norm{c(t)}_{\LO{\infty}}\leq \mathcal{C}_{c,s,\eta}\pa{1+t^{\delta_{\scaleto{k_r+k_c,\frac{2\gamma}{n\pa{1+\eta}}}{10pt}\mathstrut}}}e^{-\min\pa{k_r+k_c,\frac{2\gamma}{n\pa{1+\eta}}}t}.$$
	The proof of the rate of convergence of $p(x,t)-\pinf$ to zero is identical to that presented in the proof of Theorem \ref{thm:main}, and as such we conclude the proof of the theorem.
\end{proof}
With our main investigation complete, we now turn our attention to a few final remarks.

\section{Final Remarks}\label{sec:final}

While some of the calculations presented in this work are quite technical, the true heart of proofs - the definition of a ``cut-off'' entropy functional and the study of the interplay between it and a decreasing mass term - is simple and powerful enough that we believe it could be widely used in many other open and irreversible systems. We would like to end this study with a few remarks/observations.

\subsection{The functional inequality}\label{subsec:functional_inequality} As was mentioned in our introduction \S\ref{subsec:main}, showing the decay of our new entropy functional, $\E_{\eps_c,\eps_{s},k}$, heavily relied on a functional inequality of the form
\begin{equation*}
	\E_{\eps_c,\eps_s,k}(e,s,c) \lesssim \begin{cases}
		 \int_{\O}\big(\d_{\eps_c,\eps_s}(x) +h(e(x)|\overline{e}) + \m(x)\big)dx & d_e,d_s,d_c>0,\\
		 \int_{\O}\big(\d_{\eps_c,\eps_s}(x)  + \m(x)\big)dx & d_e=d_c=0
	\end{cases}.
\end{equation*} 
While this has not been stated explicitly as a lemma, proposition or a theorem, this inequality is sole ingredient of the proof of Theorem \ref{thm:entropic_convergence}, and its proof can be found there.

\subsection{The case where $d_c=0$ and $d_e>0$}\label{subsec:another_case}
One can apply our techniques and find explicit exponential convergence to equilibrium for the $L^\infty$ norms when all diffusion coefficients but $d_c$ are strictly positive. In this case the equilibrium will be the same as that for when all diffusion coefficients were strictly positive. This situation, however, is not chemically relevant and as such we have elected to not treat it.
\subsection{Optimal rate of convergence in the case where all diffusion coefficients are strictly positive}\label{subsec:optimal_rate}
It is clear that the explicit rate of convergence given in Remark \ref{rem:explicit_gamma} is not optimal. This stems from the multiple estimations we have made to achieve our results, estimations that are extremely hard to optimise simultaneously. Nevertheless, since we have shown exponential convergence to equilibrium we know that eventually (which can be expressed explicitly) the solution will be in a small neighbourhood of the equilibrium. This allows us to consider the linearised version of our equations and attain the optimal long time behaviour of the solutions, at least when this linear system indeed approximates the full set of equations with respect to this behaviour.

Denoting by $\wt{y} = y - y_\infty$ for $y$ which can be $e$, $s$, $c$ or $p$, we find that the linearised system of equations around the equilibrium $(\einf, \sinf, \cinf, \pinf)$ is
\begin{equation}\label{linearised_sys}
\begin{cases}
	\p_t \wt{e}(x,t) - d_e\Delta \wt{e}(x,t) = -k_f\einf \wt{s}(x,t) + (k_r+k_c)\wt{c}(x,t), &x\in\Omega, t>0,\\
	\p_t \wt{s}(x,t) - d_s\Delta \wt{s}(x,t) = - k_f\einf \wt{s}(x,t) + k_r\wt{c}(x,t), &x\in \Omega, t>0,\\
	\p_t\wt{c}(x,t) - d_c\Delta \wt{c}(x,t) = k_f\einf \wt{s}(x,t) - (k_r+k_c)\wt{c}(x,t), &x\in\Omega, t>0,\\
	\p_t\wt{p}(x,t) - d_p\Delta \wt{p}(x,t) = k_c\wt{c}(x,t), &x\in\Omega, t>0,
\end{cases}
\end{equation}
with initial data $\wt{y}(x,0) = y_0(x) - y_\infty$ and homogeneous Neumann boundary conditions. 

As before, we notice that the equation of $\wt{p}$ is decoupled from the rest of the equations.\\
Denoting by 
$$0=\lam_0 < \lam_1 < \lam_2 \leq \lam_3 \leq \cdots \to \infty$$
 the eigenvalues of $-\Delta$ with homogeneous Neumann boundary condition in $\O$, and by $\{\omega_j\}_{j\in \N\cup \br{0}}$ the corresponding eigenfunctions basis of $\LO{2}$ (see for instance \cite[\S 5.7]{taylor1996partial}), we claim the following
\begin{proposition}
	The solution to the system \eqref{linearised_sys} decays to zero in $\LO{\infty}$-norm with the optimal rates
	\begin{equation}\label{eq:optimal_convergence}
		\begin{gathered}
		\norm{\wt{c}(t)}_{\LO{\infty}}+\norm{\wt{s}(t)}_{\LO{\infty}}\leq \mathcal{C}_{s,c}e^{-\mu_{\mathrm{opt}} t},\\
		\norm{\wt{e}(t)}_{\LO{\infty}}\leq \mathcal{C}_{e,s,c}\pa{1+t^{\delta_{d_e \lambda_1,\mu_{\mathrm{opt}}}}}e^{-\min\pa{d_e\lambda_1,\mu_{\mathrm{opt}} }t},\\
		\norm{\wt{p}(t)}_{\LO{\infty}}\leq \mathcal{C}_{p,s,c}\pa{1+t^{\delta_{d_p \lambda_1,\mu_{\mathrm{opt}}}}}e^{-\min\pa{d_p\lambda_1,\mu_{\mathrm{opt}} }t},\\
		\end{gathered}
	\end{equation}
	where\footnote{Indeed
	$$\mu_{\mathrm{opt}}=\frac{\pa{k_f\einf +k_r+k_c }^2-\pa{k_f\einf - k_r-k_c}^2 - 4k_rk_f\einf}{2\pa{k_f\einf +k_r+k_c +\sqrt{\pa{k_f\einf - k_r-k_c}^2 + 4k_rk_f\einf}}}$$
	$$=\frac{4k_f \einf \pa{k_r+k_c}- 4k_rk_f\einf}{2\pa{k_f\einf +k_r+k_c +\sqrt{\pa{k_f\einf - k_r-k_c}^2 + 4k_rk_f\einf}}}>0$$}
	\begin{equation}\label{eq:optimal_rate}
		\mu_{\mathrm{opt}}= \frac 12\pa{k_f\einf +k_r+k_c - \sqrt{\pa{k_f\einf - k_r-k_c}^2 + 4k_rk_f\einf}}>0,
	\end{equation}
and 
$$\delta_{x,y}=\begin{cases}
1 & x=y \\
0 & x\not=y	
\end{cases}.$$
\end{proposition}
\begin{remark}
	One notices from \eqref{eq:optimal_convergence} and \eqref{eq:optimal_rate} that the optimal decay rates do not depend on the diffusion rate of $s$ or $c$.
\end{remark}
\begin{proof}
We give a more formal proof to this proposition, and will not concern ourselves with discussing the existence and uniqueness of solutions, or other technical issues. Writing
\begin{equation}\label{eq:orthogonal}
	\wt{y}(x,t) = \sum_{j=0}^{\infty}\wt{y}_j(t)\omega_j(x),
\end{equation}
for $y$ equals $s$ or $c$, we see that the second and third equations of  \eqref{linearised_sys} (which are decoupled from the rest) are equivalent to the infinite system of ODEs
\begin{equation*}
		\frac{d}{dt}\wt{X}_j(t) = A_j\wt{X}_j (t),\qquad j\in\N\cup\br{0},
\end{equation*}
 where
\begin{equation*}
	\wt{X}_j = \begin{pmatrix}
		\wt{s}_j \\ \wt{c}_j 
	\end{pmatrix},\qquad A_j = \begin{pmatrix}
		 -d_s\lam_j - k_f\einf & k_r\\
		 k_f\einf & -d_c\lam_j - (k_r+k_c)
				\end{pmatrix}.
\end{equation*}
The eigenvalues of $A_j$ are the solutions to the quadratic equation
\begin{equation}\label{quadratic}
	\tau^2 + \big[(d_s+d_c)\lam_j + k_f\einf + k_r+k_c\big]\tau + (d_s\lam_j + k_f\einf)(d_c\lam_j + k_r+k_c) - k_rk_f\einf = 0.
\end{equation}
and as such, the maximal eigenvalue, which determine the long time behaviour of the solution, is given by
$$\tau_{\mathrm{max},j}=-\frac{\pa{d_s+d_c}\lam_j +k_f\einf+k_r+k_c  -\sqrt{\triangle_j}}{2}$$
where 
\begin{equation*}
	\triangle_j  = \pa{d_s\lam_j + k_f\einf - (d_c\lam_j + k_r+k_c) }^2 + 4k_rk_f\einf > 0.
\end{equation*}
Since for any $\alpha, \beta \in \R$ and $\gamma>0$ we have that
$$\abs{\frac{d}{dx}\sqrt{\pa{\alpha x +\beta}^2+\gamma }}=\frac{\abs{\alpha}\abs{\pa{\alpha x +\beta}}}{\sqrt{\pa{\alpha x +\beta}^2+\gamma }}\leq \abs{\alpha}$$
we see that for any $\delta> \abs{\alpha}$
$$\frac{d}{dx}\pa{\delta x - \sqrt{\pa{\alpha x +\beta}^2+\gamma }} \geq  \delta -\abs{\alpha}>0.$$
Choosing $\delta=d_s+d_c$, $\alpha= d_s-d_c$, $\beta=k_f\einf-k_r-k_c$ and $\gamma=4k_rk_f\einf$ in the above we conclude
$$\frac{d}{d\lambda_{j}}\tau_{\mathrm{max},j}<0$$
and since $\br{\lambda_j}_{j\in\N\cup\br{0}}$ is an increasing sequence we find that
$$\sup_{j\in\N}\tau_{\mathrm{max},j}=\tau_{\mathrm{max},0}=\frac{-\pa{k_f\einf+k_r+k_c}  +\sqrt{\pa{ k_f\einf -  k_r-k_c}^2 + 4k_rk_f\einf }}{2}=-\mu_{\mathrm{opt}}.$$
This implies that $\wt{s}$ and $\wt{c}$ decay with an exponential rate of $\mu_{\mathrm{opt}}$, which is optimal.\\
Next we turn our attention to $\wt{e}$ and $\wt{p}$. Using the same orthogonal decomposition \eqref{eq:orthogonal} we find the following infinite set of ODEs:
\begin{equation}\label{eq:eq_for_tilde_e_and_p}
	\begin{gathered}
		\frac{d}{dt}\wt{e}_j(t)=-d_e \lambda_j\wt{e}_j(t) -k_f\einf \wt{s}_j(t) + (k_r+k_c)\wt{c}_j(t),\\
		\frac{d}{dt}\wt{p}_j(t)=-d_p \lambda_j\wt{p}_j(t) +k_c\wt{c}_j(t).
	\end{gathered}
\end{equation}
We will focus our attention on showing the convergence rate for $\wt{e}$. The convergence of $\wt{p}$ will be achieved in an identical way (by replacing $d_e$ with $d_p$). Equation \eqref{eq:eq_for_tilde_e_and_p} implies that 
\begin{equation}\label{eq:solution_for_e_linearised}
	\wt{e}_j(t) = e^{-d_e \lambda_j t}\wt{e}_j(0) + \int_{0}^{t}e^{-d_e\lambda_j (t-\xi)}\pa{-k_f\einf \wt{s}_j(\xi) + (k_r+k_c)\wt{c}_j(\xi)}d\xi.
\end{equation}
Thus, using the known optimal decay rate for $\wt{c}_j$ and $\wt{s}_j$ we see that
$$\abs{\wt{e}_j(t)} \leq e^{-d_e \lambda_j t}\abs{\wt{e}_j(0)} + \mathcal{C}_{s_j,c_j}\int_{0}^{t}e^{-d_e\lambda_j (t-\xi)}e^{-\mu_{\mathrm{opt}} \xi}d\xi,$$
where $\mathcal{C}_{s_j,c_j}$ is a constant that depends only on 
$$\wt{s}_j(0) = \inner{s_0,\omega_j},\quad \wt{c}_j(0) = \inner{c_0,\omega_j}$$
and $k_f$, $k_r$, $k_c$ and $\einf$. From this and \eqref{eq:exponential_integreal_estimation} we conclude that
\begin{equation}\label{eq:component_estimation_on_e_linearised}
	\begin{gathered}
		\abs{\wt{e}_j(t)} 
	\leq \pa{\abs{\wt{e}_j(0)}+\widetilde{\mathcal{C}}_{s_j,c_j}}\pa{1+t^{\delta_{d_e\lambda_j,\mu_{\mathrm{opt}}}}}e^{-\min\pa{d_e \lambda_j,\mu_{\mathrm{opt}}}t}.
\end{gathered}
\end{equation}
Since $\br{\lambda_j}_{j\in\N}$ is an increasing sequence of numbers, we find that for any $j\geq 1$
\begin{equation}\label{eq:rate_of_convergence_e_j_and_j_bigger_0}
	\abs{\wt{e}_j(t)} \leq \pa{\abs{\wt{e}_j(0)}+\widetilde{\mathcal{C}}_{s_j,c_j}}\pa{1+t^{\delta_{d_e\lambda_1,\mu_{\mathrm{opt}}}}}e^{-\min\pa{d_e \lambda_1,\mu_{\mathrm{opt}}}t}.
\end{equation}
The above approach, however, is not useful when $j=0$ as in this case $\lambda_j=0$ and \eqref{eq:component_estimation_on_e_linearised} yields only an upper bound. Instead we use the simple conservation law (much like in the full equation)
$$ \int_{\O}\pa{\wt{e}(x,t)+\wt{c}(x,t)}dx=\int_{\O}\pa{e(x,t)+c(x,t)}dx-M_0=0$$
and the fact that since $\omega_j(x)\equiv 1$ we have that
$$\wt{f}_0 = \inner{f,1}=\int_{\O}f(x)dx,$$
to conclude that
\begin{equation}\label{eq:rate_of_convergence_e_j_and_j_=_0}
	\abs{\wt{e}_0(t)}=\abs{\wt{c}_0(t)}\leq \mathcal{C}_{c}e^{-\mu_{\mathrm{opt}}t}\leq \mathcal{C}_c\pa{1+t^{\delta_{d_e\lambda_1,\mu_{\mathrm{opt}}}}}e^{-\min\pa{d_e \lambda_1,\mu_{\mathrm{opt}}}t}.
\end{equation}
Combining \eqref{eq:rate_of_convergence_e_j_and_j_bigger_0} and \eqref{eq:rate_of_convergence_e_j_and_j_=_0} gives us the desired $L^\infty$ bound on $\wt{e}(t)$. As was mentioned before, the treatment of $\wt{p}$ is exactly the same with the use of the second conservation law
$$\wt{s}_0(t)+\wt{c}_0(t)+\wt{p}_0(t) = \int_{\O}\pa{s(x,t)+c(x,t)+p(x,t)}dx=\int_{\O}\pa{e(x,t)+c(x,t)}dx-M_1=0.$$
The proof is thus complete.
\end{proof} 

\subsection{Convergence to equilibrium without the lower bound condition on $e_0+c_0$}\label{subsec:convergence_without_lower_bound}
As was mentioned in Remark \ref{remark:convergence_without_lower_bound}, and is clearer now from the proof of Theorem \ref{thm:entropic_convergence}, the lower bound \eqref{lower_bounded_e0+c0} is essential to show and quantitatively estimate the convergence to equilibrium of the system \eqref{Sys} in the case where $d_e = d_c = 0$. Intuitively, however, we can still expect a strong convergence to equilibrium in situations where \eqref{lower_bounded_e0+c0} is not fulfilled. Denoting the set
$$\O_{\text{zero}} = \{ x\in\O\,:\, e_0(x) + c_0(x) = 0\}=\{ x\in\O\,:\, e_0(x) = c_0(x) = 0\}$$
we see that as the evolution of the concentration $s$ is dominated by diffusion on $\O_{\text{zero}}$ at short times, $s$ will diffuse away to $\O\setminus \O_{\text{zero}}$ where it will get converted into product and complex and start a chain reaction that will lead to an eventual convergence to equilibrium. Proving this intuition rigorously, however, remains an interesting open problem. It is worth mentioning that the tools we've developed in this work (mainly Theorem \ref{thm:first_entropy_inequality}) are sufficient to show \textit{qualitative} convergence to equilibrium of the entropy even in this ``degenerate'' case.\\
\begin{figure}[ht]
     \centering
     \begin{subfigure}[b]{0.48\textwidth}
         \centering
         \includegraphics[width=\textwidth]{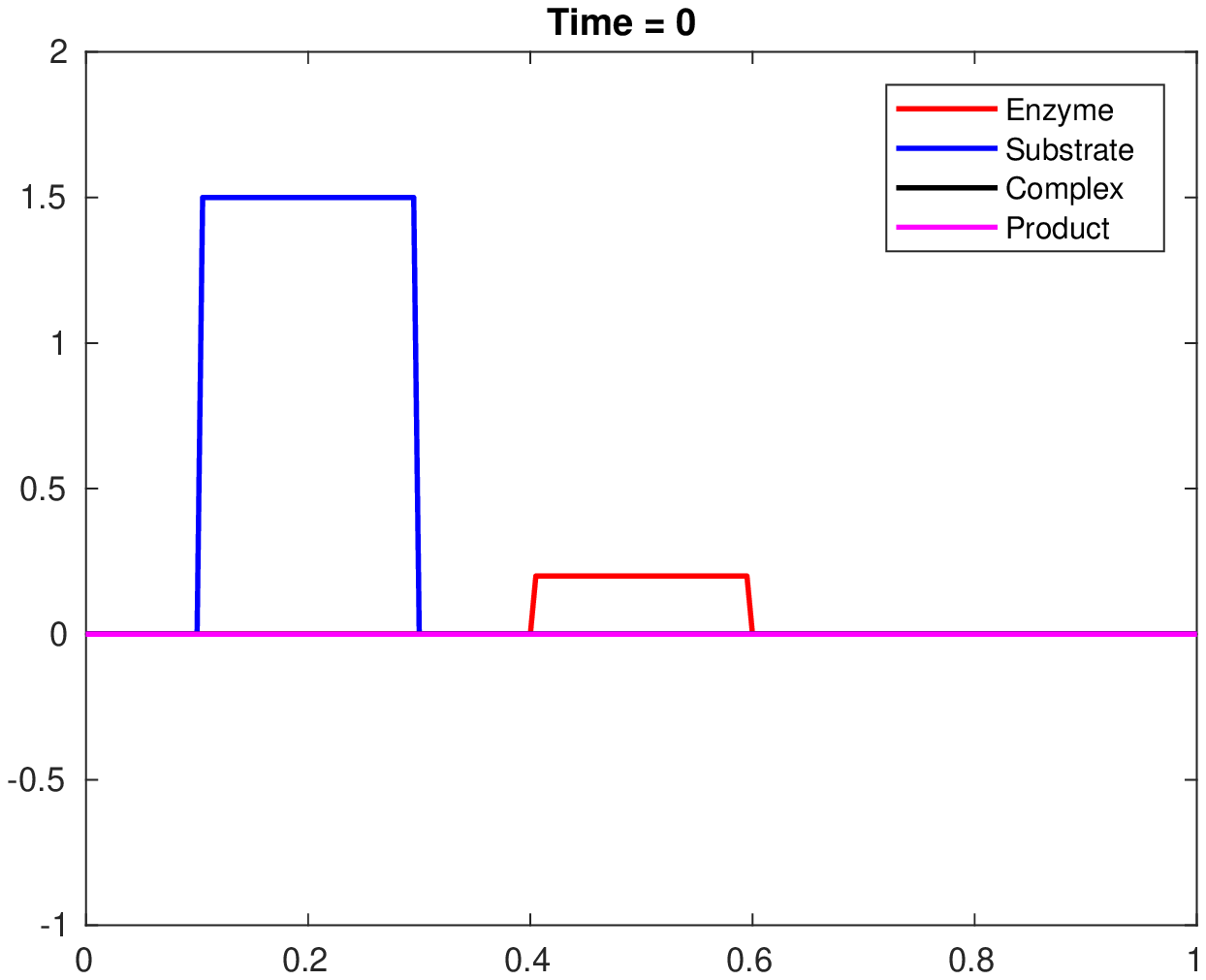}
         \caption{Initial concentrations}
     \end{subfigure}
     \hfill
     \begin{subfigure}[b]{0.48\textwidth}
         \centering
         \includegraphics[width=\textwidth]{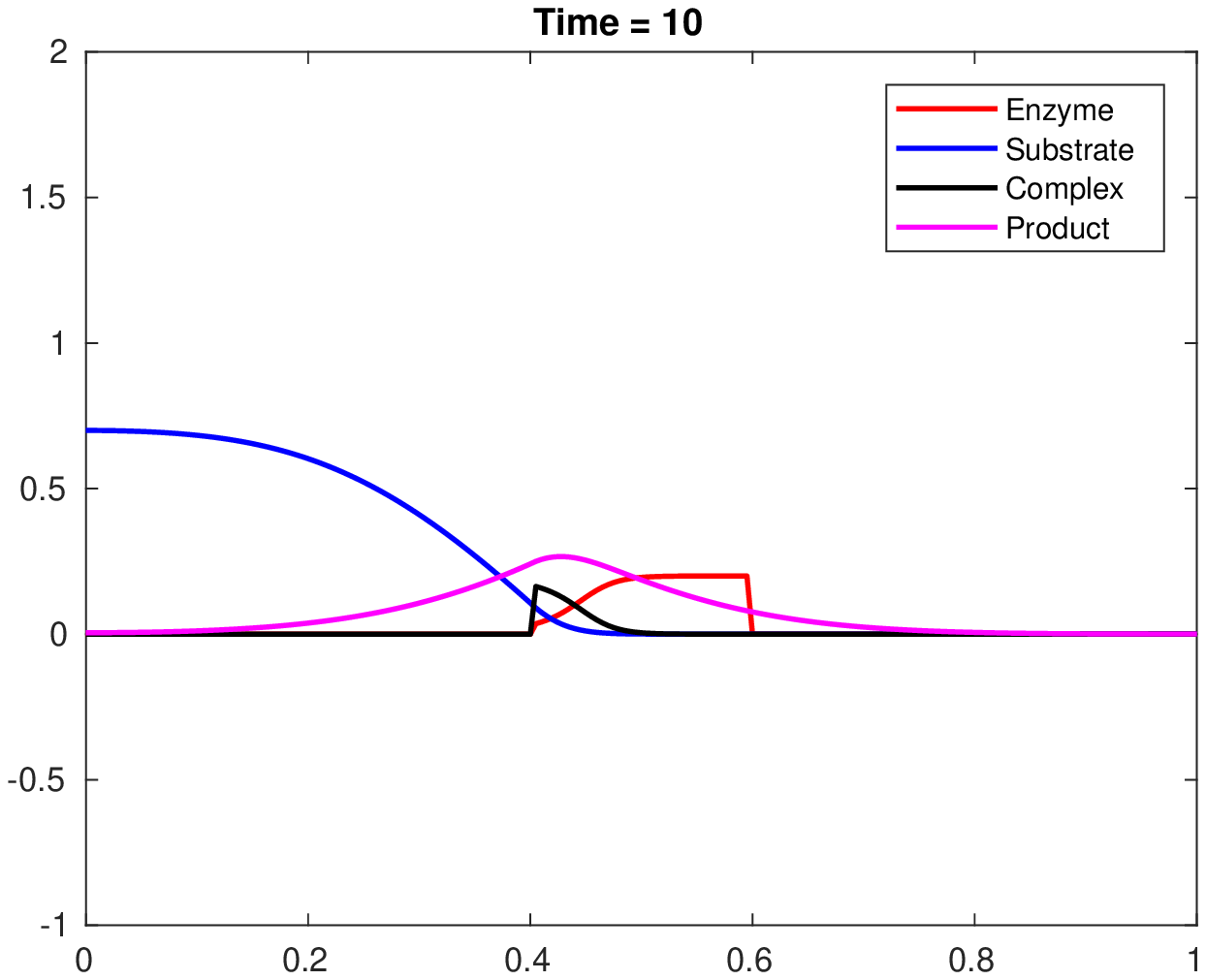}
         \caption{Concentrations at time $t=10$.}
     \end{subfigure}
     \linebreak     
     \hfill
     \begin{subfigure}[b]{0.48\textwidth}
         \centering
         \includegraphics[width=\textwidth]{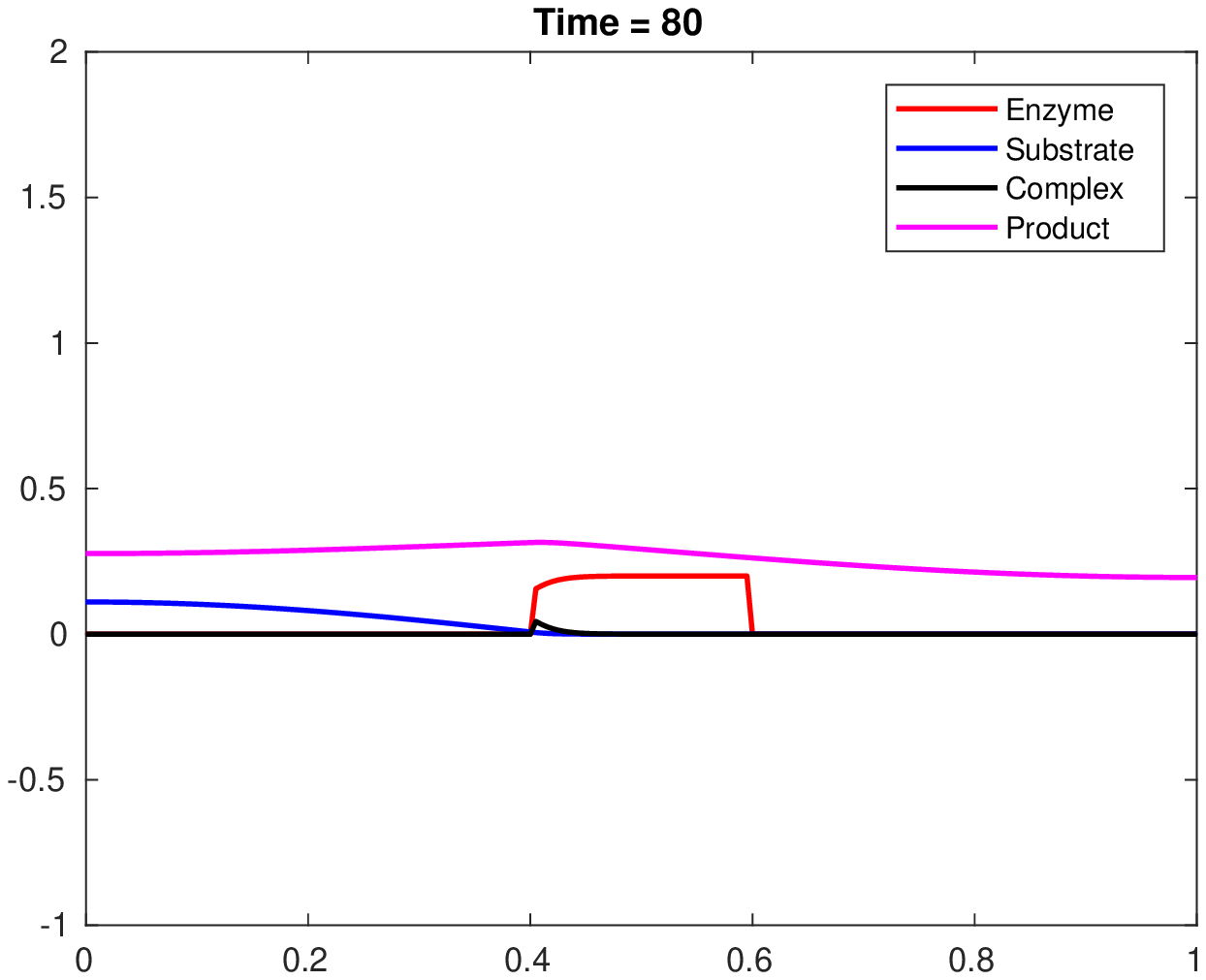}
         \caption{Concentrations  at time $t=80$.}
     \end{subfigure}
     \hfill
      \begin{subfigure}[b]{0.48\textwidth}
              \centering
             \includegraphics[width=\textwidth]{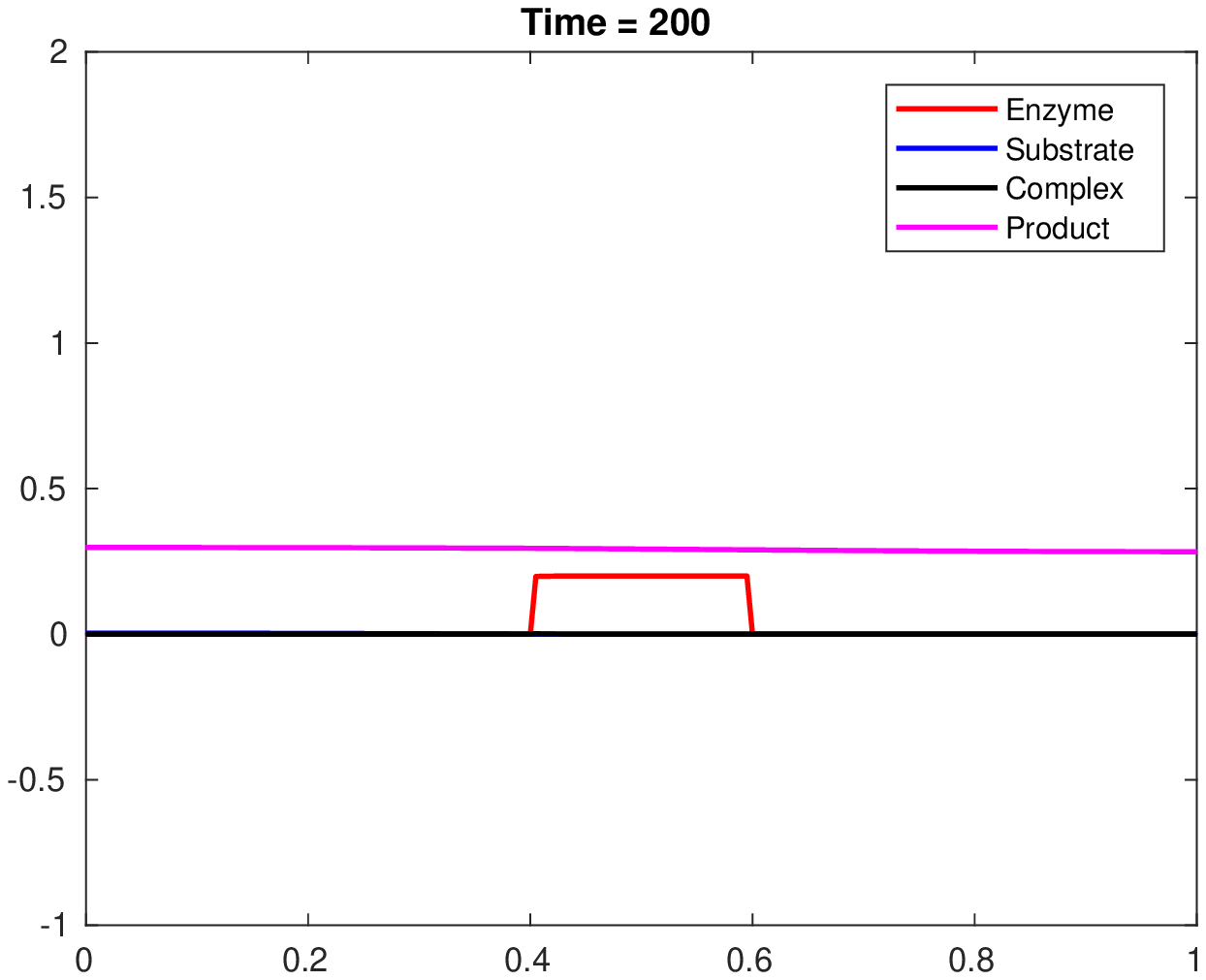}
              \caption{Concentrations at time $t=200$.}
          \end{subfigure}
        \caption{The evolution of enzyme, complex and substrate in the case where \eqref{lower_bounded_e0+c0} is not fulfilled. }
        \label{fig1}
\end{figure}
We end the main body of our work with a numerical simulation, which can be seen in Fig. \ref{fig1}, that depicts the case where $c_0 \equiv 0$, $e_0$ is not bounded away from zero, and $\mathrm{supp}e_0 \cap \mathrm{supp}s_0 = \emptyset$. More precisely, we consider $$\Omega = (0,1),\; k_f = 100,\; k_r = k_c = 1,\; d_e = d_c = 0, \;d_s = d_p= 0.02,$$
$$e_0(x) = 0.2\chi_{(0.4,0.6)}(x), \quad s_0(x) = 1.5\chi_{(0.1,0.3)},  \quad \text{ and } \quad c_0(x)=p_0(x) =0.$$
Under these assumptions we see that the equilibrium is given by
\begin{equation*}
\einf(x) = 0.2\chi_{(0.4,0.6)}(x), \quad \sinf = \cinf = 0, \quad \pinf = 0.3.
\end{equation*}

As expected, we see in Figure \ref{fig1}(B) that when the substrate $s$ diffuses to the region where the enzyme concentration is non-zero, it gets converted into the complex, which subsequently produces the product. This procedure continues to dissolve the substrate, as seen in Figure \ref{fig1}(C), and eventually converts it completely to the product, while the enzyme returns to its initial configuration in Figure \ref{fig1}(D). 
\subsection*{Acknowledgment} The authors would like to thank Phillippe Lauren\c{c}ot for pointing out a mistake in the proof of Lemma \ref{lem:heat_equation_regularisation}.
\appendix
\section{Additional Proofs}\label{secapp:additional_proofs}
In this appendix we will show the proofs of several results which we elected to defer in order to not disrupt the flow of the presented work.
\begin{lemma}\label{lem:integration_by_parts_of_absolute}
	Let $\O$ be a bounded domain with $C^1$ boundary and assume that $f\in   W^{1,p}\pa{\O}$ for some $p\in [1,\infty)$ and $g\in  W^{2,p^\prime}\pa{\O}$ where $p$ and $p^\prime$ are H\"older conjugated. Then if $\na g(x)\cdot \mathcal{n}(x)=0$, where $\mathcal{n}(x)$ is the outer normal to $\O$ at the point $x\in \partial \O$, we have that 
	\begin{equation}\label{eq:integration_by_parts_generalised}
		\int_{\O}f(x)\Delta g(x)dx = -\int_{\O}\nabla f(x) \cdot \na g(x) dx.
	\end{equation}
	Moreover, for any $c\in \R$ we have that $\max\pa{f(x),c}\in W^{1,p}\pa{\O}$ with
	\begin{equation}\label{eq:gradient_of_max}
		\nabla \max\pa{f(x),c} 
		=\begin{cases}
			\na f(x)	& f(x)>c  \\  0 & f(x)\leq c
		\end{cases}.
	\end{equation} 
	and
	\begin{equation}\label{eq:integration_by_parts_of_max}
		\int_{\O}\max\pa{f(x),c}\Delta g(x)dx = -\int_{\br{x\;|\; f(x)\geq c}}\nabla f(x) \cdot \na g(x) dx=-\int_{\br{x\;|\; f(x)>c}}\nabla f(x) \cdot \na g(x) dx.
	\end{equation}
\end{lemma}
\begin{remark}\label{rem:outward_gradient_is_defined}
	The normal derivative of $g$, $\na g(x)\cdot \mathcal{n}(x)$ on $\partial \O$ is indeed well defined. This follows from a standard Trace theorem that states that when $U$ is a bounded open set with $C^1$ boundary there exists a linear operator
	$$T:W^{1,p}\pa{U}\to L^p\pa{\partial U}$$
	such that
	$$T u=u\vert_{\partial U},\qquad \text{ when }u\in W^{1,p}\pa{U}\cap C\pa{\overline{U}}$$
	and
	$$\norm{Tu}_{L^p\pa{U}}\leq C\norm{u}_{W^{1,p}\pa{U}},$$
	for some fixed geometric constant $C>0$. A proof can be found in \cite[Theorem 1, pp. 272]{evans2010partial}.\\
	The expression $\na g(x)\cdot \mathcal{n}(x)$ on $\partial \O$ is to be understood as $T\pa{\na g(x)}\cdot \mathcal{n}(x)$.
\end{remark}
\begin{proof}
 We start by noticing that for any $f,g\in C^\infty\pa{\O}$ we have that
 \begin{equation}\label{eq:integration_by_parts}
 	\int_{\O}f(x)\Delta g(x)dx = \int_{\partial \O}f(x)\na g(x)\cdot \mathcal{n}(x)dx -\int_{\O}\nabla f(x) \cdot \na g(x) dx,
 \end{equation}
by a simple integration by parts. Since any function in $W^{m,p}\pa{\O}$ can be approximated by a $C^\infty\pa{\overline{\O}}$ sequence in the $W^{m,p}\pa{\O}$ norm (see \cite{evans2010partial}, for instance) 
we can find a sequence of $C^\infty\pa{\overline{\O}}$ functions, $\br{f_n}_{n\in\N}$ and $\br{g_n}_{n\in\N}$, such that
$$\norm{f_n-f}_{W^{1,p}\pa{\O}}\underset{n\to\infty}{\longrightarrow}0,\qquad \norm{g_n-g}_{W^{2,p^{\prime}}\pa{\O}}\underset{n\to\infty}{\longrightarrow}0,$$
from which we immediately deduce that
\begin{equation}\label{eq:limit_of_approxiation_for_integrals}
	\begin{gathered}
		\lim_{n\to\infty}\int_{\O}f_n(x)\Delta g_n(x)dx=\int_{\O}f(x)\Delta g(x)dx,\\
		\lim_{n\to\infty}\int_{\O}\nabla f_n(x) \cdot \na g_n(x) dx=\int_{\O}\nabla f(x) \cdot \na g(x) dx.
	\end{gathered}	
\end{equation}
Using the standard Trace theorem, mentioned in Remark \ref{rem:outward_gradient_is_defined}, we see that the traces of $f$ and $\na g$, $T\pa{f}$ and $T\pa{\na g}$ (to be thought of as a vector whose components are $T$ of the partial derivatives of $g$), are well defined on $\partial \O$ and
$$\norm{f_n-T(f)}_{L^p\pa{\partial \O}}=\norm{T\pa{f_n-f}}_{L^p\pa{\partial\O}}\leq C\norm{f_n-f}_{W^{1,p}\pa{\O}}\underset{n\to\infty}{\longrightarrow}0$$
and
$$\norm{\na g_n-T(\na g)}_{L^{p^\prime}\pa{\partial\O}}=\norm{T\pa{\na g_n-\na g}}_{L^p\pa{\partial\O}}\leq C\norm{g_n-g}_{W^{2,p^{\prime}}\pa{\O}}\underset{n\to\infty}{\longrightarrow}0.$$
As such, using the Neumann boundary condition, we find that
\begin{equation}\label{eq:first_approxiamtion_of_lemma_special}
	\begin{gathered}
		\norm{\na g_n(x)\cdot \mathcal{n}(x)}_{L^{p^\prime}\pa{\partial \O}}=\norm{\pa{T\pa{\na g_n(x)}-T\pa{\na g(x)}}\cdot \mathcal{n}(x)}_{L^{p^\prime}\pa{\partial \O}}\\
		\leq \norm{\mathcal{n}}_{L^\infty\pa{\partial \O}}\norm{T\pa{\na g_n(x)-\na g(x)}}_{L^{p^\prime}\pa{\partial \O}}\underset{n\to\infty}{\longrightarrow}0
	\end{gathered}
\end{equation}
which implies that
\begin{equation}\label{eq:second_approxiamtion_of_lemma_special}
	\lim_{n\to\infty}\int_{\partial \O}f_n(x)\na g_n(x)\cdot \mathcal{n}(x)dx =0.
\end{equation}
Combining \eqref{eq:integration_by_parts}, \eqref{eq:limit_of_approxiation_for_integrals} and \eqref{eq:second_approxiamtion_of_lemma_special} gives us \eqref{eq:integration_by_parts_generalised}.

The facts that $\max\pa{f(x),c}$ belongs to $W^{1,p}\pa{\O}$ and that its derivative is given by \eqref{eq:gradient_of_max} follow from Corollary 6.18 and the discussion that follows in \cite{LiebLoss}. Combining the above with the fact that any function $f\in  W^{1,p}\pa{\O}\subset W^{1,1}_{\mathrm{loc}}\pa{\O}$ satisfies that $\na f(x)=0$ for almost every $x$ in $f^{-1}\pa{\br{c}}=\br{x\;|\;f(x)=c}$, which can be found in Theorem 6.19 in \cite{LiebLoss}, we see that
\begin{equation}\nonumber
	\nabla \max\pa{f(x),c} 
	=\begin{cases}
		\na f(x)	& f(x)\geq c  \\  0 & f(x)< c
	\end{cases}.
\end{equation} 
and together with \eqref{eq:integration_by_parts_generalised} and \eqref{eq:gradient_of_max} we conclude \eqref{eq:integration_by_parts_of_max}.
\end{proof}
Next we turn our attention to the proof of Theorem \ref{thm:existence_result_degenerate}, which requires the following lemma:
	\begin{lemma}\label{lem:mild_strong}
		Assume that $\Omega\subset \R^n$, $n\geq 1$, is a bounded, open domain with $C^{2+\zeta}, \zeta>0$, boundary. Assume that $d,T>0$, $u_0\in \LO{\infty}$ and that the function $f$ belongs to $L^q{(0,T;\LO{q})}$ for all $q\in [2,\infty)$. Let
		\begin{equation}\label{eq:expression_for_u}
			u(x,t)= e^{d\Delta t}u_0(x)+\int_0^t e^{d\Delta(t-\xi)}f(x,\xi)d\xi,
		\end{equation}
		where $e^{d\Delta t}$ is the semigroup generated by the operator $d\Delta$ with homogeneous Neumann boundary conditions on $L^q(\Omega)$. Then $u$ is a strong solution to
		\begin{equation}\label{eq:heat}
			\begin{cases}
				\partial_t u(x,t)-d\Delta u(x,t) =f(x,t)& x\in\O,\; t\in(0,T)\\
				u(x,0)=u_0(x) & x\in\O\\
				\p_\nu u(x,t) = 0, &x\in\p\Omega, t\in(0,T).
			\end{cases} 
		\end{equation}
	\end{lemma}
	\begin{proof}
		Denoting by 
		$$u_2(x,t)=\int_0^t e^{d\Delta(t-\xi)}f(x,\xi)d\xi$$
		we see, according to \cite[Theorem 6.2.3]{PS16}, that for any $q\in [2,\infty)$ we have that {$$u_2\in L^q((0,T);W^{2,q}(\Omega))\cap W^{1,q}((0,T);\LO{q}).$$ }In particular, this implies that $u_2\in C^0([0,T];L^q(\Omega))$, and $u_2$ is a strong solution to \eqref{eq:heat} with $u_0\equiv 0$.\\
		To consider the general case we notice that we can't use the same considerations for
		$$u_1(x,t)=e^{d\Delta t}u_0(x)$$
		as $u_0$ does not necessarily belong to the right trace space. However, for every $t>0$, we can use the regularisation properties of the semigroup to conclude that {$u_1$ belongs to the space}  $C^0([0,T];\LO q)\cap C^1((0,T];\LO q)$ and {$u_1(\cdot,t)\in W^{2,q}(\Omega))$ for all positive $t$}. This entails that $u_1$ is continuous with respect to $\LO q$ and belongs to $L^q((\tau,T);W^{2,q}(\Omega)])\cap W^{1,q}((\tau,T);\LO{q})$ for every $\tau>0$ and $q\in [2,\infty)$. As a consequence, we obtain that $u_1$ is absolutely continuous as a $L^q(\Omega)$ valued function for positive times and is a strong solution to \eqref{eq:heat} with $f\equiv 0$. 	As $u=u_1+u_2$ we conclude the desired result.
	\end{proof}
\newcommand{\X}{\mathcal X}
\newcommand{\F}{\mathcal F}
\newcommand{\G}{\mathcal G}
\begin{proof}[Proof of Theorem \ref{thm:existence_result_degenerate}]
	The proof is based on a standard fixed point argument. For a fixed $T>0$ we denote by
	\begin{align*}
	\X &= \bigg\{\pa{e,s,c,p}\in \pa{\marcel{L^\infty}([0,T];\LO{\infty})}^4\;|\; (e(0),s(0),c(0),p(0)) = (e_0,s_0,c_0,p_0),\\
	&\qquad \text{ and } \quad \|(e,s,c,p)\|_{(\marcel{L^\infty}([0,T];\LO{\infty}))^4} \leq  \|(e_0,s_0,c_0,p_0)\|_{\LO{\infty}}+1:=M \bigg\},
	\end{align*}
	where 
	$$\norm{\pa{f_1,f_2,f_3,f_4}}_{(\marcel{L^\infty}([0,T];\LO{\infty}))^4} =\sup_{t\in [0,T]}\sum_{i=1}^{4}\norm{f_i}_{\LO{\infty}}.$$
	We define a map $\F$ 
		and $\G$ on $\X$ by
		\begin{align*}
		\F(e,s,c,p)(x,t) := &\left(e_0(x) + \int_0^t\pa{-k_fe(x,\xi)s(x,\xi) + (k_r+k_c)c(x,\xi)}d\xi, \right.\\
		&\quad e^{d_s\Delta t}s_0(x) + \int_0^t e^{ d_s\Delta (t-\xi)}\pa{-k_fe(x,\xi)s(x,\xi)+k_rc(x,\xi)}d\xi,\\
		&\quad c_0(x) + \int_0^t\pa{k_fe(x,\xi)s(x,\xi) - (k_r+k_c)c(x,\xi)}d\xi,\\
		&\quad \left.e^{ d_p\Delta t}p_0(x) + \int_0^te^{d_p\Delta (t-\xi)}k_cc(x,\xi)d\xi\right),
		\end{align*}
		and
		\begin{align*}
		\G(e,s,c,p)(x,t) :=  \ \F(e_+,s_+,c_+,p_+)(x,t),
		\end{align*}
		and where
	$$f_+:= \max\pa{f, 0}.$$ 
	We will now show that for $T$ small enough,  $\F$ is a contraction mapping from $\X$ into $\X$. This will show the existence of a unique bounded strong solution, at least on $(0,T)$. 
		Moreover, by showing that $\G$ is also a contraction mapping $\X$ into $\X$ admitting fixed point such that $e,s,c$ and $p$ are non-negative, we would be able conclude that this fixed point is in fact a fixed point for $\F$, and as such the strong solution we have found is in fact non-negative. 
	
	Clearly, $\F(e,s,c,p)(0)=\G(e,s,c,p)(0) = (e_0,s_0,c_0,p_0)$. Moreover, since
	\begin{equation}\label{eq:L_infty_bound_of_heat_flow}
	\|e^{d\Delta t}f\|_{\LO{\infty}} \leq \|f\|_{\LO{\infty}}
	\end{equation}
	and
	\begin{equation}\label{eq:L_infty_bound_fot_f_+}
	\norm{f_+}_{\LO{\infty} }\leq \norm{f}_{\LO{\infty}}
	\end{equation}
	we see that 
		\begin{align*}
		\|\F(e,s,c,p)\|_{(\marcel{L^\infty}([0,T];\LO{\infty}))^4} \leq \|(e_0,s_0,c_0,p_0)\|_{\LO{\infty}} + \mathcal{C}_0\pa{M^2+1}T
	\end{align*}
and
	\begin{align*}
	\|\G(e,s,c,p)\|_{(\marcel{L^\infty}([0,T];\LO{\infty}))^4} \leq \|(e_0,s_0,c_0,p_0)\|_{\LO{\infty}} + \mathcal{C}_0\pa{M^2+1}T
	\end{align*}
	for some constant $\mathcal{C}_0$ that depends is independent of $M$ and $T$. Choosing $T$ small enough so that
	$$\mathcal C_0(M^2+1)T\leq1$$
	we conclude that $\F$ and  $\G$ map $\X$ into $\X$ itself. 
	
	Next, using \eqref{eq:L_infty_bound_of_heat_flow} again, we find that
		\begin{align*}
		&\|\F(e_1,s_1,c_1,p_1) - \F(e_2,s_2,c_2,p_2)\|_{\pa{\marcel{L^\infty}([0,T];\LO{\infty})}^4}\\ &\leq \int_0^T\left\|-k_f\pa{e_{1}\pa{\xi}s_{1}\pa{\xi} - e_{2}\pa{\xi}s_{2}\pa{\xi}} + (k_r+k_c)\pa{c_{1}\pa{\xi} - c_{2}\pa{\xi}} \right\|_{\LO{\infty}}d\xi\\
		&\quad + \int_0^T\left\|-k_f\pa{e_{1}\pa{\xi}s_{1}\pa{\xi} - e_{2}\pa{\xi}s_{2}\pa{\xi}}+ k_r\pa{c_{1}\pa{\xi} -c_{2}\pa{\xi}} \right\|_{\LO{\infty}}d\xi\\
		&\quad + \int_0^T\left\|k_f\pa{e_{1}s_{1}\pa{\xi} - e_{2}\pa{\xi}s_{2}\pa{\xi}}- (k_r+k_c)\pa{c_{1}\pa{\xi} - c_{2}\pa{\xi}}\right\|_{\LO{\infty}}d\xi \\
		&\quad+ \int_0^T\left\|k_c\pa{c_{1}\pa{\xi} - c_{2}\pa{\xi}}\right\|_{\LO{\infty}}d\xi \\
		&\leq \mathcal{C}_1(M+1)T\left\|\pa{e_{1},s_{1},c_{1},p_{1})-(e_{2},s_{2},c_{2},p_{2}}\right\|_{(\marcel{L^\infty}([0,T];\LO{\infty}))^4}\\
		\intertext{and}
		&\|\G(e_1,s_1,c_1,p_1) - \G(e_2,s_2,c_2,p_2)\|_{\pa{\marcel{L^\infty}([0,T];\LO{\infty})}^4}\\
		=&\|\F(e_{1+},s_{1+},c_{1+},p_{1+}) - \marcel{\F}(e_{2+},s_{2+},c_{2+},p_{2+})\|_{\pa{\marcel{L^\infty}([0,T];\LO{\infty})}^4}\\
		&\leq \mathcal{C}_1(M+1)T\left\|(e_{1},s_{1},c_{1},p_{1})-(e_{2},s_{2},c_{2},p_{2})\right\|_{(\marcel{L^\infty}([0,T];\LO{\infty}))^4}
		\end{align*}
	for  a constant $\mathcal{C}_1$ that is independent of $M$ and $T$, where we have used the elementary inequality 
	$$|a_+ - b_+| \leq |a - b|.$$ 
	Restricting $T$ further so that
	$$\mathcal{C}_1(M+1)T<1$$ 
	we see that $\F$ and $\G$ are contraction mappings from $\X$ into $\X$, and therefore they each have a unique fixed point. \marcel{We denote} by $(\overline{e},\overline{s},\overline{c},\overline{p})$ the fixed point of $\G$. \marcel{According to Lemma \ref{lem:mild_strong} it is} a local strong solution to the system of equations\footnote{Note that by the definition of $\X$ and $\F$, the function $f$ from \eqref{eq:expression_for_u} is in $L^\infty([0,T];\LO{\infty})$, and as such the conditions of the Lemma are satisfied. }
	\begin{equation}\label{Sys_modified}
	\begin{cases}
	\p_t \overline{e}(x,t) = -k_f\overline{e}_{+}(x,t)\overline{s}_{+}(x,t) + (k_r+k_c)\overline{c}_{+}(x,t), & x\in \O,t>0,\\
	\p_t \overline{s}(x,t) - d_s\Delta \overline{s}(x,t) = -k_f\overline{e}_{+}(x,t)\overline{s}_{+}(x,t) + k_r\overline{c}_{+}(x,t),& x\in \O,t>0,\\
	\p_t \overline{c}(x,t) = k_f\overline{e}_{+}(x,t)\overline{s}_{+}(x,t) - (k_r+k_c)\overline{c}_{+}(x,t) ,& x\in \O,t>0,\\
	\p_t \overline{p}(x,t) - d_p\Delta \overline{p}(x,t) = k_c\overline{c}_{+}(x,t),& x\in \O,t>0,\\
	\p_\nu s(x,t) = \p_\nu p(x,t) = 0, &x\in\p\Omega, t>0,\\
	e(x,0) = e_0(x), \;s(x,0) = s_0(x),\; c(x,0) = c_0(x),\; p(x,0) = p_0(x), & x\in\Omega,
	\end{cases}
	\end{equation}
	on $(0,T)$. Let $T_{\mathrm{max}}$ be the maximal time of existence for the solution $(\overline{e},\overline{s},\overline{c},\overline{p})$. To show that $(\overline{e},\overline{s},\overline{c},\overline{p})$ is a global solution, i.e. $T_{\mathrm{max}}=+\infty$, it is enough to show that 
	\begin{equation}\label{eq:bound_for_global}
	\marcel{\norm{(\overline{e},\overline{s},\overline{c},\overline{p})}}_{(\marcel{L^\infty}([0,T];\LO{\infty}))^4}\leq C(T)
	\end{equation}
	for some continuous function $C(T):[0,\infty)\to [0,\infty)$ (see, for instance, \cite[Theorem 2.5.5]{zheng2004nonlinear}). The proof of the existence of such function is intertwined with the non-negativity property of $(\overline{e},\overline{s},\overline{c},\overline{p})$, which will also show that $(\overline{e},\overline{s},\overline{c},\overline{p})$ is in fact a solution to \eqref{Sys}. 
	
	Denoting by $f_{-}=\max\pa{0,-f}=\pa{-f}_+$ the so-called negative part of $f$, we notice that when $f$ is absolutely continuous with respect to $t$ so is $f_{-}$, and a.e. in $t$
	$$\frac{d}{dt}f^2_{-}(t) =2f_{-}(t)\frac{d}{dt}f_{-}(t)=-2f_{-}(t)\frac{d}{dt}f(t).$$
	As such, if $f(x,t)$ is a strong solution to 
	\begin{equation}\label{eq:non_negativity_strong_equations}
	\partial_t f(x,t) - d\Delta f(x,t)=g(t,x)-\alpha(x,t) f_{+}(x,t)		
	\end{equation}
	where $g(x,t)$ and $\alpha(x,t)$ are non-negative functions,  we see that by multiplying the above with $-f_{-}(x,t)\leq 0$, we find that
	$$\frac{1}{2}\partial_t f^2_{-}(x,t) - d\Delta f(x,t)f_{-}(x,t)=-g(t,x)f_{-}(x,t) \leq 0.$$
	Integrating over $\O\times (0,t)$ gives\footnote{The appearance of the norm of $\na f_{-}$, as well as its existence, is an immediate and simple variant of our discussion in Lemma \ref{lem:integration_by_parts_of_absolute} in Appendix \ref{secapp:additional_proofs}. Moreover, according to the definition of the strong solutions we find that $v$ \marcel{is absolutely continuous with respect to $L^2(\Omega))$}. As the $L^2(\Omega)$ scalar product is {bilinear} and continuous, we can apply the product rule to find that
		$$\frac{d}{dt}\norm{v(t)}_{\LO{2}}^2 = \frac{d}{dt}\int_{\O} v(x,t)v(x,t)dx=2\int_{\O}\p_t v(x,t)v(x,t)dx$$
	{for almost all $t>0$.}}
	$$\frac{1}{2}\norm{f_{-}(t)}^2_{\LO{2}}+d\int_{0}^{t}\norm{\na f_{-}(\xi)}^2_{\LO{2}}d\xi\leq \frac{1}{2}\norm{f_{-}(0)}^2_{\LO{2}},$$
	which implies that if $f(0)$ is non-negative, i.e. $f_{-}(0)\equiv0$ (a.e in $x$) then  $f_{-}(t)\equiv0$ (a.e in $x$) for any $t>0$ for which $f$ is a strong solution to  \eqref{eq:non_negativity_strong_equations}. As the equations for $\overline{e}$, $\overline{c}$, $\overline{s}$ and $\overline{p}$ are of the form \eqref{eq:non_negativity_strong_equations}, we conclude the non-negativity of $\overline{e}$, $\overline{c}$, $\overline{s}$ and $\overline{p}$, and as such the fact that it is in fact a solution to \eqref{Sys}. 
	
	Lastly, we shall show that \eqref{eq:bound_for_global} is valid for a constant function $C(T)$ which will show both the global existence and uniform boundedness, concluding the proof. \\
	Indeed, using the definition strong solution (Duhamel's formula, as is expressed in the definition of $\F$) and the non-negativity of $\overline{e}$ and $\overline{c}$ we see that 
	$$\max\pa{\abs{\overline{e}(x,t)},\abs{\overline{c}(x,t)}}=\max\pa{\overline{e}(x,t),\overline{c}(x,t)} \leq \overline{e}(x,t)+\overline{c}(x,t)= \norm{e_0}_{\LO{\infty}}+\norm{c_0}_{\LO{\infty}}.$$
	Thus, for any $T$ for which $\overline{e}$ and $\overline{c}$ are strong solutions to \eqref{Sys} we have that
	\begin{equation}\label{uni}
	\sup_{t\leq T}\max\pa{\norm{\overline{e}(t)}_{\LO{\infty}}, \norm{\overline{c}(t)}_{\LO{\infty}}}\leq \norm{e_0}_{\LO{\infty}}+\norm{c_0}_{\LO{\infty}}.
	\end{equation}
	It remains to show the boundedness of $\bs$ and $\bp$. Summing $\bs(t)$, $\bc(t)$, $\bp(t)$, and integrating over 
	 gives
	\begin{equation}\label{eq:L_1_bound_for_s_and_p}
		\|\bs (t)\|_{\LO{1}} + \|\bc(t)\|_{\LO{1}} + \|\bp(t)\|_{\LO{1}} = \|s_0\|_{\LO{1}} + \|c_0\|_{\LO{1}} + \|p_0\|_{\LO{1}} \quad \forall t\geq 0,
	\end{equation}
	where we have used the non-negativity of the solution again. For any $q\geq 2$ we have
	\begin{align*}
		\frac{d}{dt}\|\bs(t)\|_{\LO{q}}^q &= q\int_{\Omega}\bs^{q-1}(x,t)\p_t\bs(x,t) dx\\
		&\leq -q(q-1)d_s\int_{\O}\bs(x,t)^{q-2}|\na \bs(x,t)|^2dx + k_rq\int_{\O}\bs(x,t)^{q-1}\bc(x,t) dx\\
		&\leq -q(q-1)d_s\int_{\O}\bs(x,t)^{q-2}|\na \bs(x,t)|^2dx + k_rq\norm{\bs(t)}^{q-1}_{\LO{q}}\norm{\bc(t)}_{\LO{q}}\\
		&\leq -\frac{4(q-1)d_s}{q}\int_{\O}\left|\na\pa{\bs(x,t)^{\frac{q}{2}}}\right|^2dx + k_r\pa{(q-1)\|\bs(t)\|_{\LO{q}}^q + \|\bc(t)\|_{\LO{q}}^q}.
	\end{align*}
Using the uniform bound of $\bc$ in \eqref{uni} we find that
	\begin{equation}\label{b1}
		\frac{d}{dt}\|\bs(t)\|_{\LO{q}}^q + \frac{4(q-1)d_s}{q}\norm{\bs(t)^{\frac{q}{2}}}_{H^1(\Omega)}^2 \leq qC_{c_0,e_0}\pa{\|\bs\|_{\LO{q}}^q + 1}.
	\end{equation}
	Thanks to the continuous Sobolev embedding $H^1(\Omega)\hookrightarrow \LO{n^*}$ with
	\begin{equation*}
		n^*=\begin{cases}
			+\infty &\text{ if } n=1,\\
			<+\infty \text{ arbitrary } &\text{ if } n=2,\\
			\frac{2n}{n-2} &\text{ if } n\ge 3.
		\end{cases}
	\end{equation*}
	we find that
		\begin{equation}\label{b2}
		\norm{\bs^{\frac{q}{2}}}_{H^1(\Omega)}^2 \geq C_{n}\norm{\bs^{\frac{q}{2}}}_{\LO{n^\ast}}^{2}=C_{n}\norm{\bs}_{\LO{q_0}}^q,
	\end{equation}
	where $q_0=\frac{n^\ast q}{2}>q$. Using the interpolation inequality
	$$\norm{f}_{\LO{q}}^{q} \leq \norm{f}_{\LO{q_0}}^{\theta q}\norm{f}_{\LO{1}}^{\pa{1-\theta} q}$$
	where $\theta \in (0,1)$ satisfies $\frac{1}{q} = \frac{\theta}{q_0}+\frac{1-\theta}{1}$, together with \eqref{eq:L_1_bound_for_s_and_p} we see that any $\epsilon>0$ there exists a constant $C_\epsilon >0$ such that
	\begin{equation}\label{b3}
		\|\bs \|_{\LO{q}}^q \leq \epsilon\|\bs\|_{\LO{q_0}}^q + C_{\epsilon}.
	\end{equation}
	Combining \eqref{b1}, \eqref{b2} and \eqref{b3}, we see that for any $\eps>0$
	\begin{equation*}
		\frac{d}{dt}\|\bs (t)\|_{\LO{q}}^q + 2d_sC_n\|\bs(t)\|_{\LO{q}}^q \leq qC_{c_0,e_0}\pa{\epsilon\|\bs\|_{\LO{q_0}}^q + C_{\epsilon,}},
	\end{equation*}
	from which we get with a choice of $\eps$ small enough\footnote{$\eps=\frac{d_sC_n}{qC_{c_0,e_0}}$ will do}
	that
	\begin{equation}\label{eq:uniform_L_q_bound_for_s}
		\sup_{t\ge 0}\|\bs(t)\|_{\LO{q}} < +\infty
	\end{equation}
	for any $2\leq q < +\infty$. Applying a simple variant of Lemma \ref{lem:heat_equation_regularisation} with $\delta=0$, where the $L^\infty$ bound is on $\rpa{0,\frac{T}{2}}$ instead of $[0,1]$, shows that\footnote{The appropriate $f$ in this case is $f=-k_f e s +k_r c$ whose $L^p$ bound follows from the uniform $L^\infty$ bound on $e$ and $c$ and \eqref{eq:uniform_L_q_bound_for_s}.}
	\begin{equation*}
		\sup_{t\geq 0}\|\bs(t)\|_{\LO{\infty}} < +\infty.
	\end{equation*}
	The uniform bound of $\bp$ follows in the exact same way.
\end{proof}

We conclude the appendix with the proofs of lemmas \ref{lem:heat_equation_regularisation} and \ref{lem:heat_equation_avergae_convergence}.
\begin{proof}[Proof of Lemma \ref{lem:heat_equation_regularisation}]
	A known estimate on the kernel of the heat equation (see for instance \cite[Theorem 3.2.9, page 90]{Davies89}) implies that
	the solution to the heat equation with homogeneous Neumann boundary condition and initial datum  in $L^p\pa{\O}$ satisfies
	\begin{equation}\label{semigroup}
		\norm{u(t)}_{L^\infty(\Omega)} = \norm{e^{d\Delta t}u_0}_{\LO{\infty}}\leq C_{d}t^{-\frac{n}{2p}}\norm{u_0}_{L^p(\Omega)} \quad  \forall \textcolor{red}{0<t\leq 1},
	\end{equation}
	for some fixed known constant that depends on $\O$ and the diffusion coefficient $d$\footnote{One can easily extent \eqref{semigroup} to $0<t<T$ for any given $T>0$. The constant $C_{d}$ in that case will become dependent on $T$ (and one should denote it by $C_{d,T}$).}. As such, a solution to the equation
	$$\begin{cases}
		\partial_t u(x,t)-d\Delta u(x,t) =f(x,t)& x\in\O,\; t>0\\
		u(x,0)=0 & x\in\O\\
		\p_\nu u(x,t) = 0, &x\in\p\Omega, t>0,
	\end{cases} $$
	would satisfy, according to the Duhamel formula, 
	\begin{equation}\nonumber
		u(x,t+1) = e^{d\Delta }u(x,t)+\int_{0}^{1} e^{-d\Delta s}f(x,t+s)ds,\quad t>0
	\end{equation}
	which, together with \eqref{semigroup}, implies that
	$$\norm{u(t+1)}_{\LO{\infty}} \leq C_d\pa{\norm{u(t)}_{\LO{p}}+\int_{0}^{1}\pa{1+s^{-\frac{n}{2p}}}\norm{f(t+s)}_{\LO{p}}ds}.$$
	Using conditions \eqref{eq:heat_conditions} we find that 
	$$\norm{u(t+1)}_{\LO{\infty}} \leq \mathcal{C}C_d\pa{1+\int_{0}^{1}\pa{1+s^{-\frac{n}{2p}}}e^{-\delta s}ds}e^{-\delta t}.$$
	If $p>\frac{n}{2}$ then 
	$$\int_{0}^{1}s^{-\frac{n}{2p}} e^{-\delta s}ds\leq \int_{0}^{1}s^{-\frac{n}{2p}}ds=C_{n,p}<\infty,$$ 
	and we conclude that for all $t\geq 1$
	$$\norm{u(t)}_{\LO{\infty}}\leq \mathcal{C}C_d \pa{2+C_{n,p}}e^{-\delta (t-1)}.$$
	As for any $t\in[0,1]$ we have that
	$$\sup_{t\in [0,1]}\norm{u(t)}_{\LO{\infty}}\leq \su \leq \su e^{\delta}e^{-\delta t}$$
	due to \eqref{eq:heat_conditions}, we find that 
	$$\norm{u(t)}_{\LO{\infty}} \leq e^{\delta}\max\pa{\su , \mathcal{C}C_d\pa{2+C_{n,p}}}e^{-\delta t},$$
	which is the desired result.
\end{proof}

\begin{proof}[Proof of Lemma \ref{lem:heat_equation_avergae_convergence}]
	Defining the function $v(x,t)=u(x,t)-\overline{u}(t)$ we find that $v$ solves the equation 
		\begin{equation}\nonumber
		\begin{cases}
			\partial_t v(x,t)-d\Delta v(x,t) =f(x,t)-\partial_t \overline{u}(t)& x\in\O,\; t>0\\
			u(x,0)=u_0(x)-\overline{u_0} & x\in\O\\
			\p_\nu u(x,t) = 0, &x\in\p\Omega. t>0.
		\end{cases} 
	\end{equation}
Denoting by $\widetilde{f}(x,t)=f(x,t)-\partial_t \overline{u}(t)$, multiplying the first line of the equation by $v(x,t)$ and integrating over the domain yields the equality
$$\partial_t \norm{v(t)}_{\LO{2}}^2=-2d\norm{\na v(x,t)}_{\LO{2}}^2+ 2\int_{\O}\widetilde{f}(x,t)v(x,t)dx.$$
Thus, using the Poincar\'e inequality \eqref{eq:poincare} and the fact that $\overline{v}(t)=0$ we find that
$$\partial_t \norm{v(t)}_{\LO{2}}^2 \leq -\frac{2d}{C_P}\norm{v(t)}_{\LO{2}}^2+2\norm{\widetilde{f}(t)}_{\LO{2}}\norm{v(t)}_{\LO{2}} $$
$$\leq -\frac{2d}{C_P}\pa{1-\eps}\norm{v(t)}_{\LO{2}}^2 + \frac{C_{P}\norm{\widetilde{f}(t)}_{\LO{2}}^2}{2d\eps},$$
where we have used the fact that for any $\delta>0$ and $a,b\geq 0$
$$2ab \leq \delta a^2 + \frac{b^2}{\delta}$$
and chose $\delta = \frac{2d\eps}{C_P}$. The result follows from a simple integration and the fact that by integrating our heat equation over $\O$ we find that
$$\partial_t \overline{u}(t) = \int_{\O}f(x,t)dx=\overline{f}(t).$$
\end{proof}


\begin{thebibliography}{00}
	
	\bibitem[AMTU01]{arnold2001convex}
	Anton Arnold, Peter Markowich, Giuseppe Toscani, and Andreas Unterreiter.
	\newblock On convex sobolev inequalities and the rate of convergence to
	equilibrium for fokker-planck type equations.
	\newblock {\em Communications in Partial Differential Equations},
	26:43--100, 2001.
	
	\bibitem[BRZ20]{BRZ2020}
	Marcel Braukhoff, Claudia Raithel, and Nicola Zamponi.
	\newblock Partial h\"older regularity for bounded solutions of a class of
	cross-diffusion systems with entropy structure.
	\newblock {\em Preprint}, 2020.
	
	\bibitem[Dav89]{Davies89}
	E.~B. Davies.
	\newblock {\em Heat Kernels and Spectral Theory}.
	\newblock Cambridge Tracts in Mathematics. Cambridge University Press, 1989.
	
	\bibitem[DF06]{desvillettes2006exponential}
	Laurent Desvillettes and Klemens Fellner.
	\newblock Exponential decay toward equilibrium via entropy methods for
	reaction--diffusion equations.
	\newblock {\em Journal of mathematical analysis and applications},
	319(1):157--176, 2006.
	
	\bibitem[DF08]{desvillettes2008entropy}
	Laurent Desvillettes and Klemens Fellner.
	\newblock Entropy methods for reaction-diffusion equations: slowly growing
	a-priori bounds.
	\newblock {\em Revista Matem{\'a}tica Iberoamericana}, 24(2):407--431, 2008.
	
	\bibitem[DF13]{desvillettes2013exponential}
	Laurent Desvillettes and Klemens Fellner.
	\newblock Exponential convergence to equilibrium for nonlinear
	reaction-diffusion systems arising in reversible chemistry.
	\newblock In {\em IFIP Conference on System Modeling and Optimization}, pages
	96--104. Springer, 2013.
	
	\bibitem[DFT17]{desvillettes2017trend}
	Laurent Desvillettes, Klemens Fellner, and Bao~Quoc Tang.
	\newblock Trend to equilibrium for reaction-diffusion systems arising from
	complex balanced chemical reaction networks.
	\newblock {\em SIAM Journal on Mathematical Analysis}, 49(4):2666--2709, 2017.
	
	\bibitem[Eva10]{evans2010partial}
	Lawrence~C Evans.
	\newblock Partial differential equations.
	\newblock {\em Graduate Studies in Mathematics}, 19:333--339, 2010.
	
	\bibitem[FT18]{fellner2018convergence}
	Klemens Fellner and Bao~Quoc Tang.
	\newblock Convergence to equilibrium of renormalised solutions to nonlinear
	chemical reaction--diffusion systems.
	\newblock {\em Zeitschrift f{\"u}r angewandte Mathematik und Physik},
	69(3):1--30, 2018.
	
	\bibitem[GGH96]{glitzky1996free}
	Annegret Glitzky, Konrad Groger, and R~Hiinlich.
	\newblock Free energy and dissipation rate for reaction diffusion processes of
	electrically charged species.
	\newblock {\em Applicable Analysis}, 60(3-4):201--217, 1996.
	
	\bibitem[Gr{\"o}83]{groger1983asymptotic}
	Konrad Gr{\"o}ger.
	\newblock Asymptotic behavior of solutions to a class of diffusion-reaction
	equations.
	\newblock {\em Mathematische Nachrichten}, 112(1):19--33, 1983.
	
	\bibitem[Gr{\"o}92]{groger1992free}
	Konrad Gr{\"o}ger.
	\newblock {\em Free energy estimates and asymptotic behaviour of reaction
		diffusion processes}.
	\newblock Inst. f{\"u}r Angewandte Analysis und Stochastik, 1992.
	
	\bibitem[GV1]{GV10}
	
	
	\bibitem[JR11]{jabin2011selection}
	Pierre-Emmanuel Jabin and Ga{\"e}l Raoul.
	\newblock On selection dynamics for competitive interactions.
	\newblock {\em Journal of Mathematical Biology}, 63(3):493--517, 2011.
	
	\bibitem[LL01]{LiebLoss}
	Elliott~H. Lieb and Michael Loss.
	\newblock {\em Analysis}, volume~14 of {\em Graduate Studies in Mathematics}.
	\newblock American Mathematical Society, Providence, RI, second edition, 2001.
	
	\bibitem[MHM15]{mielke2015uniform}
	Alexander Mielke, Jan Haskovec, and Peter~A Markowich.
	\newblock On uniform decay of the entropy for reaction--diffusion systems.
	\newblock {\em Journal of Dynamics and Differential Equations},
	27(3-4):897--928, 2015.
	
	\bibitem[Mie16]{mielke2016uniform}
	Alexander Mielke.
	\newblock Uniform exponential decay for reaction-diffusion systems with
	complex-balanced mass-action kinetics.
	\newblock In {\em International Conference on Patterns of Dynamics}, pages
	149--171. Springer, 2016.
	
	\bibitem[MT20]{morgan2020boundedness}
	Jeff Morgan and Bao~Quoc Tang.
	\newblock Boundedness for reaction--diffusion systems with lyapunov functions
	and intermediate sum conditions.
	\newblock {\em Nonlinearity}, 33(7):3105, 2020.
	
	\bibitem[PS16]{PS16}
	Jan Pruess and Gieri Simonett.
	\newblock {\em Moving Interfaces and Quasilinear Parabolic Evolution Equations
		-}.
	\newblock Birkhaeuser, Basel, 2016.
	
	\bibitem[Tay96]{taylor1996partial}
	Michael~Eugene Taylor.
	\newblock {\em Partial differential equations. 1, Basic theory}.
	\newblock Springer, 1996.
	
	\bibitem[Zhe04]{zheng2004nonlinear}
	Songmu Zheng.
	\newblock {\em Nonlinear evolution equations}.
	\newblock CRC Press, 2004.
	
\end{thebibliography}
\end{document}